\documentclass[twoside,12pt]{article}
\usepackage{amsmath,amssymb,amsthm,amsfonts}
\usepackage{paralist}
\usepackage[title]{appendix}
\usepackage{graphics} 
\usepackage{epsfig} 
\usepackage{graphicx}
\usepackage{caption}
\usepackage{subcaption}
\usepackage{mathrsfs}
\usepackage{epstopdf}
\usepackage[colorlinks=true]{hyperref}
\usepackage{txfonts}
\usepackage{enumerate}
 \usepackage{verbatim}

\usepackage[numbers,sort&compress]{natbib}
\usepackage[font=scriptsize,labelfont=bf,width=0.85\textwidth]{caption}

\numberwithin{equation}{section}

\newcommand{\be}{\begin{equation}}
\newcommand{\ee}{\end{equation}}

\newcommand{\ben}{\begin{eqnarray*}}
\newcommand{\enn}{\end{eqnarray*}}

\newtheorem{proposition}{Proposition}[section]
\newtheorem{theorem}{\textbf Theorem}[section]
\newtheorem{lemma}{\textbf Lemma}[section]

 \numberwithin{equation}{section}
\newtheorem{remark}{Remark}[section]
\renewcommand{\theequation}{\arabic{section}.\arabic{equation}}

\begin{document}

\title{Existence and Blow-up Profiles of Ground States in Second Order Multi-population Mean-field Games Systems}
\author{
Fanze Kong \thanks{Department of Applied Mathematics, University of Washington, Seattle, WA 98195, USA; fzkong@uw.edu},
Juncheng Wei \thanks{Department of Mathematics, Chinese University of Hong Kong,
 Shatin, NT, Hong Kong; wei@math.cuhk.edu.hk}
and 
Xiaoyu Zeng \thanks{Department of Mathematics, Wuhan University of Technology, Wuhan 430070, China; xyzeng@whut.edu.cn.}
        }

\date{\today}
\maketitle
\begin{abstract}

In this paper, we utilize the variational structure to study the existence and asymptotic profiles of ground states in multi-population ergodic Mean-field Games systems subject to some local couplings with mass critical exponents.  Of concern the attractive and repulsive interactions, we impose some mild conditions on trapping potentials and  firstly classify the existence of ground states in terms of intra-population and interaction coefficients.  Next, as the intra-population and inter-population coefficients approach some critical values, we show the ground states blow up at one of global minima of potential functions and the corresponding profiles are captured by ground states to potential-free Mean-field Games systems for single population up to translations and rescalings.  Moreover, under certain types of potential functions, we establish the refined blow-up profiles of corresponding ground states. In particular, we show that the ground states concentrate at the flattest global minima of potentials.

\end{abstract}


\section{Introduction}
Mean-field Games systems are proposed to describe decision-making among a huge number of indistinguishable rational agents.  In real world, various problems involve numerous interacting players, which causes theoretical analysis and even numerical study become impractical.  To overcome this issue, Huang et al. \cite{Huang} and Lasry et al. \cite{Lasry} borrowed the ideas arising from particle physics and introduced Mean-field Games theories and systems independently.  For their rich applications in economics, finance, management, etc, we refer the readers to \cite{Gue091}.

Focusing on the derivation of Mean-field Games systems, we assume that the $i$-th agent with $i=1,\cdots,n$ satisfies the following controlled stochastic differential equation (SDE):
\begin{align*}
dX_t^i=-\gamma^i_t dt+\sqrt{2}dB_t^i, \ \ X_0^i=x^i\in\mathbb R^N,
\end{align*}
where $x^i$ is the initial state, $\gamma^i_t$ denotes the controlled velocity and $B_t^i$ represent the independent Brownian motion. Suppose all agents are indistinguishable and minimize the following average cost:   
\begin{align}\label{longsenseexpectation}
J(\gamma_t):=\mathbb E\int_0^T[L(\gamma_t)+V(X_t)+f(m(X_t))] dt + u_T(X_T),
\end{align}
where $L$ is the Lagrangian, $V$ describes the spatial preference and function $f$ depends on the population density.  By applying the standard dynamic programming principle, the coupled PDE system consisting of Hamilton–Jacobi–Bellman equation and Fokker-Planck equation is formulated, in which the second equation characterize the distribution of the population.  The crucial assumption here is all agents are homogeneous and minimize the same cost (\ref{longsenseexpectation}).  Whereas, in some scenarios, the game processes involve several classes of players with distinct objectives and constraints.  Correspondingly, the distributions of games can not be modelled by classical Mean-field Games systems.  Motivated by this, multi-population Mean-field Games systems were proposed and the derivations of multi-population stationary problems used to describe Nash equilibria are shown in \cite{feleqi2013derivation}.  For some relevant results of the study of multi-population Mean-field Games systems, we refer the readers to \cite{huang2010large,nourian2013eps,cirant2015multi,MR4083905,MR3791252}.  We also mention that some general theories for the study of Mean-field Games systems can be found in \cite{bensoussan2013mean,cardaliaguet2010notes,MR2762362,MR4214773}.  Recently, with the consideration of a common noise, some researchers established the extended mean-field games systems and discussed some properties such as well-posedness \cite{lions2020extended,cardaliaguet2022first}. 

The objective of this paper is to study the following stationary two-population second order Mean-field Games system:
\begin{align}\label{ss1}
\left\{\begin{array}{ll}
-\Delta u_1+H(\nabla u_1)+\lambda_1=V_1(x)+f_1(m_1,m_2),&x\in\mathbb R^N,\\
\Delta m_1+\nabla\cdot(m_1\nabla H(\nabla u_1))=0,&x\in\mathbb R^N,\\
-\Delta u_2+H(\nabla u_2)+\lambda_2=V_2(x)+f_2(m_1,m_2),&x\in\mathbb R^N,\\
\Delta m_2+\nabla\cdot(m_2\nabla H(\nabla u_2))=0,&x\in\mathbb R^N,\\
\int_{\mathbb R^N}m_1\,dx=\int_{\mathbb R^N}m_2\,dx=1,
\end{array}
\right.
\end{align}
where $H:\mathbb R^N\rightarrow \mathbb R$ is a Hamiltonian, $(m_1,m_2)$ represents the population density, $(u_1,u_2)$ denotes the value function and $(f_1,f_2)$ is the coupling.  Here $V_i(x)$, $i=1,2$ are potential functions and $(\lambda_1,\lambda_2)\in\mathbb R\times \mathbb R$ denotes the Lagrange multiplier.  In particular, Hamiltonian $H$ is in general chosen as 
\begin{align}\label{MFG-H}
H(p)=C_H|p|^{\gamma}\text{ with }C_H>0\text{ and }\gamma>1.
\end{align}
In light of the definition, the corresponding Lagrangian is given by 
\begin{align*}
L={C_L}\vert\gamma\vert^{\gamma'},~~\gamma'=\frac{\gamma}{\gamma-1}>1,~C_L=\frac{1}{\gamma'}(\gamma C_H)^{\frac{1}{1-\gamma}}>0.
\end{align*}

 From the viewpoint of variational methods, the single population counterpart of (\ref{ss1}) has been studied intensively when the coupling $f$ is local and satisfies $f=-em^\alpha$ with constant $e>0$, see \cite{cesaroni2018concentration,cirant2023ergodic,cirant2024critical}.  In detail, there exists a mass critical exponent $\alpha=\alpha^*:=\frac{\gamma'}{N}$ such that only when $\alpha<\alpha^*$, the stationary problem admits ground states for any $e>0$.  Moreover, when $\alpha=\alpha^*,$ one can find $e^*>0$ such that the stationary Mean-field Games system has ground states only for $e<e^*$ \cite{cirant2024critical}.  In this paper, we shall extend the above results into two-species stationary Mean-field Games system (\ref{ss1}).  Similarly as in \cite {cirant2024critical}, we consider the mass critical exponent case and define 
 \begin{align}\label{alpha12betadef}
 f_1=-\alpha_1 m_1^{\frac{\gamma'}{N}}-\beta m_1^{\frac{\gamma'}{2N}-\frac{1}{2}}m_2^{\frac{1}{2}+\frac{\gamma'}{N}},~~f_2=-\alpha_2 m_2^{\frac{\gamma'}{N}}-\beta m_2^{\frac{\gamma'}{2N}-\frac{1}{2}}m_1^{\frac{1}{2}+\frac{\gamma'}{N}},
 \end{align}
 where $\alpha_i>0$, $i=1,2$ and $\beta$ measure the strengths of intra-population and inter-population interactions, respectively.  We shall employ the variational approach to classify the existence of ground states and analyze their asymptotic profiles to (\ref{ss1}) in terms of $\alpha_i$, $i=1,2$ and $\beta.$  Noting the forms of nonlinearities shown in (\ref{alpha12betadef}), we assume $\gamma'>N$ here and in the sequel for our analysis; otherwise the strong singularities might cause difficulties for finding ground states to (\ref{ss1}) while taking limits. It is an intriguing but challenging problem to explore the existence of global minimizers in the case of $1<\gamma'\leq N.$


By employing the variational methods, the existence of ground states to (\ref{ss1}) is associated with the following constrained minimization problem:
\begin{align}\label{problem1p1}
e_{\alpha_1,\alpha_2,\beta}=\inf_{(m_1,w_1,m_2,w_2)\in \mathcal K}\mathcal E(m_1,w_1,m_2,w_2),
\end{align}
where
\begin{align}\label{energy1p3}
    \mathcal E_{\alpha_1,\alpha_2,\beta}(m_1,w_1,m_2,w_2):=&\sum_{i=1,2}\bigg(C_L\int_{\mathbb R^N}\bigg| \frac{w_i}{m_i}\bigg|^{\gamma'}m_i\, dx+\int_{\mathbb R^N}V_im_i\,dx-\frac{N}{N+\gamma'}\alpha_i\int_{\mathbb R^N}m_i^{1+\frac{\gamma'}{N}}\,dx\bigg)\nonumber\\
    &-\frac{2\beta N}{N+\gamma'}\int_{\mathbb R^N}m_1^{\frac{1}{2}+\frac{\gamma'}{2N}}m_2^{\frac{1}{2}+\frac{\gamma'}{2N}}\,dx,
\end{align}
and $\mathcal K=\mathcal K_1\times \mathcal K_2$
with
\begin{align}\label{mathcalkidefined}
\mathcal K_i=\bigg\{&(m_i,w_i)\bigg|-\int_{\mathbb R^N}\nabla m_i\cdot\nabla \varphi\,dx+\int_{\mathbb R^N}w_i\cdot\nabla\varphi\,dx=0,~\forall\varphi\in C_c^{\infty}(\mathbb R^N),\nonumber\\
&m_i\in W^{1,\gamma'}(\mathbb R^N),~w_i\in L^1(\mathbb R^N),
~\int_{\mathbb R^N}m_i\,dx=1,~\int_{\mathbb R^N}V_im_i\,dx<+\infty,~m_i\geq 0\text{~a.e.~}\bigg\}
\end{align}
for $i=1,2$.  Due to the technical restriction of our analysis, we impose the following assumptions on potential functions $V_i(x)$ with $i=1,2$:
\begin{itemize}
    \item[(H1).] 
    \begin{align}\label{Vicondition1}
\inf_{x\in\mathbb R^N} V_i(x)=0,~V_i\in C^1(\mathbb R^N)\text{ and }\lim_{|x|\rightarrow +\infty}V_i(x)=+\infty;
\end{align}
    \item[(H2).] 
    \begin{align}\label{Vicondition2}
\liminf_{|x|\rightarrow+\infty}\frac{V_i(x)}{|x|^b}>0,~~\limsup_{|x|\rightarrow +\infty}\frac{V_i(x)}{e^{\delta |x|}}<+\infty\text{ with constants }b>0,~\delta>0.
\end{align}

\end{itemize}
Similarly as shown in \cite{cirant2024critical}, the existence of ground states to (\ref{ss1}) has a strong connection with the following minimization problem for the single species potential-free Mean-field Games System:
\begin{align}\label{GNinequalitybest}
    (M^*)^{\frac{\gamma'}{N}}=\inf_{(m,w)\in\mathcal A}\frac{\bigg(\int_{\mathbb R^N}C_L\big|\frac{w}{m}\big|^{\gamma'}m\,dx\bigg) \bigg(\int_{\mathbb R^N}m\,dx\bigg)^{\frac{\gamma'}{N}}}{\frac{1}{1+\frac{\gamma'}{N}}\int_{\mathbb R^N}m^{1+\frac{\gamma'}{N}}\,dx},
\end{align}
where
\begin{align*}
\mathcal A:=\bigg\{&(m,w)\in W^{1,\gamma'}(\mathbb R^N)\cap L^1(\mathbb R^N)\bigg|-\int_{\mathbb R^N}\nabla m\cdot\nabla\varphi\,dx+\int_{\mathbb R^N}w\cdot\nabla\varphi\,dx=0,~\forall \varphi\in C_c^\infty(\mathbb R^N),
\\
&0\leq m\not\equiv 0, ~\int_{\mathbb R^N}m|x|^b\,dx<+\infty\text{~with~}b>0\text{~given by \eqref{Vicondition2}~}\bigg\}.
\end{align*}
We would like to point out that it was shown in Theorem 1.2 \cite{cirant2024critical} that problem (\ref{GNinequalitybest}) is attainable and admits at least a minimizer satisfying
\begin{align}\label{equmpotentialfree}
\left\{\begin{array}{ll}
-\Delta u+C_H|\nabla u|^{\gamma}-\frac{\gamma'}{NM^*}=-m^{\frac{\gamma'}{N}},\\
\Delta m+C_H\gamma\nabla\cdot(m|\nabla u|^{\gamma-2}\nabla u)=0,~w=-C_H\gamma m|\nabla u|^{\gamma-2}\nabla u,\\
\int_{\mathbb R^N}m\,dx=M^*,~0<m<Ce^{-\delta_0 |x|},
\end{array}
\right.
\end{align}
where $\delta_0>0$ is some constant.  As a consequence, the following Gagliardo-Nirenberg type's inequality holds:  
\begin{align}\label{GNinequalityused}
\frac{N}{N+\gamma'}\int_{\mathbb R^N}m^{1+\frac{\gamma'}{N}}\,dx \leq \frac{1}{a^*}\bigg(C_L\int_{\mathbb R^N}\bigg|\frac{w}{m}\bigg|^{\gamma'}m\,dx\bigg)\bigg(\int_{\mathbb R^N}m\,dx\bigg)^{\frac{\gamma'}{N}}\  \forall\ (m,w)\in \mathcal A,
\end{align}
where $a^*:=(M^*)^{\frac{\gamma'}{N}}$.
With the aid of \eqref{GNinequalityused}, we shall establish several results for the existence and non-existence of global minimizers to (\ref{ss1}) and further study the blow-up behaviors of ground states in terms of $\alpha_i$, $i=1,2$ and $\beta$ defined in \eqref{alpha12betadef}.  We emphasize that $\alpha_i>0$, $i=1,2$ represent the self-focusing of the $i$-th component and $\beta>0$ denotes the attractive interaction, while $\beta<0$ represents the repulsive interaction.  

In the next subsection, we shall first state our existence results for attractive and repulsive interactions then discuss the corresponding blow-up profiles results.

\subsection{Main Results}

\begin{theorem}\label{thm11multi}
Suppose that $V_i(x)$ with $i=1,2$ satisfy (H1) and (H2) given by (\ref{Vicondition1}) and (\ref{Vicondition2}), respectively.  Define $a^*:=(M^*)^{\frac{\gamma'}{N}}$ with $M^*$ given in (\ref{equmpotentialfree}), then we have
\begin{itemize}
    \item[(i).] if $0<\alpha_1,\alpha_2<a^*$ and $-\infty<\beta<\beta_*:=\sqrt{(a^*-\alpha_1)(a^*-\alpha_2)}$, problem (\ref{problem1p1}) has at least one global minimizer $(m_{1,\textbf{a}},w_{1,\textbf{a}},m_{2,\textbf{a}},w_{2,\textbf{a}})\in \mathcal K$.  Correspondingly, there exists a solution $(m_{1,\textbf{a}},m_{2,\textbf{a}},u_{1,\textbf{a}},u_{2,\textbf{a}})\in W^{1,p}(\mathbb R^N)\times W^{1,p}(\mathbb R^N)\times C^2(\mathbb R^N)\times C^2(\mathbb R^N)$ with any $p>1$ and $(\lambda_{1,\textbf{a}},\lambda_{2,\textbf{a}})\in\mathbb R\times\mathbb R$ such that
    \begin{align}\label{ss1thm11}
\left\{\begin{array}{ll}
-\Delta u_1+C_H|\nabla u_1|^{\gamma}+\lambda_1=V_1(x)-\alpha_1 m_1^{\frac{\gamma'}{N}}-\beta m_1^{\frac{\gamma'}{2N}-\frac{1}{2}}m_2^{\frac{1}{2}+\frac{\gamma'}{N}},&x\in\mathbb R^N,\\
\Delta m_1+C_H\gamma\nabla\cdot(m_1|\nabla u_1|^{\gamma-2}\nabla u_1)=0,&x\in\mathbb R^N,\\
-\Delta u_2+C_H|\nabla u_2|^{\gamma}+\lambda_2=V_2(x)-\alpha_2 m_2^{\frac{\gamma'}{N}}-\beta m_2^{\frac{\gamma'}{2N}-\frac{1}{2}}m_1^{\frac{1}{2}+\frac{\gamma'}{N}},&x\in\mathbb R^N,\\
\Delta m_2+C_H\gamma\nabla\cdot(m_2|\nabla u_2|^{\gamma-2}\nabla u_2)=0,&x\in\mathbb R^N,\\
\int_{\mathbb R^N}m_1\,dx=\int_{\mathbb R^N}m_2\,dx=1;
\end{array}
\right.
\end{align}
\item[(ii). ] either $\alpha_1>a^*$ or $\alpha_2>a^*$ or $\beta>\beta^*:=\frac{2a^*-\alpha_1-\alpha_2}{2},$ problem (\ref{problem1p1}) has no minimizer.
\end{itemize}
\end{theorem}
Theorem \ref{thm11multi} indicates that when the self-focusing cofficients $\alpha_i$, $i=1,2$ are small and the interaction is repulsive, or attractive but with the weak effect, problem (\ref{problem1p1}) admits minimizers and correspondingly, there exist classical solutions to (\ref{ss1thm11}).  Whereas, if the self-focusing effects and the attractive interaction are strong, problem (\ref{problem1p1}) does not have any minimizer.  In fact, there are some gap regions for the existence results shown in Theorem \ref{thm11multi} since we have $\beta^*\geq \beta_*$ and the equality holds only when $\alpha_1=\alpha_2$.  It is also an interesting problem to explore the case of $\alpha_i<a^*$, $i=1,2$ and $\beta_*<\beta<\beta^*$.

Of concern one borderline case $\beta=\beta_*=\beta^*$ with $\alpha_1=\alpha_2<a^*$ shown in Theorem \ref{thm11multi}, we further obtain
\begin{theorem}\label{thm12}
Assume all conditions in Theorem \ref{thm11multi} hold and suppose $V_i(x)$, $i=1,2$ satisfy
\begin{align}\label{moreassumptionforthm12}
\inf_{x\in\mathbb R^N}(V_1(x)+V_2(x))=0.
\end{align}
Then if $\alpha:=\alpha_1=\alpha_2<a^*$ and $0<\beta=\beta^*=\beta_*=a^*-\alpha<a^*,$ we have problem (\ref{problem1p1}) has no minimizer.
\end{theorem}
Theorem \ref{thm12} demonstrates that when the self-focusing effects are subcritical but the attractive interaction is strong and under critical case, there is no minimizer to problem (\ref{problem1p1}).  Besides the borderline case discussed in Theorem \ref{thm12}, we also study the case of $\alpha_i=a^*$ for $i=1$ or $2$ and obtain
\begin{theorem}\label{thmfinalexistence}
Assume all conditions in Theorem \ref{thm11multi} hold.  If one of the following conditions holds: 
\begin{itemize}
    \item[(i).]$\alpha_1=\alpha_2=a^*$ and $-\infty<\beta\leq 0;$
    \item[(ii).] $\alpha_1=a^*$, $0<\alpha_2<a^*$ and $0\leq \beta\leq \beta^*=\frac{a^*-\alpha_2}{2},$
 
\end{itemize}
then we have problem (\ref{problem1p1}) does not admit any minimizer.
\end{theorem} 
\begin{remark}
We remark that when $\alpha_2=a^*$, $0<\alpha_1<a^*$ and $0\leq \beta\leq \beta^*$, (\ref{problem1p1}) also does not have any minimizer since $m_1$-population and $m_2$-population are symmetric in (\ref{ss1}),
\end{remark}
Theorem \ref{thmfinalexistence} shows that if one of self-focusing coefficients are critical, system (\ref{ss1}) does not admit the ground state.  We next summarize results for the study of blow-up profiles of ground states in some singular limits, in which two cases are concerned: attractive interactions with $\beta>0$ and repulsive ones with $\beta<0$.  Before stating our results, we give some preliminary notations. Define
\begin{align}\label{zerosdefinition20241005}
Z_i:=\{x|V_i(x)=0\}, i=1,2.
\end{align}
For any $p>0$,  we denote
\begin{align}\label{Hmoibarnupi}
H_{\bar m,p}(y):=\int_{\mathbb R^N}|x+y|^{p}\bar m(x)\,dx,\text{ and }\bar{\nu}_{p}:=\inf_{(\bar m,\bar w)\in\mathcal M}\inf_{y\in\mathbb R^N}H_{\bar m,p}(y),
\end{align}
with
\begin{align}\label{mathcalMdefinedbythm17}
\mathcal M:=\{(\bar m,\bar w)|\exists u\text{ such that }(\bar m,\bar w,u)\text{ satisfies }(\ref{equmpotentialfree})\text{ and }(\bar m,\bar w)\text{ is a minimizer of }(\ref{GNinequalitybest})\}.
\end{align}

The following two theorems address the attractive case with $(\alpha_1,\alpha_2)\nearrow (a^*-\beta,a^*-\beta)$ and $Z_1\cap Z_2\not=\emptyset$, which are
\begin{theorem}\label{thm13attractive}
Assume that $V_i(x)$ satisfies (\ref{Vicondition1}), (\ref{Vicondition2}) and $Z_1\cap Z_2\not=\emptyset$.  Let $0<\beta<a^*$, $0<\alpha_1,\alpha_2<a^*-\beta:=\alpha^*_{\beta}$, $(m_{1,\bf{a}},w_{1,\bf{a}},m_{2,\bf{a}},w_{2,\bf{a}})$ be a minimizer of $e_{\alpha_1,\alpha_2,\beta}$ with $\bf{a}:=(\alpha_1,\alpha_2)$ and $(m_{1,\bf{a}},u_{1,\bf{a}},m_{2,\bf{a}},u_{2,\bf{a}})$ be a solution of (\ref{ss1thm11}).  Define ${\bf {a}}^*_{\beta}:=(\alpha_{\beta}^*,\alpha_{\beta}^*)=(a^*-\beta,a^*-\beta),$ then as $\bf{a}\nearrow \bf{a}_{\beta}^*$, we have for $i=1,2,$ 
\begin{small}
\begin{align}\label{thm13conclusion1}
\lim_{\bf{a}\nearrow \bf{a}_{\beta}^*}\bigg(\int_{\mathbb R^N}C_L\bigg|\frac{w_{i,\bf{a}}}{m_{i,\bf{a}}}\bigg|^{\gamma'}m_{i,\bf{a}}\,dx-\frac{N(\alpha_i+\beta)}{N+\gamma'}
\int_{\mathbb R^N}m_{i,\bf{a}}^{1+\frac{\gamma'}{N}}\,dx\bigg)=0,
\end{align}
\end{small}
\begin{align}\label{thm13conclusion2}
\lim_{\bf{a}\nearrow \bf{a}^*_{\beta}}\int_{\mathbb R^N}V_1(x)m_{1,\bf{a}}+V_2(x)m_{2,\bf{a}}\,dx=0,~\lim_{\bf{a}\nearrow \bf{a}^*_{\beta}} \int_{\mathbb R^N}\bigg(m_{1,\bf{a}}^{\frac{1}{2}+\frac{\gamma'}{2N}}-m_{2,\bf{a}}^{\frac{1}{2}+\frac{\gamma'}{2N}}\bigg)^2\,dx=0,
\end{align}
\begin{align}\label{202401719conclu}
~\lim_{\bf{a}\nearrow \bf{a}^*_{\beta}}C_L\int_{\mathbb R^N}\bigg|\frac{w_{i,\bf{a}}}{m_{i,\bf{a}}}\bigg|^{\gamma'}m_{i,\bf{a}}\,dx\rightarrow +\infty \text{~for both~}i=1,2
\end{align}
and
\begin{align}\label{thm13conclusionpart23}
\lim_{\bf{a}\nearrow \bf{a}^*_{\beta}}\frac{\int_{\mathbb R^N}\big|\frac{w_{1,\bf{a}}}{m_{1,\bf{a}}}\big|^{\gamma'}m_{1,\bf{a}}\,dx}{\int_{\mathbb R^N}\big|\frac{w_{2,\bf{a}}}{m_{2,\bf{a}}}\big|^{\gamma'}m_{2,\bf{a}}\,dx}=1,~~\lim_{\bf{a}\nearrow \bf{a}_{\beta}^*}\frac{\int_{\mathbb R^N}m_{1,\bf{a}}^{1+\frac{\gamma'}{N}}\,dx}{\int_{\mathbb R^N}m_{2,\bf{a}}^{1+\frac{\gamma'}{N}}\,dx}=1.
\end{align}
Moreover, define
\begin{align}\label{defvarepsilon316}
\varepsilon:=\varepsilon_{\textbf{a}}:=\bigg(C_L\int_{\mathbb R^N}\bigg|\frac{w_{1,\textbf{a}}}{m_{1,\textbf{a}}}\bigg|^{\gamma'}m_{1,\textbf{a}}\bigg)^{-\frac{1}{\gamma'}}\rightarrow 0.
\end{align}
Let $x_{i,\varepsilon}$, $i=1,2$ be one global minimal point of $u_{i,\textbf{a}}$ and $y_{i,\varepsilon}$, $i=1,2$ be one global maximal point of $m_{i,\textbf{a}}$.  Then we have up to a subsequence $\exists x_0$ s.t. $V_1(x_0)=V_2(x_0)=0,$ and
\begin{align*}
x_{i,\varepsilon},y_{i,\varepsilon}\rightarrow x_0,\text{ as }\textbf{a}\nearrow \textbf{a}_{\beta}^*;
\end{align*}
moreover, we find
\begin{align}\label{x1varepsilonminusx2varepsilon}
\limsup_{\varepsilon\rightarrow 0^+}\frac{|x_{1,\varepsilon}-x_{2,\varepsilon}|}{\varepsilon}<+\infty,
\end{align}
and
\begin{align}\label{moreoverholdsfinal}
\limsup_{\varepsilon\rightarrow 0^+}\frac{|x_{i,\varepsilon}-y_{j,\varepsilon}|}{\varepsilon}<+\infty,~~i,j=1,2.
\end{align}
In addition, let
\begin{align}\label{scalingthm31profile}
u_{i,\varepsilon}:=\varepsilon^{\frac{2-\gamma}{\gamma-1}}u_{i,\textbf{a}}(\varepsilon x+x_{1,\varepsilon}),~m_{i,\varepsilon}:=\varepsilon^Nm_{i,\textbf{a}}(\varepsilon x+x_{1,\varepsilon}),~{{w}}_{i,\varepsilon}:=\varepsilon^{N+1}w_{i,\textbf{a}}(\varepsilon x+x_{1,\varepsilon}), 
\end{align}
then there exist $  u\in C^2(\mathbb R^N)$, $0\leq m\in W^{1,\gamma'}(\mathbb R^N)$, and ${{w}}\in L^{\gamma'}(\mathbb R^N)$  such that
\begin{align}\label{mlimiting20923}
u_{i,\varepsilon}\rightarrow u\text{~in~}C_{\text{loc}}^2(\mathbb R^N),~~m_{i,\varepsilon}\rightarrow m\text{~in~}L^p(\mathbb R^N),~\forall p\geq 1,~~{{w}}_{i,\varepsilon}\rightharpoonup {{w}}~\text{in~}L^{\gamma'}(\mathbb R^N),~i=1,2.
\end{align}
 In particular, $(m,{{w}})$ is a minimizer of problem \eqref{GNinequalitybest} and $(u,m,{{w}})$ solves
\begin{align}\label{satisfythm13eqsinglezeropotential}
\left\{\begin{array}{ll}
-\Delta u+C_H|\nabla u|^{\gamma}-\frac{\gamma'}{N}=-a^*m^{\frac{\gamma'}{N}},&x\in\mathbb R^N,\\
\Delta m+C_H\gamma\nabla\cdot(m|\nabla u|^{\gamma-2}\nabla u)=0,~{w}=-C_H\gamma |\nabla u|^{\gamma-2}\nabla u,&x\in\mathbb R^N,\\
\int_{\mathbb R^N}m\,dx=1.
\end{array}
\right.
\end{align}
\end{theorem}
Theorem \ref{thm13attractive} implies that as $(\alpha_1,\alpha_2)\nearrow (a^*-\beta,a^*-\beta)$, there are concentration phenomena in the multi-population Mean-field Games system (\ref{ss1}) with attractive interactions under the mass critical exponent case.  In addition, the basic blow-up profiles of ground states are given in Theorem \ref{thm13attractive}. 
 Moreover, by imposing the local polynomial expansion on potential functions, we obtain the following results of refined blow-up profiles:
\begin{theorem}\label{thm16refinedblowup}
Assume all conditions in Theorem \ref{thm13attractive} hold.  Suppose that  $V_1(x)$ and $V_2(x)$ have $l$ common global minimum points, i.e., $Z_1\cap Z_2=\{x_1,\cdots,x_l\in\mathbb R^N\}$, and there exist $ d>0$, $a_{ij}>0$, $p_{ij}>0$ with $i=1,2$, $j=1,\cdots,l$ such that 
    \begin{align}\label{Vi125negativebeta2024}
    V_i=a_{ij}|x-x_j|^{p_{ij}}+O(|x-x_j|^{p_{ij}+1})\text{ for }0<|x-x_j|<d.
    \end{align}
Let $p_j:=\min\{p_{1j},p_{2j}\}$, $p_0:=\max_{1\leq j\leq l}p_j $ and
\begin{align}\label{mujthm1point5}
\mu_j=\lim_{x\rightarrow x_j}\frac{V_1(x)+V_2(x)}{|x-x_j|^{p_j}}=&\left\{\begin{array}{ll}
a_{1j},\text{ if }p_{1j}<p_{2j},\\
a_{1j}+a_{2j},\text{ if }p_{1j}=p_{2j},\\
a_{2j},\text{ if }p_{1j}>p_{2j}.
\end{array}
\right.
\end{align}
Define 
\begin{align}\label{eq-Z0}
    Z_0=\{x_j|x_j\in\bar Z\text{ and }\mu_j=\mu\}\text{ with
$\bar Z:=\{x_j|p_{j}=p_0,~j=1,\cdots,l\}$ and  $\mu:=\min\{\mu_j|x_j\in \bar Z\}$.}
\end{align}
Let 
$(u_{i,\varepsilon},m_{i,\varepsilon},w_{i,\varepsilon})$, $i=1,2$ be given as (\ref{scalingthm31profile}).  
 Then we have
\begin{align}\label{136refinedbetanegativenoteshold}
\lim_{\varepsilon\rightarrow 0^+}\frac{\varepsilon}{\big(\frac{2\gamma'}{p_0\mu\bar{\nu}_{p_0} a^*}\big)^{\frac{1}{\gamma'+p_0}}\big(a^*-\frac{\alpha_1+\alpha_2+2\beta}{2}\big)^{\frac{1}{\gamma'+p_0}}}=1,
\end{align}
and 
\begin{align}\label{136refinedbetanegativenoteshold2}
\frac{x_{1,\varepsilon}-x_0}{\varepsilon}\rightarrow y_0\text{ with }x_0\in Z_0\text{ and $y_0\in\mathbb{R}^N$} satisfying H_{m,p_0}(y_0)=\bar\nu_{p_0},
\end{align}
where $m$ and $\bar \nu_{p_0}$ are given in (\ref{mlimiting20923}) and (\ref{Hmoibarnupi}), respectively.
\end{theorem}

Next, we discuss the blow-up profiles of ground states to (\ref{ss1}) under repulsive interactions.  We remark that on one hand, one has shown in Theorem \ref{thm11multi} that (\ref{ss1}) admits ground states when $0<\alpha_1,\alpha_2<a^*$ and $\beta\leq 0;$ on the other hand, Theorem \ref{thmfinalexistence} indicates that (\ref{problem1p1}) does not have any minimizer when $\alpha_1=\alpha_2=a^*$ and $\beta\leq 0$.  Similarly as discussed in the proof of Theorem \ref{thm13attractive}, we investigate the concentration phenomena in (\ref{ss1}) with repulsive interactions and obtain
\begin{theorem}\label{thm15blowupnegative}
Assume that  $V_i(x)$ with $i=1,2$ satisfy (H1) and (H2) given by (\ref{Vicondition1}) and (\ref{Vicondition2}), respectively.  Suppose 
\begin{align}\label{c26notesbetanegative}
Z_1\cap Z_2=\emptyset,
\end{align}
where $Z_1$ and $Z_2$ are given by \eqref{zerosdefinition20241005}.  Let $\beta< 0$, $0<\alpha_1,\alpha_2<a^*$, $(m_{1,\bf{a}},w_{1,\bf{a}},m_{2,\bf{a}},w_{2,\bf{a}})$ be a minimizer of $e_{\alpha_1,\alpha_2,\beta}$ with $\bf{a}:=(\alpha_1,\alpha_2)$ and $(m_{1,\bf{a}},u_{1,\bf{a}},m_{2,\bf{a}},u_{2,\bf{a}})$ be a solution of (\ref{ss1thm11}).  Define ${\bf {a}}^*:=(a^*,a^*),$ then we have as $\bf{a}\nearrow \bf{a}^*$,
\begin{align}\label{C7notesnegative}
\lim_{\textbf{a}\nearrow\textbf{a}^*}\bigg(\int_{\mathbb R^N}C_L\bigg|\frac{w_{i,\textbf{a}}}{m_{i,\textbf{a}}}\bigg|^{\gamma'}m_{i,\textbf{a}}\,dx+\int_{\mathbb R^N}V_im_{i,\textbf{a}}\,dx-\frac{N\alpha_i}{N+\gamma'}\int_{\mathbb R^N}m_{i,\textbf{a}}^{1+\frac{\gamma'}{N}}\,dx\bigg)=0;
\end{align}
\begin{align}\label{C8notesnegative}
\lim_{\textbf{a}\nearrow \textbf{a}^*}\int_{\mathbb R^N} m_{1,\textbf{a}}^{\frac{1}{2}+\frac{\gamma'}{2N}}m_{2,\textbf{a}}^{\frac{1}{2}+\frac{\gamma'}{2N}}\,dx=0,
\end{align}
and
\begin{align}\label{C9notesnegative}
\int_{\mathbb R^N}V_1m_{1,\textbf{a}}+V_2m_{2,\textbf{a}}\,dx\rightarrow 0;
\end{align}
\begin{align}\label{C10betanegative}
C_L\int_{\mathbb R^N}\big|\frac{w_{i,\textbf{a}}}{m_{i,\textbf{a}}}\big|^{\gamma'}m_{i,\textbf{a}}\,dx\rightarrow+\infty,~~\int_{\mathbb R^N}m_{i,\textbf{a}}^{1+\frac{\gamma'}{N}}\,dx\rightarrow+\infty,~~i=1,2.
\end{align}
Moreover, define
\begin{align*}
\hat\varepsilon_{i}:=\bigg(C_L\int_{\mathbb R^N}\bigg|\frac{w_{i,\textbf{a}}}{m_{i,\textbf{a}}}\bigg|^{\gamma'}m_{i,\textbf{a}}\,dx\bigg)^{-\frac{1}{\gamma'}}\rightarrow 0\text{~as~}\textbf{a}\nearrow \textbf{a}^*,~i=1,2.
\end{align*}
Let $x_{i,\hat\varepsilon}$, $i=1,2$ be a global minimum point of $u_{i,\bf{a}}$ and
\begin{align}\label{hatepsilonrescaling20240723}
m_{i,\hat\varepsilon}=\hat \varepsilon^N_im_{i,\textbf{a}}(\hat \varepsilon_i x+x_{i,\hat \varepsilon}),~~w_{i,\hat\varepsilon}=\hat\varepsilon_i^{N+1}w_{i,\textbf{a}}( \hat \varepsilon_i x+x_{i,\hat \varepsilon}),~u_{i,\hat \varepsilon}=\hat \varepsilon_i^{\frac{2-\gamma}{\gamma-1}}u_{i,\textbf{a}}(\hat \varepsilon_i x+x_{i,\hat \varepsilon}),
\end{align}
then there exist $  (u_i,m_i, w_i)\in C^2(\mathbb R^N)\times W^{1,\gamma'}(\mathbb R^N)\times L^{\gamma'}(\mathbb R^N)$ with $i=1,2$ 
such that
\begin{align}\label{mlimiting0923}
u_{i,\varepsilon}\rightarrow u_i\text{~in~}C_{\text{loc}}^2(\mathbb R^N),~~m_{i,\varepsilon}\rightarrow m_i\text{~in~}L^p(\mathbb R^N),~\forall p\geq 1,~~{{w}}_{i,\varepsilon}\rightharpoonup {{w}}_i~\text{in~}L^{\gamma'}(\mathbb R^N),~i=1,2.
\end{align}
In particular,  $(m_i,u_i,w_i)$, $i=1,2$ both solve system
\eqref{satisfythm13eqsinglezeropotential}.

\end{theorem}
\begin{remark}
We point out that unlike the attractive case discussed in Theorem \ref{thm13attractive} and Theorem \ref{thm16refinedblowup},  $\hat \varepsilon_1$ and $\hat \varepsilon_2$ given in Theorem \ref{thm15blowupnegative} both converge to zero but might not be in the same order since $\beta<0$ and the behaviors of $V_1$ and $V_2$ might be distinct around global minimum points locally.
\end{remark}
Theorem \ref{thm15blowupnegative} indicates that when the interaction is repulsive, there are concentration phenomena within system (\ref{ss1}) in some singular limit of parameters $\alpha_1$, $\alpha_2$ and $\beta$. 
 Moreover, similarly as the conclusion shown in Theorem \ref{thm16refinedblowup}, we explore the refined blow-up profiles and obtain 

\begin{theorem}\label{thm17multipopulation}
Assume all conditions in Theorem \ref{thm15blowupnegative} hold.  Suppose that each $V_i$, $i=1,2$ has only one global minimum point $x_i$ with $x_1\not=x_2$ and there exist $ d>0$, $b_i>0$ and $q_i>0$ such that
\begin{align}\label{52viinnotesbetacopy}
V_i(x)=b_i|x-x_i|^{q_i}+O(|x-x_i|^{q_i+1})\text{ for }0<|x-x_i|<d.
\end{align}
Define for $i=1,2,$
\begin{align}\label{assumethm17notecopybeta}
\tilde {\epsilon}_i:=(a^*-\alpha_i)^{\frac{1}{\gamma'+q_i}}\ \text{ and assume $\exists\ s\in(0,1]$ such that }\tilde  \epsilon_1=O(\tilde \epsilon^s_2).
\end{align}
Let $(m_{1,\textbf{a}},w_{1,\textbf{a}},m_{2,\textbf{a}},w_{2,\textbf{a}})$ be a minimizer of (\ref{problem1p1}) and $(m_{i,\hat\varepsilon},w_{i,\hat\varepsilon},u_{i,\hat\varepsilon})$ be defined as (\ref{hatepsilonrescaling20240723}).  Then we have
\begin{align}\label{eq-1.40}
\frac{x_{i,\hat\varepsilon}-x_i}{\hat\varepsilon_i}\rightarrow y_{i0}\text{ such that }H_{m_i,q_i}(y_{i0})=\bar\nu_{q_i},
\end{align}
where $m_i$ and $\bar\nu_{q_i}$, $i=1,2$ are given by \eqref{mlimiting0923} and (\ref{Hmoibarnupi}), respectively.
Moreover, the following asymptotics hold as $\bf{a}\nearrow \bf{a}^*$,
\begin{align*}
\hat\varepsilon_i^{\gamma'}=(1+o(1))\bigg(\frac{\gamma'(a^*-\alpha_i)}{a^*b_i\bar\nu_{q_i}q_i}\bigg)^{\frac{1}{\gamma'+q_i}},~i=1,2.
\end{align*}

\end{theorem}
\begin{remark}
    In Theorem \ref{thm17multipopulation}, we discuss the refined blow-up profiles of ground states when the interaction coefficient is non-positive under some technical assumption (\ref{assumethm17notecopybeta}).  We would like to remark that this condition is technical and could be improved if the refined decay estimate of population density $m$ is given.  In fact, the improved condition will be exhibited in Section \ref{sect520240929}.
\end{remark} 

The rest of this paper is organized as follows: In Section \ref{preliminary}, we give some preliminary results for the existence and properties of the solutions to  
 Hamilton-Jacobi equations and Fokker-Planck equations, which are used to investigate the existence and blow-up behaviors of minimizers to problem \eqref{problem1p1} .
 Section \ref{sec-existence} is devoted to the exploration of the effect of    the potentials $V_i(x),$ $i=1,2$ and coefficients $\alpha_1,\alpha_2,\beta$ on the existence of minimizers.  Correspondingly, the proof of Theorems \ref{thm11multi}-\ref{thmfinalexistence} will be finished.  In Section   \ref{sect4multipopulation}, we perform the blow-up analysis of minimizers under the case of attractive interactions $\beta>0$, and show the conclusions of Theorem \ref{thm13attractive} and Theorem \ref{thm16refinedblowup}. Finally, in Section \ref{sect520240929}, we focus on the asymptotic profiles of ground states with $\beta<0$ and complete the proof of Theorem \ref{thm15blowupnegative} and Theorem \ref{thm17multipopulation}.

\section{Preliminary Results}\label{preliminary}
In this section, we collect some preliminaries for the existence and regularities of solutions to Hamilton-Jacobi equations and Fokker-Planck equations, respectively.  Furthermore, some useful equalities and estimates satisfied by the solution to the single population Mean-field Games system will be listed. 
\subsection{Hamilton-Jacobi Equations}
Consider the following second order Hamilton-Jacobi equations:
\begin{align}\label{HJB-regularity}
-\Delta u_k+C_H|\nabla u_k|^{\gamma}+\lambda_k=V_k(x)+f_k(x),\ \ x\in\mathbb R^N,
\end{align}
where $\gamma>1$ is fixed, $C_H$ is a given positive  constant independent of $k$ and $(u_k,\lambda_k)$ denote the solutions to (\ref{HJB-regularity}).  For the gradient estimates of $u_k$, we find
\begin{lemma}\label{sect2-lemma21-gradientu}
Suppose that $f_k\in L^{\infty}(\mathbb R^N)$ satisfies  $\Vert f_k\Vert_{L^\infty}\leq C_f$, $|\lambda_k|\leq \lambda$, and the potential functions $V_k(x)\in C^{0,\theta}_{\rm loc}(\mathbb R^N)$ with $\theta\in(0,1)$ satisfy  $0\leq V_k(x)\rightarrow +\infty$ as $|x|\rightarrow +\infty,$ and $\exists~ R>0$  sufficiently large such that 
\begin{align*}
0< C_1\leq \frac{V_k(x+y)}{V_k(x)}\leq C_2,\text{~for~all~}k\text{~and~all~}|x|\geq R \text{~with~}|y|<2,
\end{align*}
where the positive constants $C_f$, $\lambda$, $R$, $C_1$ and $C_2$ are independent of $k$.  Let $(u_k,\lambda_k)\in C^2(\mathbb R^N)\times \mathbb R$ be a sequence of solutions to (\ref{HJB-regularity}).  Then,  for all $k$,
\begin{align*}
|\nabla u_k(x)|\leq C(1+V_k(x))^{\frac{1}{\gamma}}, \text{ for all } x\in\mathbb{R}^N,
\end{align*}
where constant $C$ depends on $C_H$, $C_1$, $C_2$, $\lambda$, $\gamma$, $N$ and $C_f.$

In particular, if there exist   $b\geq 0$ and $C_{F}>0$  independent of $k,$ such that  following conditions hold on $V_k$
\begin{align}\label{cirant-Vk}
 C_F^{-1}(\max\{|x|-C_F,0\})^b\leq V_k(x)\leq C_F(1+|x|)^b,~~\text{for all }k\text{ and }x\in\mathbb R^N,
 \end{align}
 then  we have 
\begin{align*}
|\nabla u_k|\leq C(1+|x|)^{\frac{b}{\gamma}}, ~\text{for all }k\text{ and }x\in\mathbb R^N,
\end{align*}
where constant $C$ depends on $C_H$, $C_{F}$, $b$, $\lambda$, $\gamma$, $N$ and $C_f.$ 

\end{lemma}
 \begin{proof}
     See Lemma 3.1 in \cite{cirant2024critical} and the argument is the slight modification of the proof of Theorem 2.5 in \cite{cesaroni2018concentration}. \end{proof}
For the lower bound of $u_k$, we have 
\begin{lemma}[C.f. Lemma 3.2 in \cite{cirant2024critical}]\label{lowerboundVkgenerallemma22}
Suppose all conditions in Lemma \ref{sect2-lemma21-gradientu} hold. 
Let $u_k$ be a family of $C^2$ solutions and assume that $u_k(x)$ are bounded from below uniformly.  Then there exist positive constants $C_3$ and $C_4$ independent of $k$ such that 
\begin{align}\label{29uklemma22}
u_k(x)\geq C_3V^{\frac{1}{r'}}_k(x)-C_4,\text{~}\forall x\in\mathbb R^n,~\text{for all }k.
\end{align}
In particular, if the following conditions hold on $V_k$
\begin{align}\label{cirant-Vk-1}
 C_F^{-1}(\max\{|x|-C_F,0\})^b\leq V_k(x)\leq C_F(1+|x|)^b,~~\text{for all }k\text{ and }x\in\mathbb R^n,
 \end{align}
 where constants $b> 0$ and $C_{F}$ are independent of $k,$ then we have 
\begin{align}\label{usolutionlowerestimatepre-11}
 u_k(x)\geq C_3|x|^{1+\frac{b}{r'}}-C_4,\text{~for all }k, x\in\mathbb R^n.
\end{align}
If $b=0$ in (\ref{cirant-Vk-1}) and there exist $R>0$ and $\hat \delta>0$ independent of $k$ such that 
\begin{align*}
f_k+V_k-\lambda_k>\hat \delta>0\text{~for~all }|x|>R,
\end{align*}
then \eqref{usolutionlowerestimatepre-11} also holds.
\end{lemma}

The existence result of the classical solution to (\ref{HJB-regularity}) is summarized as 
\begin{lemma}[C.f. Lemma 3.3 in \cite{cirant2024critical}]  \label{lemma22preliminary}
Suppose $V_k+f_k$ are locally H\"{o}lder continuous and bounded from below uniformly in $k$.  Define 
\begin{align*}
\bar \lambda_k:=\sup\{\lambda\in\mathbb R~|~(\ref{HJB-regularity})\text{ has a solution }u_k\in C^2(\mathbb R^n)\}.
\end{align*}
Then 
\begin{itemize}
    \item[(i).] $\bar \lambda_k$ are finite for every $k$ and (\ref{HJB-regularity}) admits a solution $(u_k,\lambda_k)\in C^2(\mathbb R^n)\times \mathbb R$  with $\lambda_k=\bar \lambda_k$ and $u_k(x)$ being bounded from below (may not uniform in $k$).  Moreover,
    $$\bar \lambda_k=\sup\{\lambda\in\mathbb R~|~(\ref{HJB-regularity})\text{ has a subsolution }u_k\in C^2(\mathbb R^n)\}.$$
    \item[(ii).] If $V_k$ satisfies (\ref{cirant-Vk}) with $b>0$, then $u_k$ is unique up to constants for fixed $k$ and there exists a positive constant $C$ independent of $k$ such that 
    \begin{align}\label{lowerboundusect2}
    u_k(x)\geq C|x|^{\frac{b}{r'}+1}-C, \forall x\in\mathbb R^n.
    \end{align}
    In particular, if $V_k\equiv 0$ in (\ref{HJB-regularity}) and there exists $\sigma>0$ independent of $k$ such that 
    \begin{align*}
    f_k-\lambda_k\geq \sigma>0, \ \ \text{for~} |x|>K_2,
    \end{align*}
    where $K_2>0$ is a large constant independent of $k$, then (\ref{lowerboundusect2}) also holds.
 \end{itemize}
 (iii). If $V_k$ satisfies \eqref{Vicondition2}, then there exist uniformly bounded from below classical solutions $u_k$ to problem (\ref{HJB-regularity}) satisfying estimate (\ref{29uklemma22}).
\end{lemma}

\subsection{Fokker-Planck Equations}

Of concern the second order Fokker-Planck equation
\begin{align}\label{sect2-FP-eq}
-\Delta m+\nabla\cdot w=0,\ \ x\in\mathbb R^N,
\end{align}
where $w$ is given and $m$ denotes the solution, we have the following results for the regularity:
\begin{lemma}\label{lemma21-crucial-cor}
Let $(m,w)\in \left(L^1(\mathbb R^N)\cap W^{1, \hat q}(\mathbb R^N)\right)\times L^1(\mathbb R^N)$ be a  solution to (\ref{sect2-FP-eq}) with 
\begin{equation*}
\hat q:=\begin{cases}
\frac{N}{N-\gamma'+1} &\text{ if }\gamma'<N,\\
\in\left(\frac{2N}{N+2},N\right)&\text{ if }\gamma'=N,\\
\gamma' &\text{ if }\gamma'> N.
\end{cases}
\end{equation*}
Assume that 
\begin{equation*}
\Lambda_{\gamma'}:=\int_{\mathbb R^n}|m|\Big|\frac{w}{m}\Big|^{\gamma'}\, dx<\infty,
\end{equation*}
then we have $w\in L^{1}(\mathbb R^N)\cap L^{\hat q}(\mathbb R^N)$ and there exists $\mathcal{C}=\mathcal{C}(\Lambda_{\gamma'},\|m\|_{L^1(\mathbb R^N)})>0$ such that
$$\|m\|_{W^{1,\hat q}(\mathbb R^N)}, \|w\|_{L^1(\mathbb R^N)},\|w\|_{L^{\hat q}(\mathbb R^N)}\leq \mathcal{C}.$$

\end{lemma}

\begin{proof}
    See the proof of Lemma 3.5 in \cite{cirant2024critical}.
\end{proof}

Next, we state some useful identities satisfied by the single population Mean-field Games system.  First of all, we have the exponential decay estimates of $m$ when some condition is imposed on the Lagrange multiplier, which is 

\begin{lemma}[ C.f. Proposition 5.3 in \cite{cesaroni2018concentration} ]\label{mdecaylemma}
Assume $\gamma'>N$.  Let $(u,\lambda, m)\in C^2(\mathbb R^n)\times \mathbb R\times \big(W^{1,\gamma'}(\mathbb R^n)\cap L^1(\mathbb R^n)\big)$ with $u$ bounded from below,  and $\lambda<0$ be the solution of the following Mean-field Games system
\begin{align}\label{26preliminaryfinal}
\left\{\begin{array}{ll}
-\Delta u+C_H|\nabla u|^{\gamma}+\lambda=-m^{\nu}, &x\in\mathbb R^N,\\
\Delta m+C_H\gamma\nabla\cdot(m|\nabla u|^{\gamma-2}\nabla u )=0, &x\in\mathbb R^N,
\end{array}
\right.
\end{align}
where $\nu\in(0,\frac{\gamma'}{N}].$
 Then, we have there exist $\kappa_1,\kappa_2>0$ such that 
\begin{align*}
m(x)\leq \kappa_1 e^{-\kappa_2|x|}  ~\text{ for all } x\in \mathbb R^N.
\end{align*}
\end{lemma}
With the aid of Lemma \ref{mdecaylemma}, we have the following results for the Pohozaev identities satisfied by the solution to system (\ref{26preliminaryfinal}):
\begin{lemma}[C.f. Proposition 3.1 in \cite{cirant2016stationary}]\label{poholemma}
Assume all conditions in Lemma  \ref{mdecaylemma} hold and denote $w=-C_H\gamma m|\nabla u|^{\gamma-2}\nabla u$. 
 Then we have the following Pohozaev type identities hold:
\begin{align*}
\left\{\begin{array}{ll}
\lambda\int_{\mathbb R^N}m\, dx=-\frac{(\nu+1)\gamma'-N\nu}{(\alpha+1)\gamma'}\int_{\mathbb R^N}m^{\nu+1}\,dx,\\
C_L\int_{\mathbb R^N}m\big|\frac{w}{m}\big|^{\gamma'}\, dx=\frac{N\nu}{(\nu+1)\gamma'}\int_{\mathbb R^N}m^{\nu+1}\, dx=(\gamma-1)C_H\int_{\mathbb R^N} m|\nabla u|^{\gamma}\, dx.
\end{array}
\right.
\end{align*}

\end{lemma}
 
\section{Existence of ground states}\label{sec-existence}
In this section, we shall discuss the existence of ground states to system (\ref{ss1}) under some conditions of coefficients $\alpha_i$ with $i=1,2$ and $\beta$.  To this end, we first estimate the energy $\mathcal E_{\alpha_1,\alpha_2,\beta}(m_1,w_1,m_2,w_2)$ from below.  Then, if the energy is shown to have some finite lower bound and the minimizers is proved to exist, we will find the existence of ground states to (\ref{ss1}) by the standard duality argument.  Before stating our main results for the existence of minimizers, we give some preliminary definitions, which are
\begin{align}\label{problem51innotescopy}
e^i_{\alpha_i}:=\inf_{(m,w)\in\mathcal K_i}\mathcal E^i_{\alpha_i}(m,w),~i=1,2,
\end{align}
where $\mathcal K_i$ is given by \eqref{mathcalkidefined} and
\begin{align}\label{mathcalealphaii089}
\mathcal E_{\alpha_i}^i(m,w)=C_L\int_{\mathbb R^N}\bigg|\frac{w}{m}\bigg|^{\gamma'}m\,dx+\int_{\mathbb R^N}V_im\,dx-\frac{\alpha_i}{1+\frac{\gamma'}{N}}\int_{\mathbb R^N}m^{1+\frac{\gamma'}{N}}\,dx.
\end{align}
Concerning the existence of ground states in (\ref{ss1}), we have
\begin{lemma}\label{lemma31existenceleastenergy}
Assume all conditions in Theorem \ref{thm11multi} hold, then we have
\begin{itemize}
    \item[(i).] if $0<\alpha_1<a^*$, $0<\alpha_2<a^*$ and $-\infty<\beta<\beta_*:=\sqrt{(a^*-\alpha_1)(a^*-\alpha_2)}$, then problem (\ref{problem1p1}) has a global minimizer $(m_{1,\textbf{a}},w_{1,\textbf{a}},m_{2,\textbf{a}},w_{2,\textbf{a}})\in \mathcal K$;
\item[(ii). ] either $\alpha_1>a^*$ or $\alpha_2>a^*$ or $\beta>\beta^*:=\frac{2a^*-\alpha_1-\alpha_2}{2},$ then problem (\ref{problem1p1}) has no minimizer.
\end{itemize}
\end{lemma}
\begin{proof}
(i).  Invoking inequality (\ref{GNinequalityused}) and condition (\ref{Vicondition1}) satisfied by $V_i$ with $i=1,2$, we have for any $(m_1,w_1,m_2,w_2)\in\mathcal K$,
\begin{align}\label{citeinequality1}
&\mathcal E_{\alpha_1,\alpha_2,\beta}(m_1,w_1,m_2,w_2)\nonumber\\
\geq &  \sum_{i=1}^2\int_{\mathbb R^N}V_im_i\,dx+\frac{N}{N+\gamma'}\bigg[\sum_{i=1}^2(a^*-\alpha_i)\int_{\mathbb R^N}m_i^{1+\frac{\gamma'}{N}}\,dx-2\beta\int_{\mathbb R^N} m_1^{\frac{1}{2}+\frac{\gamma'}{2N}}m_2^{\frac{1}{2}+\frac{\gamma'}{2N}}\,dx\bigg]\nonumber\\
\geq &\frac{2(\beta_*-\beta)N}{N+\gamma'}\int_{\mathbb R^N}m_1^{\frac{1}{2}+\frac{\gamma'}{2N}}m_2^{\frac{1}{2}+\frac{\gamma'}{2N}}\,dx,
\end{align}
where $\mathcal E_{\alpha_1,\alpha_2,\beta}$ is given by (\ref{energy1p3}).  Then, letting $\{(m_{1,k},w_{1,k},m_{2,k},w_{2,k}\}\subset \mathcal K$ with $k\in\mathbb Z^+$ being a minimizing sequence of $e_{\alpha_1,\alpha_2,\beta}$, one has from (\ref{citeinequality1}) and $-\infty<\beta<\beta_*$ that
\begin{align}\label{between31and32embedding}
\sup_{k}\sum_{i=1}^2 \int_{\mathbb R^N} V_im_{i,k}\,dx<+\infty,~~\sup_{k}\int_{\mathbb R^N} m_{1,k}^{\frac{1}{2}+\frac{\gamma'}{2N}}m_{2,k}^{\frac{1}{2}+\frac{\gamma'}{2N}}\,dx<+\infty,
\end{align}
and then
\begin{align}\label{moreimportantthm11justminimizer}
\sup_{k}\sum_{i=1}^2\int_{\mathbb R^N}\bigg|\frac{w_{i,k}}{m_{i,k}}\bigg|^{\gamma'}m_{i,k}\,dx<+\infty.
\end{align}
Thanks to Lemma \ref{lemma21-crucial-cor} and (\ref{moreimportantthm11justminimizer}), one obtains as $k\rightarrow +\infty,$ for $i=1,2$,
\begin{align*}
(m_{i,k},w_{i,k})\rightharpoonup (m_{i,\textbf{a}},w_{i,\textbf{a}})\text{~in~}W^{1,\gamma'}(\mathbb R^N)\times L^{\gamma'}(\mathbb R^N).
\end{align*}
Moreover, by the compactly Sobolev embedding (C.f. Lemma 5.1 in \cite{cirant2024critical}) and Fatou's lemma, we find from (\ref{between31and32embedding}) that
 \begin{align*}
m_{i,k}\rightarrow m_{i,\textbf{a}}\text{~in~}L^{1+\frac{\gamma'}{N}}(\mathbb R^N)\cap L^1(\mathbb R^N).
\end{align*}
Then it follows that $(m_{1,\textbf{a}},w_{1,\textbf{a}},m_{2,\textbf{a}},w_{2,\textbf{a}})\in \mathcal K$ is a minimizer.

(ii). Let $\mathcal M$ be given by (\ref{mathcalMdefinedbythm17}).
Since $\gamma'>N$, by using Morrey's embedding, the standard elliptic regularity and the maximum principle, one follows the idea shown in \cite{bernardini2023ergodic} then obtain for any $(m,w)\in \mathcal M$, $m(x)>0$ for all $x\in\mathbb R^N.$  Next, we utilize some rescaled pair of $(m_0,w_0)\in \mathcal M$ to analyze the bound of $\mathcal E_{\alpha_1,\alpha_2,\beta}$ from below.

Let $(m_0,w_0)\in \mathcal M$ and define
\begin{align}\label{byusingintro1}
(m_t,w_t)=\bigg(\frac{t^N}{M^*}m_0(t(x-x_0)),\frac{t^{N+1}}{M^*}w_0(t(x-x_0))\bigg),~\text{for~}  t>0\text{~and~} x_0\in\mathbb R^N.
\end{align}
From Lemma \ref{mdecaylemma} and Lemma \ref{poholemma}, we have that 
\begin{align}\label{combine20247131}
C_L\int_{\mathbb R^N}\bigg|\frac{w_0}{m_0}\bigg|^{\gamma'}m_0\,dx=1,~~\int_{\mathbb R^N}m_0^{1+\frac{\gamma'}{N}}\,dx=\frac{N+\gamma'}{N},~~\int_{\mathbb R^N}m_0\,dx=M^*.
\end{align}
Combining (\ref{byusingintro1}) with (\ref{combine20247131}), one finds
\begin{align}\label{onefindsintro1}
C_L\int_{\mathbb R^N}\bigg|\frac{w_t}{m_t}\bigg|^{\gamma'}m_t\,dx=C_L\frac{t^{\gamma'}}{M^*}\int_{\mathbb R^N}\bigg|\frac{w_0}{m_0}
\bigg|^{\gamma'}m_0\,dx=\frac{t^{\gamma'}}{M^*},
\end{align}
and 
\begin{align}\label{onefindsintro2}
\int_{\mathbb  R^N}m_t^{1+\frac{\gamma'}{N}}\,dx=\frac{t^{N(1+\frac{\gamma'}{N})}}{(M^*)^{1+\frac{\gamma'}{N}}}\int_{\mathbb R^N}m^{1+\frac{\gamma'}{N}}_0(tx)\,dx=\frac{N+\gamma'}{N}\frac{t^{\gamma'}}{(M^*)^{1+\frac{\gamma'}{N}}}.
\end{align}
Then it follows from (\ref{mathcalealphaii089}), (\ref{byusingintro1}), (\ref{onefindsintro1}) and (\ref{onefindsintro2}) that 
\begin{align}\label{colecting1intro}
\mathcal E^1_{\alpha_1}(m_t,w_t)=&C_L\int_{\mathbb R^N}\bigg|\frac{w_t}{m_t}\bigg|^{\gamma'}m_t\,dx-\frac{N\alpha_1}{N+\gamma'}\int_{\mathbb R^N}m_t^{1+\frac{\gamma'}{N}}\,dx+\int_{\mathbb R^N}V_1m_t\,dx\nonumber\\
=&\frac{t^{\gamma'}}{M^*}\bigg(1-\frac{\alpha_1}{a^*}\bigg)+\frac{1}{M^*}\int_{\mathbb R^N}V\bigg(\frac{x}{t}+x_0\bigg)m_0\,dx.
\end{align}
On the other hand, we choose 
\begin{align*}
\bar m=\frac{e^{-\delta_1|x|}}{\Vert e^{-\delta_1 |x|}\Vert_{L^1}},~~\bar w=\nabla \bar m\text{ with }(\bar m,\bar w)\in  \mathcal K_2,
\end{align*}
and apply H\"{o}lder's inequality to get 
\begin{align}\label{colecting2intro}
\int_{\mathbb R^N}m_t^{\frac{1}{2}+\frac{\gamma'}{2N}}\bar m^{\frac{1}{2}+\frac{\gamma'}{2N}}\,dx\leq &\bigg(\int_{\mathbb R^N}m_t^{1+\frac{\gamma'}{N}}\,dx\bigg)^{\frac{1}{2}}\bigg(\int_{\mathbb R^N}\bar m^{1+\frac{\gamma'}{N}}\,dx\bigg)^{\frac{1}{2}}
\leq  Ct^{\frac{\gamma'}{2}},
\end{align}
where $C>0$ is some constant.  Upon collecting (\ref{colecting1intro}) and (\ref{colecting2intro}), we obtain if $\alpha_1>a^*$,
\begin{align*}
\mathcal E_{\alpha_1,\alpha_2,\beta}(m_t,w_t,\bar m,\bar w)\geq \frac{t^{\gamma'}}{M^*}\bigg(1-\frac{\alpha_1}{a^*}\bigg)-Ct^{\frac{\gamma'}{2}}-C\rightarrow -\infty,\text{~as } t\rightarrow+\infty.
\end{align*}
Thus,
$e_{\alpha_1,\alpha_2,\beta}=-\infty \text{~when~}\alpha_1>a^*.$
Similarly, we find if $\alpha_2>a^*,$ then $e_{\alpha_1,\alpha_2,\beta}=-\infty$.
Consequently, we have if any $\alpha_i>a^*$ or $\alpha_2>a^*$, problem (\ref{problem1p1}) does not have a minimizer.

It is left to study the case of $\beta>\beta^*$.  To this end, we compute and obtain
\begin{align*}
\mathcal E_{\alpha_1,\alpha_2,\beta}(m_t,w_t,m_t,w_t)=&\frac{t^{\gamma'}}{M^*}\bigg(2-\frac{\alpha_1}{a^*}-\frac{\alpha_2}{a^*}-\frac{2\beta}{a^*}\bigg)+O(1)\rightarrow -\infty,\text{ as } t\rightarrow +\infty,
\end{align*}
 when $\beta>\beta^*:=\frac{2a^*-\alpha_1-\alpha_2}{2}.$  This completes the proof.

\end{proof}

Lemma \ref{lemma31existenceleastenergy} states some existence results for the global minimizers $(m_1,w_1,m_2,w_2)$ to (\ref{problem1p1}) under some conditions of $\alpha_1$, $\alpha_2$ and $\beta$.  In particular, when intra-population and inter-population coefficients are all small, Lemma \ref{lemma31existenceleastenergy} implies there exists a minimizer to (\ref{problem1p1}).  Whereas, the existence of ground states to (\ref{ss1}) can not be shown unless $(u_1,u_2)$ and $(\lambda_1,\lambda_2)$ are obtained.    Hence, to finish the proof of Theorem \ref{thm11multi}, we establish the following lemma for the existence of the value function pair $(u_1,u_2)$ and Lagrange multipliers $(\lambda_1,\lambda_2)$:
\begin{lemma}\label{lemma32multimfg}
    Let $(m_{1,\textbf{a}},{{w}}_{1,\textbf{a}},m_{2,\textbf{a}},{{w}}_{2,\textbf{a}})\in\mathcal K$ be a minimizer of $e_{\alpha_1,\alpha_2,\beta}$ with $\mathcal K=\mathcal K_1\times \mathcal K_2$ defined by (\ref{mathcalkidefined}), then there exist $(u_{1,\textbf{a}},u_{2,\textbf{a}})\in \big(C^2(\mathbb R^N)\big)^2$ and $(\lambda_{1,\textbf{a}},\lambda_{2,\textbf{a}})\in\mathbb {R}^2$ such that $(m_{1,\textbf{a}},u_{1,\textbf{a}},m_{2,\textbf{a}},u_{2,\textbf{a}},\lambda_{1,\textbf{a}},\lambda_{2,\textbf{a}})$ solves
\begin{align}\label{ss11new}
\left\{\begin{array}{ll}
-\Delta u_1+C_H|\nabla u_1|^{\gamma}+\lambda_1=V_1(x)-\alpha_1 m_1^{\frac{\gamma'}{N}}-\beta m_1^{\frac{\gamma'}{2N}-\frac{1}{2}}m_2^{\frac{1}{2}+\frac{\gamma'}{N}},&x\in\mathbb R^N,\\
\Delta m_1+\nabla\cdot {{w}}_1=0,~{{w}}_1=-\gamma C_Hm_1|\nabla u_1|^{\gamma-2}\nabla u_1&x\in\mathbb R^N,\\
-\Delta u_2+C_H|\nabla u_2|^{\gamma}+\lambda_2=V_2(x)-\alpha_2 m_2^{\frac{\gamma'}{N}}-\beta m_2^{\frac{\gamma'}{2N}-\frac{1}{2}}m_1^{\frac{1}{2}+\frac{\gamma'}{N}},&x\in\mathbb R^N,\\
\Delta m_2+\nabla\cdot {{w}}_2=0,~{{w}}_2=-\gamma C_Hm_2|\nabla u_2|^{\gamma-2}\nabla u_2,&x\in\mathbb R^N.
\end{array}
\right.
\end{align}
  Moreover, we have the following identities and estimates hold:
\begin{small}
\begin{align}\label{lambda177innotes}
\lambda_{i,\textbf{a}}=C_L\int_{\mathbb R^N}\bigg|\frac{w_{i,\textbf{a}}}{m_{i,\textbf{a}}}\bigg|^{\gamma'}m_{i,\textbf{a}}\,dx+\int_{\mathbb R^N}V_im_{i,\textbf{a}}\,dx-\alpha_i\int_{\mathbb R^N}m_{i,\textbf{a}}^{1+\frac{\gamma'}{N}}\,dx-\beta\int_{\mathbb R^N}m_{1,\textbf{a}}^{\frac{1}{2}+\frac{\gamma'}{2N}}m_{2,\textbf{a}}^{\frac{1}{2}+\frac{\gamma'}{2N}}\,dx,~i=1,2,
\end{align}
\end{small}
and there exists a constant $C>0$ such that 
\begin{align}\label{gradientuiViloweboun}
|\nabla u_{i,\textbf{a}}(x)|\leq C\Big(1+V_i^{\frac{1}{\gamma}}(x)\Big),~~u_{i,\textbf{a}}(x)\geq C V_i^{\frac{1}{\gamma}}-C,\text{ for all }x\in\mathbb R^N,~i=1,2.
\end{align}

\end{lemma}
\begin{proof}
To prove this lemma, we follow the approaches employed to show Proposition 3.4 in \cite{cesaroni2018concentration} and make slight modifications.  Define admissible sets $\mathcal A_i$ as 
\begin{align*}
\mathcal A_i=\bigg\{\psi\in C^2(\mathbb R^N)\bigg|\limsup_{|x|\rightarrow \infty}\frac{|\nabla \psi|}{V_i^{\frac{1}{\gamma}}}<+\infty,~\limsup_{|x|\rightarrow \infty}\frac{|\Delta \psi|}{V_i}<+\infty~\bigg\},~~i=1,2,
\end{align*}
then we proceed the similar argument shown in the proof of Proposition 5.1 in \cite{cirant2024critical} and obtain 
\begin{align}\label{80innotessecondeqibp}
-\int_{\mathbb R^N}m_{i,\textbf{a}}\Delta\psi\,dx=\int_{\mathbb R^N}{w}_{i,\textbf{a}}\cdot \nabla\psi\,dx,~~\forall \psi\in \mathcal A_i,~i=1,2.
\end{align}

Next, we define 
\begin{align}\label{definitionJ1tilde}
\tilde J_{1}(m,{w}):=\int_{\mathbb R^N}\bigg[C_L\bigg|\frac{w}{m}\bigg|^{\gamma'}m+[V_1(x)+f_1(m_{1,\textbf{a}},m_{2,\textbf{a}})]m\bigg]\,dx,
\end{align}
where 
\begin{align*}
f_1(m_{1,\textbf{a}},m_{2,\textbf{a}}):=-\alpha_1m_{1,\textbf{a}}^{\frac{\gamma'}{N}}-\beta m_{2,\textbf{a}}^{\frac{\gamma'}{2N}+\frac{1}{2}}m_{1,\textbf{a}}^{\frac{\gamma'}{2N}-\frac{1}{2}},
\end{align*}
and set
\begin{align*}
\mathcal B_i:=\bigg\{&(m,w)\in (L^1(\mathbb R^N)\cap W^{1,\gamma'}(\mathbb R^N))\times L^{\gamma'}(\mathbb R^N)\bigg|-\int_{\mathbb R^N}m\Delta \psi\,dx=\int_{\mathbb R^N}{w}\cdot\nabla\psi\,dx~,\forall \psi\in \mathcal A_i,\\
&m\geq 0\text{~a.e.~in~}\mathbb R^N,~\int_{\mathbb R^N}m\,dx=1,~\int_{\mathbb R^N}V_im\,dx<+\infty,~\int_{\mathbb R^N} |{w}|V_i^{\frac{1}{\gamma'}}\,dx<+\infty \bigg\},~i=1,2.
\end{align*}
We have the fact that $(m_{1,\textbf{a}},{w}_{1,\textbf{a}},m_{2,\textbf{a}},{w}_{2,\textbf{a}})$ is a minimizer of $\mathcal E_{\alpha_1,\alpha_2,\beta}$ in $\mathcal B_1\times \mathcal B_2,$ i.e.
\begin{align}\label{minimumB1B2K1K2}
e_{\alpha_1,\alpha_2,\beta}:=\inf_{(m_1,{w}_1,m_2,{w}_2)\in\mathcal K_1\times \mathcal K_2} \mathcal E_{\alpha_1,\alpha_2,\beta}(m_1,{w}_1,m_2,{w}_2)=\inf_{(m_1,{w}_1,m_2,{w}_2)\in\mathcal B_1\times \mathcal B_2}\mathcal E_{\alpha_1,\alpha_2,\beta}(m_1,{w}_1,m_2,{w}_2).
\end{align}
Now, we claim 
\begin{align}\label{tildeJ1m1w1attained}
\tilde J_{1}(m_{1,\textbf{a}},{w}_{1,\textbf{a}})=\min_{(m,{w})\in \mathcal B_1}\tilde J_1(m,w),
\end{align}
where $\tilde J_1$ is defined by (\ref{definitionJ1tilde}).
Indeed, we set 
\begin{align}\label{definitionofj1mandw}
J_1(m,{w}):=\mathcal E_{\alpha_1,\alpha_2,\beta}(m,{w},m_{2,\textbf{a}},{w}_{2,\textbf{a}}):=\varphi(m,{w})+\Lambda(m)+\tilde G,
\end{align}
where 
\begin{align*}
\varphi(m,{w}):=C_L\int_{\mathbb R^N}m\bigg|\frac{w}{m}\bigg|^{\gamma'}\,dx,
\end{align*}
\begin{align*}
\Lambda(m):=-\frac{N\alpha_1}{N+\gamma'}\int_{\mathbb R^N}m^{1+\frac{\gamma'}{N}}\,dx+\int_{\mathbb R^N}V_1m\,dx-\frac{2\beta N}{N+\gamma'}\int_{\mathbb R^N}m^{\frac{1}{2}+\frac{\gamma'}{2N}}m_{2,\textbf{a}}^{\frac{1}{2}+\frac{\gamma'}{2N}}\,dx,
\end{align*}
and
\begin{align*}
\tilde G:=C_L\int_{\mathbb R^N}\bigg|\frac{w_{2,\textbf{a}}}{m_{2,\textbf{a}}}\bigg|^{\gamma'}m_{2,\textbf{a}}\,dx-\frac{N\alpha_1}{\gamma'+N}\int_{\mathbb R^N}m_{2,\textbf{a}}^{1+\frac{\gamma'}{N}}\,dx+\int_{\mathbb R^N}V_2m_{2,\textbf{a}}\,dx.
\end{align*}
For any $(m,{w})\in \mathcal B_1$, we define
\begin{align*}
m_{\lambda}=\lambda m+(1-\lambda)m_{1,\textbf{a}},~w_{\lambda}=\lambda w+(1-\lambda)w_{1,\textbf{a}},~0<\lambda<1,
\end{align*}
and have the fact that  $(m_{\lambda},w_{\lambda})\in \mathcal B_1$.  Thus, by using (\ref{minimumB1B2K1K2}) and (\ref{definitionofj1mandw}), we obtain
\begin{align*}
J_1(m_{\lambda},{w}_{\lambda})=\mathcal E_{\alpha_1,\alpha_2,\beta}(m_{\lambda},{w}_{\lambda},m_{2,\textbf{a}},{w}_{2,\textbf{a}})\geq \mathcal E_{\alpha_1,\alpha_2,\beta}(m_{1,\textbf{a}},{w}_{1,\textbf{a}},m_{2,\textbf{a}},{w}_{2,\textbf{a}})=J_1(m_{1,\textbf{a}},{w}_{1,\textbf{a}}),\end{align*}
which implies
\begin{align*}
\varphi(m_{\lambda},{w}_\lambda)+\Lambda(m_{\lambda})\geq \varphi(m_{1,\textbf{a}},{w}_{1,\textbf{a}})+\Lambda(m_{1,\textbf{a}}),
\end{align*}
i.e.
\begin{align}\label{89innotes}
\varphi(m_{\lambda},{w}_{\lambda})-\varphi(m_{1,\textbf{a}},{w}_{1,\textbf{a}})\geq \Lambda(m_{1,\textbf{a}})-\Lambda(m_{\lambda}).
\end{align}
Next, we simplitfy (\ref{89innotes}).  On one hand, by the convexity of $\varphi$ in $(m,w)$, we have
\begin{align*}
\varphi(m_{\lambda},{w}_{\lambda})\leq \lambda\varphi(m,{w})+(1-\lambda)\varphi(m_{1,\textbf{a}},{w}_{1,\textbf{a}}),
\end{align*}
i.e.
\begin{align}\label{90innotes}
\varphi(m_{\lambda},{w}_{\lambda})-\varphi(m_{1,\textbf{a}},{w}_{1,\textbf{a}})\leq \lambda[\varphi(m,{w})-\varphi(m_{1,\textbf{a}},{w}_{1,\textbf{a}})].
\end{align}
On the other hand, for $\lambda>0$ sufficiently small, we have
\begin{align}\label{91innotes}
\Lambda(m_{\lambda})=\Lambda(m_{1,\textbf{a}})+\lambda\langle \nabla\Lambda (m_{1,\textbf{a}}),(m-m_{1,\textbf{a}})\rangle+O(\lambda).
\end{align}
In addition, invoking (\ref{definitionJ1tilde}) and (\ref{definitionofj1mandw}), one can obtain
\begin{align*}
\nabla\Lambda(m_{1,\textbf{a}})=V_1+f_1(m_{1,\textbf{a}},m_{2,\textbf{a}}).
\end{align*}
 Upon substituting (\ref{90innotes}) and (\ref{91innotes}) into (\ref{89innotes}), we get 
\begin{align*}
\varphi(m,{w})-\varphi(m_{1,\textbf{a}},{w}_{1,\textbf{a}})\geq -\langle \nabla\Lambda(m_{1,\textbf{a}}),m-m_{1,\textbf{a}}\rangle.
\end{align*}
Hence,
\begin{align*}
\tilde J_1(m,{w})=\varphi(m,{w})+\langle \nabla\Lambda (m_{1,\textbf{a}}),m\rangle\geq \varphi(m_{1,\textbf{a}},{w}_{1,\textbf{a}})+\langle \nabla\Lambda(m_{1,\textbf{a}}),m_{1,\textbf{a}}\rangle=\tilde J_1(m_{1,\textbf{a}},{w}_{1,\textbf{a}}),
\end{align*}
which indicates that claim (\ref{tildeJ1m1w1attained}) holds.

Now, we prove
\begin{align}\label{ourclaimhowtogetmulti}
\sup\{\lambda:-\Delta \psi+C_H|\nabla\psi|^{\gamma}+\lambda\leq V_1+f_1(m_{1,\textbf{a}},m_{2,\textbf{a}})\text{~in~}\mathbb R^N\text{~for some~}\psi\in\mathcal B_1\}=\min_{(m,{w})\in\mathcal B_1}\tilde J_1(m,{w}).
\end{align}
In fact, by following the similar argument shown in the proof of Proposition 3.4 in \cite{cesaroni2018concentration}, we define
\begin{align*}
\mathcal L_1(m,w,\lambda,\psi):=\tilde J_1(m,w)+\int_{\mathbb R^N}\big(m\Delta \psi+{w}\cdot\nabla \psi-\lambda m\big)\,dx+\lambda,
\end{align*}
and obtain
\begin{align*}
\min_{(m,{w})\in\mathcal B_1}\tilde J_1(m,{w})=\min_{(m,{w})\in\Gamma}\sup_{(\lambda,\psi)\in\mathbb R\times \mathcal A_1}\mathcal L_1(m,{w},\lambda,\psi),
\end{align*}
where $\Gamma:=(L^1(\mathbb R^N)\cap W^{1,\gamma'}(\mathbb R^N))\times L^{\gamma'}(\mathbb R^N)$.  Invoking the convexity of $\mathcal L_1(\cdot,\cdot,\lambda,\psi)$ and the linearity of $\mathcal L_1(m,{w},\cdot,\cdot)$, one has
\begin{align*}
&\min_{(m,{w})\in\Gamma}\sup_{(\lambda,\psi)\in\mathbb R\times \mathcal A_1}\mathcal L_1(m,{w},\lambda,\psi)=\sup_{(\lambda,\psi)\in\mathbb R\times \mathcal A_1}\min_{(m,{w})\in \Gamma}\mathcal L_1(m,{w},\lambda,\psi)\\
=&\sup_{(\lambda,\psi)\in\mathbb R\times \mathcal A_1}\int_{\mathbb R^N} \min_{(m,{w})\in\mathbb R\times \mathbb R^N}\bigg[C_L\bigg|\frac{w}{m}\bigg|^{\gamma'}m+[V_1+f_1(m_{1,\textbf{a}},m_{2,\textbf{a}})]m+m\Delta\psi+{w}\cdot\nabla\psi-\lambda m\bigg]\,dx+\lambda\\
=&\left\{\begin{array}{ll}
0,&V_1+f_1(m_{1,\textbf{a}},m_{2,\textbf{a}})-[-\Delta \psi+C_H|\nabla\psi|^{\gamma}]\geq 0,\\
-\infty,&V_1+f_1(m_{1,\textbf{a}},m_{2,\textbf{a}})-[-\Delta \psi+C_H|\nabla\psi|^{\gamma}]<0
\end{array}
\right.\\
=&\sup\{\lambda|V_1+f_1(m_{1,\textbf{a}},m_{2,\textbf{a}})-[-\Delta \psi+C_H|\nabla\psi|^{\gamma}]\geq 0\text{ for some }\psi\in\mathcal A_1\},
\end{align*}
which shows (\ref{ourclaimhowtogetmulti}).  Moreover, with the aid of Lemma \ref{lemma22preliminary}, we have 
\begin{align}\label{lambda1supmathcalAfinite}
\lambda_{1,\textbf{a}}:=&\sup\{\lambda|V_1+f_1(m_{1,\textbf{a}},m_{2,\textbf{a}})-[-\Delta \psi+C_H|\nabla\psi|^{\gamma}]\geq 0\text{ for some }\psi\in\mathcal A_1\}\nonumber\\
=&\min_{(m,{w})\in\mathcal B_1}\tilde J_1(m,{w})<+\infty,
\end{align}
and there exists $ u_{1,\textbf{a}}\in C^2(\mathbb R^N)$ such that 
\begin{align}\label{u1eqf1m1m2thm13}
-\Delta u_{1,\textbf{a}}+C_H|\nabla u_{1,\textbf{a}}|^{\gamma}+\lambda_{1,\textbf{a}}=V_1+f_1(m_{1,\textbf{a}},m_{2,\textbf{a}})\text{~in~}\mathbb R^N.
\end{align}
In particular, we have from Lemma \ref{sect2-lemma21-gradientu} and Lemma \ref{lowerboundVkgenerallemma22} that (\ref{gradientuiViloweboun})  holds for $u_{1,\textbf{a}}$.

Since $m_{1,\textbf{a}}$, $m_{2,\textbf{a}}\in L^\infty(\mathbb R^N)$ by Sobolev embedding, one obtains $f_1(m_{1,\textbf{a}},m_{2,\textbf{a}})\in L^\infty(\mathbb R^N).$  Then it follows from  (\ref{gradientuiViloweboun}) and (\ref{u1eqf1m1m2thm13}) that   
\begin{align*}
|-\Delta u_{1,\textbf{a}}(x)|\leq C(1+V_1(x)).
\end{align*}
Thus, $u_{1,\textbf{a}}\in\mathcal A_1$.  Combining (\ref{tildeJ1m1w1attained}) with (\ref{lambda1supmathcalAfinite}), one finds \eqref{lambda177innotes} holds for $i=1,$ i.e.
\begin{align*}
\lambda_{1,\textbf{a}}=\tilde J_1(m_{1,\textbf{a}},{w}_{1,\textbf{a}})=\int_{\mathbb R^N}\bigg[C_L\bigg|\frac{w_{1,\textbf{a}}}{m_{1,\textbf{a}}}\bigg|^{\gamma'}m_{1,\textbf{a}}\,+[V_1+f_1(m_{1,\textbf{a}},m_{2,\textbf{a}})]m_{1,\textbf{a}}\bigg]\,dx,
\end{align*}
where we have used (\ref{definitionJ1tilde}).   Next, we shall show
$${w}_{1,\textbf{a}}=-C_H\gamma m_{1,\textbf{a}}|\nabla u_{1,\textbf{a}}|^{\gamma-2}\nabla u_{1,\textbf{a}}.$$
First of all, (\ref{lambda177innotes}) and \eqref{u1eqf1m1m2thm13} imply that 
\begin{align*}
0=&\int_{\mathbb R^N}\bigg[C_L\bigg|\frac{{w}_{1,\textbf{a}}}{m_{1,\textbf{a}}}\bigg|^{\gamma'}+V_1+f_1(m_{1,\textbf{a}},m_{2,\textbf{a}})-\lambda_{1,\textbf{a}}\bigg]m_{1,\textbf{a}}\,dx\\
=&\int_{\mathbb R^N}\bigg[C_L\bigg|\frac{{w}_1}{m_{1,\textbf{a}}}\bigg|^{\gamma'}-\Delta u_{1,\textbf{a}}+C_H|\nabla u_{1,\textbf{a}}|^{\gamma'}\bigg]m_{1,\textbf{a}}\,dx.
\end{align*}
  Then we take $\psi=u_{1,\textbf{a}}$ in (\ref{80innotessecondeqibp})  to get
\begin{align}\label{thereforeconcludethm11multi}
0=&\int_{\{x|m_{1,\textbf{a}}>0\}}\bigg[C_L\bigg|\frac{{w}_{1,\textbf{a}}}{m_{1,\textbf{a}}}\bigg|^{\gamma'}+C_H|\nabla u_{1,\textbf{a}}|^{\gamma'}+\nabla u_{1,\textbf{a}}\cdot \frac{{w}_{1,\textbf{a}}}{m_{1,\textbf{a}}}\bigg]m_{1,\textbf{a}}\,dx.
\end{align}
By using the definition of $H$ that
\begin{align*}
L\bigg(-\frac{{w}_{1,\textbf{a}}}{m_{1,\textbf{a}}}\bigg)=C_L\bigg|\frac{{w}_{1,\textbf{a}}}{m_{1,\textbf{a}}}\bigg|^{\gamma'}=&\sup_{p\in\mathbb R^N}\bigg(-p\frac{{w}_{1,\textbf{a}}}{m_{1,\textbf{a}}}-H(p)\bigg)\geq  -C_H|\nabla u_{1,\textbf{a}}|^{\gamma}-\nabla u_{1,\textbf{a}}\cdot\frac{{w}_{1,\textbf{a}}}{m_{1,\textbf{a}}},
\end{align*}
where $H(p)=C_H|p|^{\gamma}.$  Therefore, (\ref{thereforeconcludethm11multi}) indicates that 
\begin{align}\label{indicatesthatthmmultilemmakey}
C_L\bigg|\frac{{w}_{1,\textbf{a}}}{m_{1,\textbf{a}}}\bigg|^{\gamma'}+C_H|\nabla u_{1,\textbf{a}}|^{\gamma}+\nabla u_{1,\textbf{a}}\cdot \frac{{w}_{1,\textbf{a}}}{m_{1,\textbf{a}}}\geq 0\text{~a.e.~in~}\{x\in \mathbb R^N|m_{1,\textbf{a}}>0\}.
\end{align}
Since $\sup\limits_{p\in\mathbb R^N}(-p\frac{{w}_{1,\textbf{a}}}{m_{1,\textbf{a}}}-H(p))$ is attained by $p=\nabla u_{1,\textbf{a}}$ when $m_{1,\textbf{a}}>0,$ one has from  \eqref{indicatesthatthmmultilemmakey} that 
\begin{align*}
\frac{{w}_{1,\textbf{a}}}{m_{1,\textbf{a}}}=-\nabla H(\nabla u_{1,\textbf{a}})\text{~in~}\{x\in\mathbb R^N|m_{1,\textbf{a}}>0\}.
\end{align*}
Thus, we obtain
\begin{align*}
-\Delta m_{1,\textbf{a}}-C_H\nabla\cdot (m_{1,\textbf{a}}|\nabla u_{1,\textbf{a}}|^{\gamma-2}\nabla u_{1,\textbf{a}})=0\text{~in a weak sense}.
\end{align*}
Proceeding the similar argument shown above, we have (\ref{lambda177innotes}) holds for $i=2$ and there exists $u_2\in C^2(\mathbb R^N)$ such that 
$${w}_{2,\textbf{a}}=-C_H\gamma m_{2,\textbf{a}}|\nabla u_{2,\textbf{a}}|^{\gamma-2}\nabla u_{2,\textbf{a}},~x\in\mathbb R^N\text{~in a weak sense}.$$
Finally, by the standard elliptic regularity, we find (\ref{ss11new}) holds in a classical sense.  This completes the proof of this lemma.
\end{proof}

By summarizing Lemma \ref{lemma31existenceleastenergy} and Lemma \ref{lemma32multimfg}, we are able to show conclusions stated in Theorem \ref{thm11multi}, which are

{\it{Proof of Theorem \ref{thm11multi}}}:
\begin{proof}
For Conclusion (i), we invoke Lemma \ref{lemma31existenceleastenergy} to get there exists a minimizer $(m_{1,\textbf{a}},w_{1,\textbf{a}},m_{2,\textbf{a}},m_{2,\textbf{a}})\in\mathcal K$ to (\ref{problem1p1}).  Moreover, Lemma \ref{lemma32multimfg} implies there exist $(u_{1,\textbf{a}},u_{2,\textbf{a}})\in C^2(\mathbb R^N)\times C^2(\mathbb R^N)$ and $(\lambda_{1,\textbf{a}},\lambda_{2,\textbf{a}})\in\mathbb {R}^2$ such that $(m_{1,\textbf{a}},m_{2,\textbf{a}},u_{1,\textbf{a}},u_{2,\textbf{a}},\lambda_{1,\textbf{a}},\lambda_{2,\textbf{a}})$ solves (\ref{ss1thm11}).  By standard regularity arguments, we have from Lemma \ref{lemma21-crucial-cor} that $$(m_{1,\textbf{a}},m_{2,\textbf{a}},u_{1,\textbf{a}},u_{2,\textbf{a}})\in W^{1,p}(\mathbb R^N)\times W^{1,p}(\mathbb R^N)\times C^2(\mathbb R^N)\times C^2(\mathbb R^N),$$
which completes the proof of this conclusion. 
 Conclusion (ii) is the straightforward corollary of Lemma \ref{lemma31existenceleastenergy}.
\end{proof}
We next focus on the borderline case when $\alpha_1=\alpha_2$ shown in Theorem \ref{thm11multi}.  In detail, we impose the extra assumption \eqref{moreassumptionforthm12} on the potentials and investigate the conclusions shown in Theorem \ref{thm12}, which are

{\it{Proof of Theorem \ref{thm12}:}}
\begin{proof}
In light of the assumption (\ref{moreassumptionforthm12}), we let $(m_t,w_t)$ be (\ref{byusingintro1}) with $x_0\in\mathbb R^N$ satisfying 
    \begin{align*}
    V_1(x_0)=V_2(x_0)=0.
    \end{align*}
    Then for $i=1,2,$ we compute to get
    \begin{align*}
    \int_{\mathbb R^N}V_i(x)m_t\,dx=&\frac{1}{M^*}\int_{\mathbb R^N}V_i(x)t^{N}m_0(t(x-x_0))\,dx
    =\frac{1}{M^*}\int_{\mathbb R^N}V_i\bigg(\frac{y}{t}+x_0\bigg)m_0(y)\,dy.
    \end{align*}
By invoking Lebesgue dominated theorem, we further obtain as $t\rightarrow +\infty,$ 
\begin{align*}
 \int_{\mathbb R^N}V_i(x)m_t\,dx\rightarrow V_i(x_0)=0,\text{~for~}i=1,2.
 \end{align*}
 Proceeding the similar argument shown in the proof of Lemma \ref{lemma31existenceleastenergy}, we get
\begin{align}\label{takeinftybefore}
\mathcal E_{a^*-\beta,a^*-\beta,\beta}(m_t,w_t,m_t,w_t)=[V_1(x_0)+V_2(x_0)]+o_t(1),
\end{align}
where $o_t(1)\rightarrow 0$ as $t\rightarrow+\infty.$  We take $t\rightarrow +\infty$ in (\ref{takeinftybefore}) to obtain
\begin{align}\label{criticaldeduce1}
e_{\alpha^*-\beta,\alpha^*-\beta,\beta}\leq 0.
\end{align}
On the other hand, we rewrite (\ref{energy1p3}) as 
\begin{align}\label{uponsubstitutingintobefore}
\mathcal E_{\alpha_1,\alpha_2,\beta}(m_1,w_1,m_2,w_2)=&\sum_{i=1}^2\bigg(\int_{\mathbb R^N}C_L\big|\frac{w_i}{m_i}\bigg|^{\gamma'}m_i+V_im_i-\frac{N(\alpha_i+\beta)}{N+\gamma'}\int_{\mathbb R^N}m_i^{1+\frac{\gamma'}{N}}\,dx\bigg)\nonumber\\
&+\frac{N\beta}{N+\gamma'}\int_{\mathbb R^N}\bigg(m_1^{\frac{1}{2}+\frac{\gamma'}{2N}}-m_2^{\frac{1}{2}+\frac{\gamma'}{2N}}\bigg)^2\,dx.
\end{align} 
Upon substituting $\alpha_1=\alpha_2=a^*-\beta$ and $\beta=a^*-\alpha$ into (\ref{uponsubstitutingintobefore}), we deduce that
\begin{align}\label{criticaldeduce2}
e_{a^*-\beta,a^*-\beta,\beta}\geq 0.
\end{align}
Combining (\ref{criticaldeduce1}) with (\ref{criticaldeduce2}), one has
\begin{align}\label{eastar0check}
e_{a^*-\beta,a^*-\beta,\beta}=0.
\end{align}

Now, we argue by contradiction and assume that  $(m_1,w_1,m_2,w_2)$ is a minimizer of (\ref{problem1p1}) with $\alpha_1=\alpha_2=a^*-\beta$ and $\beta=a^*-\alpha$.  Then we have 
\begin{align}\label{findsfromfirstone}
\mathcal E_{a^*-\beta,a^*-\beta,\beta}(m_1,w_1,m_2,w_2)=&\sum_{i=1}^2C_L\int_{\mathbb R^N}\bigg|\frac{w_i}{m_i}\bigg|^{\gamma'}m_i\,dx-\frac{Na^*}{N+\gamma'}\int_{\mathbb R^N}m_i^{1+\frac{\gamma'}{N}}\,dx\nonumber\\
&+\frac{N\beta}{N+\gamma'}\int_{\mathbb R^N}\bigg(m_1^{\frac{1}{2}+\frac{\gamma'}{2N}}-m_2^{\frac{1}{2}+\frac{\gamma'}{2N}}\bigg)^2\,dx\nonumber\\
&+\int_{\mathbb R^N}V_1(x)m_1+V_2(x)m_2\,dx\nonumber\\
:=&I_1+I_2+I_3.
\end{align}
In light of (\ref{eastar0check}), one finds from (\ref{findsfromfirstone}) that $I_1=I_2=I_3=0,$ in which $I_1=0$ implies each $(m_i,w_i)$, $i=1,2$ is a minimizer of problem (\ref{GNinequalitybest}).  In addition, $I_2=0$ indicates that $m_1=m_2$ in $\mathbb R^N.$  Morever, one gets from $I_3=0$ that 
\begin{align*}
\int_{\mathbb R^N}V_1(x)m_1+V_2(x)m_2\,dx=0,
\end{align*}
which leads to a contradiction since $m_i>0$ for $i=1,2$ by using the compactly Sobolev embedding and the maximum principle as shown in \cite{bernardini2023ergodic}.
\end{proof}

For the existence of minimizers, we next consider the case of $\alpha_1=a^*$ and show Theorem \ref{thmfinalexistence}, which is

{\it{Proof of Theorem \ref{thmfinalexistence}:}}

\begin{proof}
We define the test solution-pair as  
\begin{align}\label{scalingtestthm14}
m_{i,\tau}(x)=\frac{\tau^N}{M^*}m_0\bigg(\tau\bigg(x-\bar x_i+(-1)^{i}\iota\frac{\ln \tau}{\tau}\nu\bigg)\bigg),~w_{i,\tau}=\frac{\tau^{N+1}}{M^*}w_0\bigg(\tau\bigg(x-\bar x_i+(-1)^{i}\iota\frac{\ln \tau}{\tau}\nu\bigg)\bigg),
\end{align}
where $(m_0,w_0)$ denotes a minimizer of (\ref{GNinequalitybest}) satisfying \eqref{equmpotentialfree}, $\nu\in \mathbb S^{N-1}$, $\bar x_i\in\mathbb R^N$ and constant $\iota$ will be determined later. 

By using Lemma \ref{poholemma}, we have 
\begin{align*}
C_L\int_{\mathbb R^N}\bigg|\frac{w_{i,\tau}}{m_{i,\tau}}\bigg|^{\gamma'}m_{i,\tau}\,dx=\frac{\tau^{\gamma'}}{M^*},~~\int_{\mathbb R^N}m_{i,\tau}^{1+\frac{\gamma'}{N}}\,dx=\frac{N+\gamma'}{N}\frac{\tau^{\gamma'}}{(M^*)^{1+\frac{\gamma'}{N}}},
\end{align*}
and
\begin{align}\label{B6innotesmfgmulti}
\int_{\mathbb R^N}m_{1,\tau}^{\frac{1}{2}+\frac{\gamma'}{2N}}m_{2,\tau}^{\frac{1}{2}+\frac{\gamma'}{2N}}\,dx=\frac{\tau^{\gamma'}}{(M^*)^{1+\frac{\gamma'}{N}}}\int_{\mathbb R^N}m_0^{\frac{1}{2}+\frac{\gamma'}{2N}}(x)m_0^{\frac{1}{2}+\frac{\gamma'}{2N}}(x+\tau(\bar x_1-\bar x_2)+2\iota\ln \tau\nu)\,dx.
\end{align}
We have the fact that 
\begin{align*}
|\tau(\bar x_1-\bar x_2)+2\iota \ln\tau\nu|\geq 2\iota \ln \tau ~~\text{when }\tau\gg 1.
\end{align*}
Hence, for $\tau$ large, if $x\in B_{\iota\ln\tau}=\{x||x|<\iota \ln\tau\}$, one gets from Lemma \ref{mdecaylemma} that 
\begin{align}\label{byusingonegets1}
m_0(x+\tau(\bar x_1-\bar x_2)+2\iota\ln \tau\nu)\leq Ce^{-\delta_0 \iota\ln\tau},
\end{align}
where $C>0$ and $\delta_0>0$ are some constants.  And if $x\in B_{\iota\ln\tau}^c$, then 
\begin{align}\label{exponentialdecayusing}
m_0(x)\leq Ce^{-\delta_0\iota\ln \tau}.
\end{align}
Combining (\ref{byusingonegets1}) and (\ref{exponentialdecayusing}), one finds from (\ref{B6innotesmfgmulti}) that as $\tau\rightarrow+\infty$,
\begin{align}\label{B8innotesmultimfg}
    \int_{\mathbb R^N}m_{1,\tau}^{\frac{1}{2}+\frac{\gamma'}{2N}}m_{2,\tau}^{\frac{1}{2}+\frac{\gamma'}{2N}}\,dx=\frac{\tau^{\gamma'}}{(M^*)^{1+\frac{\gamma'}{N}}}\bigg[&\int_{B_{\iota\ln\tau}}m_0^{\frac{1}{2}+\frac{\gamma'}{2N}}(x)m_0^{\frac{1}{2}+\frac{\gamma'}{2N}}(x+\tau(\bar x_1-\bar x_2)+2\iota\ln \tau\nu)\,dx\nonumber\\
    &+\int_{B_{\iota\ln\tau}^c}m_0^{\frac{1}{2}+\frac{\gamma'}{2N}}(x)m_0^{\frac{1}{2}+\frac{\gamma'}{2N}}(x+\tau(\bar x_1-\bar x_2)+2\iota\ln \tau\nu)\,dx\bigg]\nonumber\\
    \leq &C_{N,\gamma'}\tau^{\gamma'}e^{-(\frac{1}{2}+\frac{\gamma'}{2N})\delta_0\iota\ln\tau}=C_{N,\gamma'}\tau^{\gamma'-\big(\frac{1}{2}+\frac{\gamma'}{2N}\big)\delta_0\iota}\rightarrow 0,
\end{align}
where constant $\iota$ is chosen as $\iota>\frac{2\gamma'N}{(N+\gamma')\delta_0}.$  In addition,
\begin{align*}
\int_{\mathbb R^N}V_i(x)m_{i,\tau}(x)\,dx=\frac{1}{M^*}\int_{\mathbb R^N}V_i\bigg(\frac{x}{\tau}+\bar x_{i}-(-1)^{i}\iota\frac{\ln \tau}{\tau}\nu\bigg)m_0\,dx:=\frac{1}{M^*}\int_{\mathbb R^N}g_{\tau}(x)\,dx.
\end{align*}
Noting that $g_{\tau}(x)\rightarrow V_i(\bar x_{i})m_0(x)$ a.e. in $\mathbb R^N$, we obtain from (\ref{Vicondition2}) and Lemma \ref{mdecaylemma} that
 when $\tau$ is large,
\begin{align*}
|g_{\tau}(x)|\leq Ce^{\delta  |\frac{x}{\tau}+\bar x_i-(-1)^i\iota\frac{\ln \tau}{\tau}\nu|}e^{-\delta_0|x|}\leq Ce^{-\frac{\delta_0}{2}|x|}\in L^1(\mathbb R^N).
\end{align*}
Thus, by Lebesgue dominated theorem, we further get
\begin{align}\label{b10innotesmultimfg}
\int_{\mathbb R^N}V_im_{i,\tau}\,dx\rightarrow V_i(\bar x_i)\text{~as~}\tau\rightarrow +\infty.
\end{align}
Collecting (\ref{B6innotesmfgmulti}), (\ref{B8innotesmultimfg}) and (\ref{b10innotesmultimfg}), one finds if $\alpha_1=\alpha_2=a^*$ and $\beta\leq 0$, then
\begin{align*}
\mathcal E_{\alpha_1,\alpha_2,\beta}(m_{1,\tau},w_{1,\tau},m_{2,\tau},w_{2,\tau})=V_1(\bar x_1)+V_2(\bar x_2)+o_{\tau}(1),
\end{align*}
where $o_{\tau}(1)\rightarrow 0$ as $\tau\rightarrow +\infty.$
Choose $\bar x_i$ such that $V_i(\bar x_i)=0$ for $i=1,2,$, it then follows that 
\begin{align*}
e_{\alpha_1,\alpha_2,\beta}\leq V_1(\bar x_1)+V_2(\bar x_2)=0.
\end{align*}
Moreover, by using (\ref{GNinequalityused}) and $\beta\leq 0$, one has  $e_{\alpha_1,\alpha_2,\beta}\geq 0.$  Therefore, we summarize to get $e_{\alpha_1,\alpha_2,\beta}=0.$  Proceeding the same argument as shown in the proof of Theorem \ref{thm12}, we show there is no minimizer in case (i). 

For case (ii), if $\beta=0,$ one finds
\begin{align*}
e_{\alpha_1,\alpha_2,0}=e^1_{a^*}+e^2_{\alpha_2},
\end{align*}
where $e^1_{a^*}$ and $e^2_{\alpha_2}$ are given by (\ref{problem51innotescopy}).  Noting that this is the decoupled case, we have the fact that there is no minimizer as shown in \cite{cirant2024critical}.

If $0<\beta\leq\frac{a^*-\alpha_2}{2}$, taking $\iota=0$ in (\ref{scalingtestthm14}), we compute to get 
\begin{align}\label{b12innotesmultimfg}
\int_{\mathbb R^N}m_{1,\tau}^{\frac{1}{2}+\frac{\gamma'}{2N}}m_0^{\frac{1}{2}+\frac{\gamma'}{2N}}\,dx=\frac{\tau^{\frac{1}{2}(\gamma'-N)}}{(M^*)^{\frac{1}{2}+\frac{\gamma'}{2N}}}\int_{\mathbb R^N}m_0\bigg(\frac{x}{\tau}+\bar x_1\bigg)m_0^{\frac{1}{2}+\frac{\gamma'}{2N}}\,dx:=\frac{\tau^{\frac{1}{2}(\gamma'-N)}}{(M^*)^{\frac{1}{2}+\frac{\gamma'}{2N}}}I_{\tau}.
\end{align}
We choose $\bar x_1\in\mathbb R^N$ such that $m_0(\bar x_1)>C_0>0$ then obtain
\begin{align*}
\lim_{\tau\rightarrow +\infty}I_{\tau}\geq C_0\int_{\mathbb R^N}m_0^{\frac{1}{2}+\frac{\gamma'}{2N}}\,dx\geq C_1>0\text{ as }\tau\rightarrow+\infty.
\end{align*}
   Thus, (\ref{b12innotesmultimfg}) implies
\begin{align*}
\int_{\mathbb R^N}m_{1,\tau}^{\frac{1}{2}+\frac{\gamma'}{2N}}m_0^{\frac{1}{2}+\frac{\gamma'}{2N}}\,dx\geq C_{1}\tau^{\frac{1}{2}(\gamma'-N)}\rightarrow+\infty,
\end{align*}
 It follows that 
\begin{align*}
\mathcal E_{a^*.\alpha_2,\beta}(m_{1,\tau},w_{1,\tau},m_0,w_0)\leq o_{\tau}(1)+C-C_{\gamma'}\beta\tau^{\frac{1}{2}(\gamma'-N)}\rightarrow -\infty\text{ for }\beta>0.
\end{align*}
Hence $e_{a^*,\alpha_2,\beta}=-\infty$ if $\beta>0$, which indicates (\ref{problem1p1}) has no minimizer.


\end{proof}
As shown in Theorem \ref{thm11multi} and Theorem \ref{thm12}, we have obtained when all coefficients $\alpha_1,$ $\alpha_2$ and $\beta$ are subcritical, (\ref{ss1}) admits classical ground states; whereas, if $\alpha_1=\alpha_2$ are subcritical and $\beta$ is critical, then (\ref{problem1p1}) has no minimizer.  A natural question is the behaviors of ground states as $(\alpha_1,\alpha_2)\nearrow (a^*-\beta,a^*-\beta)$.  In fact, we can show there are concentration phenomena as coefficients approach critical ones.     In the next section, we shall discuss the asymptotic profiles of ground states in the singular limits mentioned above.

\section{Asymptotic Profiles of Ground States with $\beta>0$}\label{sect4multipopulation}
This section is devoted to the blow-up behaviors of ground states to (\ref{ss1}) in some singular limits under the attractive interaction case. We proceed the proof of Theorem \ref{thm13attractive} as follows.

{\it{Proof of Theorem \ref{thm13attractive}:}}
\begin{proof}
First of all, we have from (\ref{uponsubstitutingintobefore}) that 
\begin{align}\label{thm13scalinglimitbestconstant}
\mathcal E_{\alpha_1,\alpha_2,\beta}(m_{1,\textbf{a}},w_{1,\textbf{a}},m_{2,\textbf{a}},w_{2,\textbf{a}})=&\sum_{i=1}^2\bigg(\int_{\mathbb R^N}C_L\big|\frac{w_{i,\textbf{a}}}{m_{i,\textbf{a}}}\bigg|^{\gamma'}m_{i,\textbf{a}}-\frac{N(\alpha_i+\beta)}{N+\gamma'}\int_{\mathbb R^N}m_{i,\textbf{a}}^{1+\frac{\gamma'}{N}}\,dx\bigg)\nonumber\\
&+\frac{N\beta}{N+\gamma'}\int_{\mathbb R^N}\bigg(m_{1,\textbf{a}}^{\frac{1}{2}+\frac{\gamma'}{2N}}-m_{2,\textbf{a}}^{\frac{1}{2}+\frac{\gamma'}{2N}}\bigg)^2\,dx\nonumber\\
&+\int_{\mathbb R^N}V_1(x)m_{1,\textbf{a}}+V_2(x)m_{2,\textbf{a}}\,dx\nonumber\\
:=&II_1+II_2+II_3.
\end{align}
In light of (\ref{Vicondition1}) and (\ref{GNinequalityused}), one finds $II_j\geq 0$, $j=1,2,3.$  Moreover, assumption (\ref{Vicondition1}) implies $II_3\geq 0.$  Proceeding the same argument shown in the proof of Lemma \ref{lemma31existenceleastenergy}, we use the test pair \eqref{byusingintro1}  and compute from (\ref{thm13scalinglimitbestconstant}) that
\begin{align}\label{thm13scalinglimitbestconstant1}
\lim_{\textbf{a}\nearrow \textbf{a}^*_{\beta}}e_{\alpha_1,\alpha_2,\beta}=e_{a^*-\beta,a^*-\beta,\beta}=0.
\end{align}
Combining (\ref{thm13scalinglimitbestconstant}) with (\ref{thm13scalinglimitbestconstant1}), we obtain (\ref{thm13conclusion1}) and (\ref{thm13conclusion2}).

We next prove \eqref{202401719conclu} and argue by contradiction.  Without loss of generality, we assume that 
\begin{align*}
\limsup_{\textbf{a}\nearrow \textbf{a}^*_{\beta}}C_L\int_{\mathbb R^N}\bigg|\frac{w_{1,\textbf{a}}}{m_{1,\textbf{a}}}\bigg|^{\gamma'}m_{1,\textbf{a}}\,dx<+\infty.
\end{align*}
  Then, it follows from (\ref{thm13conclusion1}), (\ref{thm13conclusion2}) and Lemma \ref{lemma21-crucial-cor} that $(m_{1,\textbf{a}},w_{1,\textbf{a}},m_{2,\textbf{a}},w_{2,\textbf{a}})$ is uniformly bounded in $(W^{1,\gamma'}(\mathbb R^N)\times L^{\gamma'}(\mathbb R^N))^2.$  Moreover, by compactly Sobolev embedding (C.f. Lemma 5.1 in \cite{cirant2024critical}), one finds $m_{i,\textbf{a}}\rightarrow m_{i,0}$ strongly in $L^1(\mathbb R^N)\cap L^{1+\frac{\gamma'}{N}}(\mathbb R^N)$ for $i=1,2.$  By using the convexity of $\int_{\mathbb R^N}\big|\frac{w}{m}\big|^{\gamma'}m\,dx$, we have
\begin{align*}
e_{a^*-\beta,a^*-\beta,\beta}=\lim_{\textbf{a}\nearrow \textbf{a}_{\beta}^*}e_{\alpha_1,\alpha_2,\beta}=&\lim_{\textbf{a}\nearrow \textbf{a}^*_{\beta}}\mathcal E_{\alpha_1,\alpha_2,\beta}(m_{1,\textbf{a}},w_{1,\textbf{a}},m_{2,\textbf{a}},w_{2,\textbf{a}})\\
\geq& \mathcal E_{a^*-\beta,a^*-\beta,\beta} (m_{1,0},w_{1,0},m_{2,0},w_{2,0})\geq e_{a^*-\beta,a^*-\beta,\beta},
\end{align*}
which implies $(m_{1,0},w_{1,0},m_{2,0},w_{2,0})$ is a minimizer of $e_{a^*-\beta,a^*-\beta,\beta}$ and it is a contradiction since we have showed that $e_{a^*-\beta,a^*-\beta,\beta}$ has no minimizer in Theorem \ref{thm12}.

Now, we find \eqref{202401719conclu} holds and further obtain from (\ref{thm13conclusion1}) that for $i=1,2$ 
\begin{align*}
\lim_{\textbf{a}\nearrow \textbf{a}_{\beta}^*}\int_{\mathbb R^N}m_{i,\textbf{a}}^{1+\frac{\gamma'}{N}}\,dx= +\infty\text{ and }~\lim_{\textbf{a}\nearrow \textbf{a}_{\beta}^*}\frac{C_L\int_{\mathbb R^N}\big|\frac{w_{i,\textbf{a}}}{m_{i,\textbf{a}}}\big|^{\gamma'}m_{i,\textbf{a}}\,dx}{\int_{\mathbb R^N}m_{i,\textbf{a}}^{1+\frac{\gamma'}{N}}\,dx}=\frac{N+\gamma'}{N}a^*.
\end{align*}
Noting that as $\textbf{a}
\nearrow \textbf{a}_{\beta}^*,$ 
\begin{align*}
\bigg[\bigg(\int_{\mathbb R^N}m_{1,\textbf{a}}^{1+\frac{\gamma'}{N}}\,dx\bigg)^{\frac{1}{2}}-\bigg(\int_{\mathbb R^N}m_{2,\textbf{a}}^{1+\frac{\gamma'}{N}}\,dx\bigg)^{\frac{1}{2}}\bigg]^2\leq \int_{\mathbb R^N}\bigg(m_{1,\textbf{a}}^{\frac{1}{2}+\frac{\gamma'}{2N}}-m_{2,\textbf{a}}^{\frac{1}{2}+\frac{\gamma'}{2N}}\bigg)^2\,dx\rightarrow 0,
\end{align*}
one gets (\ref{thm13conclusionpart23}) holds. 

Noting that $(m_{1,\textbf{a}},{{w}}_{1,\textbf{a}},m_{2,\textbf{a}},{{w}}_{2,\textbf{a}})$ satisfy (\ref{ss11new}), we have from the integration by parts that 
\begin{align}
&\int_{\mathbb R^N}\nabla u_{1,\textbf{a}}\cdot\nabla m_{1,\textbf{a}}\,dx+C_H\int_{\mathbb R^N}|\nabla u_{1,\textbf{a}}|^{\gamma}m_{1,\textbf{a}}\,dx+\lambda_1\nonumber\\
=&\int_{\mathbb R^N}V_1m_{1,\textbf{a}}\,dx-\alpha_1\int_{\mathbb R^N}m_{1,\textbf{a}}^{1+\frac{\gamma'}{N}}\,dx-\beta\int_{\mathbb R^N}m_{1,\textbf{a}}^{\frac{1}{2}+\frac{\gamma'}{2N}}m_{2,\textbf{a}}^{\frac{1}{2}+\frac{\gamma'}{2N}}\,dx,\label{ibp120240719}
\end{align}
and
\begin{align}\label{ibp2202407192}
\int_{\mathbb R^N}\nabla u_{1,\textbf{a}}\cdot\nabla m_{1,\textbf{a}}\,dx=-C_H\gamma\int_{\mathbb R^N}m_{1,\textbf{a}}|\nabla u_{1,\textbf{a}}|^{\gamma}\,dx.
\end{align}
Combining (\ref{ibp120240719}) with (\ref{ibp2202407192}), one finds
\begin{align}\label{lambda1formulalimit39}
\lambda_1=&C_L\int_{\mathbb R^N}\bigg|\frac{w_{1,\textbf{a}}}{m_{1,\textbf{a}}}\bigg|^{\gamma'}m_{1,\textbf{a}}\,dx+\int_{\mathbb R^N}V_1m_{1,\textbf{a}}\,dx-\alpha_1\int_{\mathbb R^N}m_{1,\textbf{a}}^{1+\frac{\gamma'}{N}}\,dx-\beta\int_{\mathbb R^N}m_{1,\textbf{a}}^{\frac{1}{2}+\frac{\gamma'}{2N}}m_{2,\textbf{a}}^{\frac{1}{2}+\frac{\gamma'}{2N}}\,dx\nonumber\\
=&\bigg(C_L\int_{\mathbb R^N}\bigg|\frac{w_{1,\textbf{a}}}{m_{1,\textbf{a}}}\bigg|^{\gamma'}m_{1,\textbf{a}}\,dx-\frac{N(\alpha_1+\beta)}{N+\gamma'}\int_{\mathbb R^N}m_{1,\textbf{a}}^{1+\frac{\gamma'}{N}}\,dx\bigg)-\frac{\gamma'(\alpha_1+\beta)}{\gamma'+N}\int_{\mathbb R^N}m_{1,\textbf{a}}^{1+\frac{\gamma'}{N}}\,dx\nonumber\\
&+\beta\int_{\mathbb R^N}m_{1,\textbf{a}}^{\frac{1}{2}+\frac{\gamma'}{2N}}\bigg(m_{1,\textbf{a}}^{\frac{1}{2}+\frac{\gamma'}{2N}}-m_{2,\textbf{a}}^{\frac{1}{2}+\frac{\gamma'}{2N}}\bigg)\,dx+\int_{\mathbb R^N}V_1m_{1,\textbf{a}}\,dx\nonumber\\
=&o_{\varepsilon}(1)-\frac{\gamma'(\alpha_1+\beta)}{\gamma'+N}\int_{\mathbb R^N}m_{1,\textbf{a}}^{1+\frac{\gamma'}{N}}\,dx+\beta\int_{\mathbb R^N}m_{1,\textbf{a}}^{\frac{1}{2}+\frac{\gamma'}{2N}}\bigg(m_{1,\textbf{a}}^{\frac{1}{2}+\frac{\gamma'}{2N}}-m_{2,\textbf{a}}^{\frac{1}{2}+\frac{\gamma'}{2N}}\bigg)\,dx,
\end{align}
where we have used \eqref{thm13conclusion1} and \eqref{thm13conclusion2} as $\textbf{a}\nearrow \textbf{a}_{\beta}^*.$  To further simplify (\ref{lambda1formulalimit39}), we use (\ref{thm13conclusion2}) to get
\begin{align}\label{lambda1formulalimit40}
\bigg|\int_{\mathbb R^N}m_{1,\textbf{a}}^{\frac{1}{2}+\frac{\gamma'}{2N}}\bigg(m_{1,\textbf{a}}^{\frac{1}{2}+\frac{\gamma'}{2N}}-m_{2,\textbf{a}}^{\frac{1}{2}+\frac{\gamma'}{2N}}\bigg)\,dx\bigg|\leq &\bigg(\int_{\mathbb R^N}m_{1,\textbf{a}}^{1+\frac{\gamma'}{N}}\,dx\bigg)^{\frac{1}{2}}\bigg[\int_{\mathbb R^N}\bigg(m_{1,\textbf{a}}^{\frac{1}{2}+\frac{\gamma'}{2N}}-m_{2,\textbf{a}}^{\frac{1}{2}+\frac{\gamma'}{2N}}\bigg)^2\,dx\bigg]^{\frac{1}{2}}\nonumber\\
=&o_{\varepsilon}(1)\bigg(\int_{\mathbb R^N}m_{1,\textbf{a}}^{1+\frac{\gamma'}{N}}\,dx\bigg)^{\frac{1}{2}}.
\end{align}
By utilizing (\ref{thm13conclusion1}) and \eqref{defvarepsilon316}, one finds 
\begin{align}\label{lambda1formulalimit41}
\lim_{\textbf{a}\nearrow \textbf{a}_{\beta}^*}{\frac{N\varepsilon_{\textbf{a}}^{\gamma'}(\alpha_1+\beta)}{N+\gamma'}\int_{\mathbb R^N}m_{1,\textbf{a}}^{1+\frac{\gamma'}{N}}\,dx}=1.
\end{align}
Collecting (\ref{lambda1formulalimit39}), (\ref{lambda1formulalimit40}) and (\ref{lambda1formulalimit41}), we have 
\begin{align}\label{lambda1inthm13asymptotic}
\lambda_1=-\frac{(\alpha_1+\beta)\gamma'}{N+\gamma'}\int_{\mathbb R^N}m_{1,\textbf{a}}^{1+\frac{\gamma'}{N}}\,dx+o_{\varepsilon}(1)\bigg(\int_{\mathbb R^N}m_{1,\textbf{a}}^{1+\frac{\gamma'}{N}}\,dx\bigg)^{\frac{1}{2}}+o_{\varepsilon}(1)= \frac{\gamma'}{N}\varepsilon^{-\gamma'}+o_{\varepsilon}(1)(1+\varepsilon^{-\frac{\gamma'}{2}}),
\end{align}
where $\varepsilon\rightarrow 0$ given by (\ref{defvarepsilon316}). This implies that $\lim_{\varepsilon\rightarrow 0}\lambda_{1}\varepsilon^{\gamma'}=-\frac{\gamma'}{N}$.  Proceeding the similar argument shown above, one obtains from (\ref{thm13conclusionpart23}) that 
$
\lim_{\varepsilon\rightarrow 0}\lambda_{2}\varepsilon^{\gamma'}=-\frac{\gamma'}{N}
$.
Now, we substitute (\ref{scalingthm31profile}) into  (\ref{ss11new}) and obtain
\begin{align}\label{ss11newnew}
\left\{\begin{array}{ll}
-\Delta u_{1,\varepsilon}+C_H|\nabla u_{1,\varepsilon}|^{\gamma}+\lambda_1\varepsilon^{\gamma'}=\varepsilon^{\gamma'}V_1(\varepsilon x+x_{1,\varepsilon})-\alpha_1 m_{1,\varepsilon}^{\frac{\gamma'}{N}}-\beta m_{1,\varepsilon}^{\frac{\gamma'}{2N}-\frac{1}{2}}m_{2,\varepsilon}^{\frac{1}{2}+\frac{\gamma'}{N}},&x\in\mathbb R^N,\\
\Delta m_{1,\varepsilon}+\nabla\cdot {{w}}_{1,\varepsilon}=0,~{{w}}_{1,\varepsilon}=-\gamma C_Hm_{1,\varepsilon}|\nabla u_{1,\varepsilon}|^{\gamma-2}\nabla u_{1,\varepsilon}&x\in\mathbb R^N,\\
-\Delta u_{2,\varepsilon}+C_H|\nabla u_{2,\varepsilon}|^{\gamma}+\lambda_2\varepsilon^{\gamma'}=\varepsilon^{\gamma'}V_2(\varepsilon x+x_{1,\varepsilon})-\alpha_2 m_{2,\varepsilon}^{\frac{\gamma'}{N}}-\beta m_{2,\varepsilon}^{\frac{\gamma'}{2N}-\frac{1}{2}}m_{1,\varepsilon}^{\frac{1}{2}+\frac{\gamma'}{N}},&x\in\mathbb R^N,\\
\Delta m_{2,\varepsilon}+\nabla\cdot {{w}}_{2,\varepsilon}=0,~{{w}}_{2,\varepsilon}=-\gamma C_Hm_{2,\varepsilon}|\nabla u_{2,\varepsilon}|^{\gamma-2}\nabla u_{2,\varepsilon},&x\in\mathbb R^N,\\
\int_{\mathbb R^N}m_{1,\varepsilon}\,dx=\int_{\mathbb R^N}m_{2,\varepsilon}\,dx=1.
\end{array}
\right.
\end{align}
Without loss of the generality, we assume
\begin{align*}
\inf_{x\in\mathbb R^N}u_{1,\textbf{a}}(x)=\inf_{x\in\mathbb R^N}u_{2,\textbf{a}}(x)=0.
\end{align*}
In light of (\ref{thm13conclusionpart23}),  (\ref{defvarepsilon316}) and \eqref{scalingthm31profile}, one finds
\begin{align*}
\sup_{\varepsilon\rightarrow 0^+}C_L\int_{\mathbb R^N}\bigg|\frac{w_{i,\varepsilon}}{m_{i,\varepsilon}}\bigg|^{\gamma'}m_{i,\varepsilon}\,dx<+\infty,~i=1,2.
\end{align*}
Then it follows from Lemma \ref{lemma21-crucial-cor} that for $i=1,2$
\begin{align}\label{fromregularitythm31}
\sup_{\varepsilon\rightarrow 0^+}\Vert m_{i,\varepsilon}\Vert_{W^{1,\gamma'}(\mathbb R^N)}<+\infty,~~\sup_{\varepsilon\rightarrow 0^+}\Vert w_{i,\varepsilon}\Vert_{L^1(\mathbb R^N)}<+\infty,~~\sup_{\varepsilon\rightarrow 0^+}\Vert w_{i,\varepsilon}\Vert_{L^{\gamma'}(\mathbb R^N)}<+\infty.
\end{align}
Invoking (\ref{thm13conclusion2}) and (\ref{defvarepsilon316}) , one finds for $i=1,2$, 
\begin{align}\label{convergencem1m2potentialvarepsilonthm13}
\lim_{\varepsilon\rightarrow 0^+}\int_{\mathbb R^N}V_i(\varepsilon x+x_{1,\varepsilon})m_{i,\varepsilon}(x)\,dx=0,
\end{align}
and 
\begin{align}\label{convergencem1m2varepsilonthm13}
\lim_{\varepsilon\rightarrow 0^+}\int_{\mathbb R^N}\bigg(m_{1,\varepsilon}^{\frac{1}{2}+\frac{\gamma'}{2N}}-m_{2,\varepsilon}^{\frac{1}{2}+\frac{\gamma'}{2N}}\bigg)^2\,dx= 0.
\end{align}
By using the standard Sobolev embedding, we have from (\ref{fromregularitythm31}) and (\ref{convergencem1m2varepsilonthm13}) that  
\begin{align}\label{convergencem1m2varepsilonthm131}
m_{i,\varepsilon}\rightharpoonup m~\text{in~}W^{1,\gamma'}(\mathbb R^N),~~m_{i,\varepsilon}\rightarrow m\geq 0,\text{~a.e.~in~}\mathbb R^N.
\end{align}
Moreover, by using the Morrey's embedding $W^{1,\gamma'}(\mathbb R^N)\hookrightarrow C^{0,\theta}(\mathbb R^N)$ with $\theta\in (0,1-\frac{\gamma'}{N})$, one finds 
\begin{align}\label{convergencethm13holdercontinuousnew}
m_{i,\varepsilon}\rightarrow m\text{~in~}C^{0,\theta}(\mathbb R^N),~~\sup_{\varepsilon\rightarrow 0^+}\Vert m_{i,\varepsilon}\Vert_{C^{0,\theta}(\mathbb R^N)}<+\infty,~~i=1,2.
\end{align}
Recall that $u_{1,\textbf{a}}(x_{1,\varepsilon})=\inf_{x\in\mathbb R^N}u_{1,\textbf{a}}(x)=0,$ then we have $u_{1,\varepsilon}(0)=\inf_{x\in\mathbb R^N}u_{1,\varepsilon}$. Moreover, by applying the maximum principle, one gets from the first equation of (\ref{ss11newnew})  and \eqref{convergencethm13holdercontinuousnew} that
\begin{align*}
\lambda_{1}\varepsilon^{\gamma'}\geq &-\alpha_1m_{1,\varepsilon}^{\frac{\gamma'}{N}}(0)-\beta m_{2,\varepsilon}^{\frac{1}{2}+\frac{\gamma'}{2N}}(0) m^{\frac{\gamma'}{2N}-\frac{1}{2}}_{1,\varepsilon}(0)\\
=&-(\alpha_1+\beta)m_{1,\varepsilon}^{\frac{\gamma'}{N}}(0)+o_{\varepsilon}(1),
\end{align*}
 In addition, noting (\ref{lambda1inthm13asymptotic}) and $\alpha_1+\beta\nearrow a^*$, we have
\begin{align*}
\lim_{\varepsilon\rightarrow 0}m_{1,\varepsilon}^{\frac{\gamma'}{N}}(0)\geq \frac{\gamma'}{Na^*}.
\end{align*}
Recall again $W^{1,\gamma'}(\mathbb R^N)\hookrightarrow C^{0,\theta}(\mathbb R^N),$ we have from (\ref{convergencem1m2varepsilonthm13}) that for $i=1,2,$ there exists $R_0>0$ and $C>0$ such that 
\begin{align}\label{mivarepsilonlowboundholdercontinuity}
m_{i,\varepsilon}(x)\geq C>0,\text{~~}\forall |x|<R_0,\text{ for }i=1,2.
\end{align} 
Moreover, we utilize (\ref{convergencem1m2potentialvarepsilonthm13}) and (\ref{mivarepsilonlowboundholdercontinuity}) to get up to a subsequence,
\begin{align*}
{\lim_{\varepsilon\rightarrow 0}x_{1,\varepsilon}=x_0,\text{~s.t.~}V_1(x_0)=0=V_2(x_0).}
\end{align*}
Combining (\ref{convergencethm13holdercontinuousnew}) with (\ref{mivarepsilonlowboundholdercontinuity}), one also has 
\begin{align}\label{mlowerboundpositivelocallimit}
m(x)\geq C>0,\text{~}\forall |x|<R_0.
\end{align}
Next, we study the regularity of the value function $u$.  To this end, we rewrite the $u_1$-equation in (\ref{ss11newnew}) as 
\begin{align}\label{rewriteueqthm13c2regularity}
-\Delta u_{1,\varepsilon}+C_H|\nabla u_{1,\varepsilon}|^{\gamma}=&-\lambda_1\varepsilon^{\gamma'}+\varepsilon^{\gamma'}V_1(\varepsilon x+x_{1,\varepsilon})-\alpha_1 m_{1,\varepsilon}^{\frac{\gamma'}{N}}-\beta m_{1,\varepsilon}^{\frac{\gamma'}{2N}-\frac{1}{2}}m_{2,\varepsilon}^{\frac{1}{2}+\frac{\gamma'}{2N}}\nonumber\\
:=&g_{\varepsilon}(x)\in L^{\infty}_{\text{loc}}(\mathbb R^N)\cap C_{\text{loc}}^{0,\theta}(\mathbb R^N).
\end{align}
For $ R>0$ large enough, we have 
\begin{align*}
\Vert g_{\varepsilon}(x)\Vert_{L^\infty(B_R(0))}<C_R<+\infty,\text{~}\forall |x|<2R,
\end{align*}
where $C_R>0$ is independent of $\varepsilon.$   Then it follows from \eqref{rewriteueqthm13c2regularity} and Sobolev embedding that 
\begin{align*}
|\nabla u_{1,\varepsilon}(x)|\leq C_R,~\forall |x|<2R.
\end{align*}
Since $u_{1,\varepsilon}(0)=0$, we further have
\begin{align*}
|u_{1,\varepsilon}(x)|\leq C_R,\text{~}\forall |x|<2R.
\end{align*}
By using the $W^{2,p}-$ estimate, one gets
\begin{align*}
\Vert u_{1,\varepsilon}\Vert_{W^{2,p}(B_{R+1}(0))}\leq C_{p,R}\big(\Vert u_{1,\varepsilon}\Vert_{L^p(B_{2R}(0))}+\Vert g_{\varepsilon}\Vert_{L^p(B_{2R}(0))}+\Vert|\nabla u_{1,\varepsilon}|^{\gamma}\Vert_{L^p(B_{2R}(0))}\big),~\forall p>1,
\end{align*}
where $C_{p,R}>0$ is a constant depending on $p$ and $R.$  Let $p>N,$ then we obtain
$$\Vert u_{1,\varepsilon}\Vert_{C^{1,\theta_1}(B_{R+1}(0))}\leq C_{\theta_1,R}<+\infty,$$
for some $\theta_1\in(0,1).$  Moreover, we rewrite (\ref{rewriteueqthm13c2regularity}) as
\begin{align*}
-\Delta u_{1,\varepsilon}=-C_H|\nabla u_{1,\varepsilon}|^{\gamma}+g_{\varepsilon}\in C^{1,\theta_2}(B_{R+1}(0)).
\end{align*}
One further deduces from the standard $W^{2,p}$ estimate that  
\begin{align*}
\Vert u_{1,\varepsilon}\Vert_{C^{2,\theta_3}(B_R(0))}\leq C_{\theta_3,R}<+\infty.
\end{align*}
  Then by the standard diagonal procedure and Arzel\`{a}-Ascoli theorem, we have from (\ref{ss11newnew}), (\ref{fromregularitythm31}), (\ref{convergencem1m2varepsilonthm131}) and (\ref{convergencethm13holdercontinuousnew}) that there exist $ u_1\in C^2(\mathbb R^N)$ and ${ {w}}_1\in L^{\gamma'}(\mathbb R^N)$ such that 
\begin{align}\label{eq4.180}
u_{1,\varepsilon}\rightarrow u_1\text{~in~}C_{\text{loc}}^2(\mathbb R^N),~~{{w}}_{1,\varepsilon}\rightharpoonup {{w}}_1\text{~in~}L^{\gamma'}(\mathbb R^N),
\end{align}
and $(u_1,m,{{w}}_1)$ satisfies
\begin{align*}
\left\{\begin{array}{ll}
-\Delta u_1+C_H|\nabla u_1|^{\gamma}-\frac{\gamma'}{N}=-a^*m^{\frac{\gamma'}{N}},&x\in\mathbb R^N,\\
-\Delta m=\gamma C_H\nabla\cdot(m|\nabla u_1|^{\gamma-2}\nabla u_1)=-\nabla\cdot {w}_1,&x\in\mathbb R^N,\\
0<\int_{\mathbb R^N}m\,dx\leq 1,
\end{array}
\right.
\end{align*}
where we have used (\ref{lambda1inthm13asymptotic}) and (\ref{mlowerboundpositivelocallimit}).  In addition, by Lemma \ref{poholemma} and (\ref{GNinequalityused}), one finds
\begin{align}\label{4182024106}
\int_{\mathbb R^N}m\,dx=1.
\end{align}
Thus, with the aid of (\ref{convergencem1m2varepsilonthm131}), we obtain for $i=1,2$, $m_{i,\varepsilon}\rightarrow m \text{~in~} L^1(\mathbb R^N).$  Moreover, \eqref{convergencethm13holdercontinuousnew} indicates
\begin{align}\label{wefurtherhavelpconvergencethm13}
m_{i,\varepsilon}\rightarrow m\text{~in~} L^p(\mathbb R^N),~\forall p\geq 1.
\end{align}
This combine with \eqref{eq4.180}
show that \eqref{mlimiting20923} holds  for $i=1$.

Next, to prove  \eqref{x1varepsilonminusx2varepsilon},
we first recall that $u_{2,\textbf{a}}(x_{2,\varepsilon})=0=\inf_{x\in \mathbb R^N}u_{2,\textbf{a}}(x)$.  Then, we have from (\ref{ss11newnew}) and \eqref{scalingthm31profile} that 
\begin{align*}
\lambda_2\geq& V_2(x_{2,\varepsilon})-\alpha_2 m_{2,\textbf{a}}^{\frac{\gamma'}{N}}(x_{2,\varepsilon})-\beta m_{2,\textbf{a}}^{\frac{\gamma'}{2N}-\frac{1}{2}}(x_{2,\varepsilon})m_{1,\textbf{a}}^{\frac{1}{2}+\frac{\gamma'}{N}}(x_{2,\varepsilon})\\
\geq &\varepsilon^{-\gamma'}\bigg[-\alpha_2m_{2,\varepsilon}^{\frac{\gamma'}{N}}\bigg(\frac{x_{2,\varepsilon}-x_{1,\varepsilon}}{\varepsilon}\bigg)-\beta m_{2,\varepsilon}^{\frac{\gamma'}{2N}-\frac{1}{2}}\bigg(\frac{x_{2,\varepsilon}-x_{1,\varepsilon}}{\varepsilon}\bigg)m_{1,\varepsilon}^{\frac{1}{2}+\frac{\gamma'}{N}}\bigg(\frac{x_{2,\varepsilon}-x_{1,\varepsilon}}{\varepsilon}\bigg)\bigg],
\end{align*}
which implies \begin{align}\label{lambda2thm13lowerboundvarepsilonconvergence}
\alpha_2 m_{2,\varepsilon}^{\frac{\gamma'}{N}}\bigg(\frac{x_{2,\varepsilon}-x_{1,\varepsilon}}{\varepsilon}\bigg)+\beta m_{2,\varepsilon}^{\frac{\gamma'}{2N}-\frac{1}{2}}\bigg(\frac{x_{2,\varepsilon}-x_{1,\varepsilon}}{\varepsilon}\bigg)m_{1,\varepsilon}^{\frac{1}{2}+\frac{\gamma'}{N}}\bigg(\frac{x_{2,\varepsilon}-x_{1,\varepsilon}}{\varepsilon}\bigg)\geq \varepsilon^{\gamma'}\lambda_2\geq \frac{\gamma'}{2N}\text{~as~}\varepsilon\rightarrow 0.
\end{align}
Combining (\ref{convergencethm13holdercontinuousnew}) with (\ref{wefurtherhavelpconvergencethm13}), one can easily check that for $i=1,2$
\begin{align*}
\lim_{|x|\rightarrow +\infty}m_{i,\varepsilon}(x)=0\text{~uniformly~in~}\varepsilon.
\end{align*}
Combining this with
\eqref{lambda2thm13lowerboundvarepsilonconvergence}, one has   (\ref{x1varepsilonminusx2varepsilon}) holds.  

We next similarly show that there exist $ u_2\in C^2(\mathbb R^N)$ and ${w}_2\in L^{\gamma'}(\mathbb R^N)$ such that 
\begin{align*}
u_{2,\varepsilon}\rightarrow u_2\text{~in~}C_{\text{loc}}^2(\mathbb R^N),\text{~and~}{\textbf{w}}_{2,\varepsilon}\rightharpoonup {\textbf{w}}_2 \text{~in~} L^{\gamma'}(\mathbb R^N),
\end{align*}
and $(u_2,m,{w}_2)$ satisfies (\ref{satisfythm13eqsinglezeropotential}), in which $(m,{w}_2)$ is a minimizer of (\ref{GNinequalitybest}).  Indeed, we rewrite the $u_2$-equation in \eqref{ss11newnew} as
\begin{align}\label{argueforbddofu2}
-\Delta u_{2,\varepsilon}+C_H|\nabla u_{2,\varepsilon}|^{\gamma}&=-\lambda_2\varepsilon^{\gamma'}+\varepsilon^{\gamma'}V_2(\varepsilon x+x_{1,\varepsilon})-\alpha_2 m_{2,\varepsilon}^{\frac{\gamma'}{N}}-\beta m_{2,\varepsilon}^{\frac{\gamma'}{2N}-\frac{1}{2}}m_{1,\varepsilon}^{\frac{1}{2}+\frac{\gamma'}{2N}}\nonumber\\
&:=h_{\varepsilon}\in L^{\infty}_{\text{loc}}(\mathbb R^N)\cap C_{\text{loc}}^{0,\theta_4}(\mathbb R^N).
\end{align}
Moreover, by Lemma \ref{sect2-lemma21-gradientu}, one has for any $R>0$ large enough, 
\begin{align}\label{gradientestimateboundedthm13u2component}
|\nabla u_{2,\varepsilon}(x)|\leq C_R<+\infty,~\forall |x|<2R.
\end{align}
In light of  
$
u_{2,\varepsilon}\bigg(\frac{x_{2,\varepsilon}-x_{1,\varepsilon}}{\varepsilon}\bigg)=0=\inf_{x\in\mathbb R^N}u_{2,\varepsilon}(x),
$ we use (\ref{x1varepsilonminusx2varepsilon}) and (\ref{gradientestimateboundedthm13u2component}) to get
\begin{align*}
|u_{2,\varepsilon}(0)|\leq C_{R}\bigg|\frac{x_{2,\varepsilon}-x_{1,\varepsilon}}{\varepsilon}\bigg|+\bigg|u_{2,\varepsilon}\bigg(\frac{x_{2,\varepsilon}-x_{1,\varepsilon}}{\varepsilon}\bigg)\bigg|\leq \tilde C_{R}<+\infty.
\end{align*}
  Thus, thanks to  (\ref{gradientestimateboundedthm13u2component}), we find
\begin{align}\label{69innoteseq}
|u_{2,\varepsilon}(x)|\leq C_R,\text{~~}\forall |x|<2R.
\end{align}
Upon collecting (\ref{argueforbddofu2}), (\ref{gradientestimateboundedthm13u2component}) and (\ref{69innoteseq}), one obtains
\begin{align*}
\Vert u_{2,\varepsilon}\Vert_{C^{2,\theta_5}(B_R(0))}\leq C_{\theta_5,R}<+\infty.
\end{align*}
Moreover, we similarly get $(u_2,m,{w}_2)$ satisfies (\ref{satisfythm13eqsinglezeropotential}), in which $(m,{w}_2)$ is a minimizer of (\ref{GNinequalitybest}). To finish the proof of \eqref{mlimiting20923}, it remains to show that $u_1=u_2$ and $w_1=w_2$, which can be obtained by
 following the argument shown in the proof of Theorem 2.4 in \cite{Lasry}, 
 Indeed, since $(m,u_1,\lambda)$ and $(m,u_2,\lambda)$ solve (\ref{satisfythm13eqsinglezeropotential}) with $w_i:=\gamma m|\nabla u_i|^{\gamma-2}\nabla u_i$, we test the $u_1-u_2$ equation and $m_1-m_2$ equation against $m_1-m_2$ and $u_1-u_2$ and integrate them by parts, then subtract them to get a useful identity.  With the aid of the strict convexity of $|p|^{\gamma}$, $\gamma>1$, one has the conclusion $\nabla u_1=\nabla u_2$ and  then $w_1=w_2.$  By fixing the same minimum points of $u_1$ and $u_2$, we obtain $u_1=u_2.$ 
 
 Finally,  proceeding the similar argument as shown in the proof of Case (ii), Theorem 1.4 in \cite{cirant2024critical}, we have (\ref{moreoverholdsfinal}) holds. 
\end{proof}

Theorem \ref{thm13attractive} demonstrates that under mild assumptions \eqref{Vicondition1} and \eqref{Vicondition2}, ground states are localized as $\bf{a}\nearrow \bf{a}_{\beta}^*$. We next present the proof of Theorem \ref{thm16refinedblowup}, which is for the refined asymptotic profiles of ground states.  First of all, we establish the  following upper bound of $e_{\alpha_1,\alpha_2,\beta}$ given by (\ref{problem1p1}):
\begin{lemma}\label{lemma4120240724}
Under the assumptions of Theorem \ref{thm16refinedblowup},  we have  as $(\alpha_1,\alpha_2)\nearrow(a^*-\beta,a^*-\beta)$,
\begin{align}\label{2point1negativebetanotes}
0\leq e_{\alpha_1,\alpha_2,\beta}\leq \bigg(\frac{\gamma'+p_0}{p_0}\bigg)\bigg(\frac{\mu\bar{\nu}_{p_0} p_0}{\gamma'}\bigg)^{\frac{\gamma'}{\gamma'+p_0}}\bigg(\frac{2}{a^*}\bigg)^{\frac{p_0}{\gamma'+p_0}}\bigg(a^*-\frac{\alpha_1+\alpha_2+2\beta}{2}\bigg)^{\frac{p_0}{\gamma'+p_0}}(1+o(1)).
\end{align}
\end{lemma}
\begin{proof}
From the definition of $\bar {\nu}_{p_0}$ in \eqref{Hmoibarnupi}, one can easily derive that, for any $\nu>\bar {\nu}_{p_0}$,  there exist $(m_0,w_0)\in \mathcal{M}$  and $y\in\mathbb R^N$  such that 
\begin{align}\label{eq-nu}
\bar {\nu}_{p_0}\leq H_{m_0,p_0}(y)=\int_{\mathbb R^N}|x+y|^{p_0}m_0(x)dx\leq \nu.\end{align} 
Since  $(m_0,w_0)\in \mathcal{M}$ is  a minimizer of (\ref{GNinequalitybest}), we have from (\ref{GNinequalitybest}) and Lemma \ref{poholemma} that 
\begin{align}\label{428negativebeta1}
\int_{\mathbb R^N}m_0\,dx=1,~C_L\int_{\mathbb R^N}\bigg|\frac{w_0}{m_0}\bigg|^{\gamma'}m_0\,dx=1,\text{ and }\frac{N}{N+\gamma'}\int_{\mathbb R^N}m^{1+\frac{\gamma'}{N}}_0\,dx=\frac{1}{a^*}.
\end{align}
Let    $x_j\in Z_0$ with $Z_0$ given by \eqref{eq-Z0}, and define 
\begin{align}\label{429negativebeta2}
m_{\tau}(x)=\tau^{N}m_0(\tau(x-x_j)-y),~~w_{\tau}(x)=\tau^{N+1}w_0(\tau(x-x_j)-y),
\end{align}
then one finds from (\ref{428negativebeta1}) and (\ref{429negativebeta2}) that  
\begin{align*}
C_L\int_{\mathbb R^N}\bigg|\frac{w_{\tau}}{m_\tau}\bigg|^{\gamma'}m_\tau\,dx=\tau^{\gamma'} C_L\int_{\mathbb R^N}\bigg|\frac{w_0}{m_0}\bigg|^{\gamma'}m_0\,dx=\tau^{\gamma'},
\end{align*}
\begin{align*}
\frac{N}{N+\gamma'}\int_{\mathbb R^N}m^{1+\frac{\gamma'}{N}}_{\tau}\,dx=\frac{N}{N+\gamma'}\tau^{\gamma'}\int_{\mathbb R^N}m_0^{1+\frac{\gamma'}{N}}\,dx=\frac{\tau^{\gamma'}}{a^*},
\end{align*}
and
\begin{align}\label{combining1432negativerefine}
\int_{\mathbb R^N}(V_1+V_2)m_{\tau}\,dx=&\int_{\mathbb R^N}(V_1+V_2)\bigg(\frac{x+y}{\tau}+x_j\bigg)m_0(x)\,dx\nonumber\\
=&\frac{1}{\tau^{p_0}}\int_{\mathbb R^N}\frac{(V_1+V_2)\big(\frac{x+y}{\tau}+x_j\big)}{\big|\frac{x+y}{\tau}\big|^{p_0}}|x+y|^pm_0\,dx.
\end{align}
Note from \eqref{eq-Z0} that 
\begin{align}\label{combining1432negativerefine2}
\lim_{\tau\rightarrow+\infty}\frac{(V_1+V_2)\big(\frac{x+y}{\tau}+x_j\big)}{\big|\frac{x+y}{\tau}\big|^{p_0}}=\mu.
\end{align}
Combining \eqref{eq-nu}, (\ref{combining1432negativerefine}) with (\ref{combining1432negativerefine2}), one can get as $\tau\to+\infty$,
\begin{align*}
\int_{\mathbb R^N}(V_1+V_2)m_{\tau}(x)\,dx=\frac{\mu\nu}{\tau^{p_0}}+O\bigg(\frac{1}{\tau^{p_0}}\bigg).
\end{align*}
Finally, by taking 
$\tau=\bigg(\frac{\mu\nu p_0a^*}{2\gamma'\big(a^*-\frac{\alpha_1+\alpha_2+2\beta}{2}\big)}\bigg)^{\frac{1}{\gamma'+p_0}}$
in (\ref{energy1p3}), we obtain
\begin{align*}
0\leq e_{\alpha_1,\alpha_2,\beta}&\leq\mathcal E(m_{\tau},w_{\tau},m_{\tau},w_{\tau})=\tau^{\gamma'}\bigg[2-\frac{\alpha_1+\alpha_2+2\beta}{a^*}\bigg]+\frac{\mu\nu}{\tau^{p_0}}+O\bigg(\frac{1}{\tau^{p_0}}\bigg)\nonumber\\
=&\frac{\gamma'+p_0}{p_0}\bigg(\frac{\mu\nu p_0}{\gamma'}\bigg)^{\frac{\gamma'}{\gamma'+p_0}}\bigg(\frac{2}{a^*}\bigg)^{\frac{p_0}{\gamma'+p_0}}\bigg(a^*-\frac{\alpha_1+\alpha_2+2\beta}{2}\bigg)^{\frac{p_0}{\gamma'+p_0}}(1+o_{\tau}(1)),
\end{align*}
which indicates \eqref{2point1negativebetanotes} since $\nu>\bar {\nu}_{p_0}$ is arbitrary. 
\end{proof}
Now, we are ready to prove Theorem \ref{thm16refinedblowup}, which is
\vspace{2mm}

{\it{Proof of Theorem \ref{thm16refinedblowup}:}}
\begin{proof}
In light of (\ref{scalingthm31profile}), we compute
\begin{align}\label{3point1notesbetanegative2}
e_{\alpha_1,\alpha_2,\beta}=&\mathcal E(m_{1,\varepsilon},w_{1,\varepsilon},m_{2,\varepsilon},w_{2,\varepsilon})\nonumber\\
=&\sum_{i=1}^2\bigg[\varepsilon^{-\gamma'}C_L\int_{\mathbb R^N}\bigg|\frac{w_{i,\varepsilon}}{m_{i,\varepsilon}}\bigg|^{\gamma'}m_{i,\varepsilon}\,dx-\frac{\alpha_i\varepsilon^{-\gamma'}}{1+\frac{\gamma'}{N}}\int_{\mathbb R^N}m_{i,\varepsilon}^{1+\frac{\gamma'}{N}}\,dx+\int_{\mathbb R^N}V(\varepsilon x+x_{\varepsilon})m_{i,\varepsilon}\,dx\bigg]\nonumber\\
&-\frac{2\beta\varepsilon^{-\gamma'}}{1+\frac{\gamma'}{N}}\int_{\mathbb R^N}m^{\frac{1}{2}+\frac{\gamma'}{2N}}_{1,\varepsilon}m^{\frac{1}{2}+\frac{\gamma'}{2N}}_{2,\varepsilon}\,dx,
\end{align}
where we redefine $x_{1,\varepsilon}$ as $x_{\varepsilon}$ here and in the sequel for simplicity.  Noting that $V_1(x_0)=V_2(x_0)$ shown in Theorem \ref{thm13attractive}, we find there exists some  $ 1\leq j\leq l$ such that $x_0=x_j$.  Then, we rewrite the potential energy as 
\begin{align}\label{3point2notesnegativebeta2}
\int_{\mathbb R^N}V_i(\varepsilon x+x_{\varepsilon})m_{i,\varepsilon}(x)\,dx=\varepsilon^{p_j}\int_{\mathbb R^N}\frac{V_i(\varepsilon x+x_{\varepsilon})}{|\varepsilon x+x_{\varepsilon}-x_j|^{p_j}}\bigg|x+\frac{x_{\varepsilon}-x_j}{\varepsilon}\bigg|^{p_j}m_{i,\varepsilon}\,dx,
\end{align}
where $i=1,2.$  In addition, since $x_{\varepsilon}\rightarrow x_j$, we obtain
\begin{align*}
\lim_{\varepsilon\rightarrow 0^+}\sum_{i=1}^2\frac{V_i(\varepsilon x+x_{\varepsilon})}{|\varepsilon x+x_{\varepsilon}-x_j|^{p_{j}}}=\mu_j\text{ a.e. in } \mathbb R^N,
\end{align*}
where $\mu_j$ is defined in $\eqref{mujthm1point5}$.  Without loss of generality, we assume $p_{2j}\geq p_{1j}=p_j$ with $p_{1j}$ and $p_{2j}$ defined by \eqref{Vi125negativebeta2024}.

Now, we claim that 
    \begin{align}\label{claimattractiverefined}
    p_j=p_0=\max\{p_1,\cdots,p_l\},\text{ and }\big|\frac{x_\varepsilon-x_j}{\varepsilon}\big|\text{ is uniformly bounded as }\varepsilon \rightarrow 0^+.
    \end{align}
    To show (\ref{claimattractiverefined}), we argue by contradiction and obtain either $p_j<p_0$ or up to a subsequence, 
    \begin{align*}
    \lim_{\varepsilon\rightarrow 0^+}\bigg|\frac{x_{\varepsilon}-x_j}{\varepsilon}\bigg|=+\infty.
    \end{align*}
    By using (\ref{3point2notesnegativebeta2}) and $m_{i,\varepsilon}\rightarrow m_0\text{ in }L^1\cap L^\infty$ shown in Theorem \ref{thm13attractive}, one deduces that for any $\Gamma>0$ large enough,
    \begin{align}\label{441negativebeta}
    &\lim_{\varepsilon\rightarrow 0}\varepsilon^{-p_0}\int_{\mathbb R^N}V_1(\varepsilon x+x_{\varepsilon})m_{1,\varepsilon}\,dx\nonumber\\
    =&\lim_{\varepsilon\rightarrow 0}\varepsilon^{p_j-p_0}\int_{\mathbb R^N}\frac{V_1(\varepsilon x+x_{\varepsilon})}{|\varepsilon x+x_{\varepsilon}-x_j|^{p_j}}\big|x+\frac{x_{\varepsilon}-x_j}{\varepsilon}\big|^{p_j}m_{1,\varepsilon}\,dx \geq \Gamma.
\end{align}
Recall the definition of $\varepsilon$ shown in \eqref{defvarepsilon316} and the estimate of \eqref{thm13conclusion2},
then we find
\begin{align}\label{interactionbetanegative442}
\int_{\mathbb R^N}m_{1,\varepsilon}^{\frac{1}{2}+\frac{\gamma'}{2N}}m^{\frac{1}{2}+\frac{\gamma'}{2N}}_{2,\varepsilon}\,dx&=\int_{\mathbb R^N}m^{1+\frac{\gamma'}{N}}_{1,\varepsilon}\,dx+\int_{\mathbb R^N}\bigg(m^{\frac{1}{2}+\frac{\gamma'}{2N}}_{2,\varepsilon}-m^{\frac{1}{2}+\frac{\gamma'}{2N}}_{1,\varepsilon}\bigg)m^{\frac{1}{2}+\frac{\gamma'}{2N}}_{1,\varepsilon}\,dx\nonumber\\
&=\int_{\mathbb R^N}m^{1+\frac{\gamma'}{N}}_{1,\varepsilon}\,dx+o_{\varepsilon}(1).
\end{align}
Thus, one finds
\begin{align}\label{3p5notesbetanegative}
&C_L\int_{\mathbb R^N}\bigg|\frac{w_{1,\varepsilon}}{m_{1,\varepsilon}}\bigg|^{\gamma'}m_{1,\varepsilon}\,dx-\frac{\alpha_1}{1+\frac{\gamma'}{N}}\int_{\mathbb R^N}m_{1,\varepsilon}^{1+\frac{\gamma'}{N}}\,dx-\frac{2\beta}{1+\frac{\gamma'}{N}}\int_{\mathbb R^N}m_{1,\varepsilon}^{\frac{1}{2}+\frac{\gamma'}{2N}}m_{2,\varepsilon}^{\frac{1}{2}+\frac{\gamma'}{2N}}\,dx\nonumber\\
=&C_L\int_{\mathbb R^N}\big|\frac{w_{1,\varepsilon}}{m_{1,\varepsilon}}\big|^{\gamma'}m_{1,\varepsilon}\,dx-\frac{\alpha_1+2\beta}{1+\frac{\gamma'}{N}}\int_{\mathbb R^N}m_{1,\varepsilon}^{1+\frac{\gamma'}{N}}\,dx+o_{\varepsilon}(1)\nonumber\\
\geq &\bigg(1-\frac{\alpha_1+2\beta}{a^*}\bigg)\int_{\mathbb R^N}\bigg|\frac{w_{1,\varepsilon}}{m_{1,\varepsilon}}\bigg|^{\gamma'}m_{1,\varepsilon}\,dx+o_{\varepsilon}(1).
\end{align}
In addition, in light of \eqref{thm13conclusionpart23} and \eqref{defvarepsilon316}, one has
\begin{align}\label{alsohavefact445negatvebeta}
\int_{\mathbb R^N}\bigg|\frac{w_{i,\varepsilon}}{m_{i,\varepsilon}}\bigg|^{\gamma'}m_{i,\varepsilon}\,dx=1+o_{\varepsilon}(1),~i=1,2,
\end{align}
and obtain from (\ref{3p5notesbetanegative}) that 
\begin{align}\label{3p6notesbetanegative}
\int_{\mathbb R^N}\bigg|\frac{w_{2,\varepsilon}}{m_{2,\varepsilon}}\bigg|^{\gamma'}m_{2,\varepsilon}\,dx-\frac{\alpha_2}{1+\frac{\gamma'}{N}}\int_{\mathbb R^N}m^{1+\frac{\gamma'}{N}}_{2,\varepsilon}\,dx\geq \bigg(1-\frac{\alpha_2}{a^*}\bigg)\int_{\mathbb R^N}\bigg|\frac{w_{2,\varepsilon}}{m_{2,\varepsilon}}\bigg|^{\gamma'}m_{2,\varepsilon}\,dx.
\end{align}
Upon substituting (\ref{441negativebeta}), (\ref{interactionbetanegative442}) and  (\ref{3p6notesbetanegative}), (\ref{alsohavefact445negatvebeta}), one finds from  (\ref{3point1notesbetanegative2}) that 
\begin{align*}
\mathcal E(m_{1,\varepsilon},w_{1,\varepsilon},m_{2,\varepsilon},w_{2,\varepsilon}) \geq &\varepsilon^{-\gamma'}\bigg[1-\frac{\alpha_1+\alpha_2+2\beta}{a^*}\bigg](1+o(1))+\Gamma\varepsilon^{-p_0}\nonumber\\
\geq &(1+o_{\varepsilon}(1))\frac{p_0+\gamma'}{p_0}\bigg(\frac{p_0\Gamma}{\gamma'}\bigg)^{\frac{\gamma'}{\gamma'+p_0}}\bigg(\frac{2}{a^*}\bigg)^{\frac{\gamma'}{\gamma'+p_0}}\bigg(a^*-\frac{\alpha_1+\alpha_2+2\beta}{2}\bigg)^{\frac{p_0}{\gamma'+p_0}},
\end{align*}
which is contradicted to Lemma \ref{lemma4120240724}.  This completes the proof of claim (\ref{claimattractiverefined}).  Hence, we obtain $\exists y_0\in\mathbb R^N$ such that 
\begin{align*}
\lim_{\varepsilon\rightarrow 0}\frac{x_{\varepsilon}-x_{j}}{\varepsilon}=y_0.
\end{align*}
We next show that $y_0$ satisfies  (\ref{136refinedbetanegativenoteshold2}).  Since $p_i=p_0$, it follows from Theorem \ref{thm13attractive} that 
\begin{align}\label{4100negativebetanotes}
&\lim_{\varepsilon\rightarrow 0^+}\varepsilon^{-p_0}\int_{\mathbb R^N}\sum_{i=1}^2V_i(\varepsilon x+x_{\varepsilon})m_{i,\varepsilon}(x)\,dx\nonumber\\
=&\lim_{\varepsilon\rightarrow 0^+}\int_{\mathbb R^N}\frac{\sum_{i=1}^2V_i\bigg(\varepsilon(x+\frac{x-x_j}{\varepsilon})+x_j\bigg)}{\big|\varepsilon(x+\frac{x-x_j}{\varepsilon})\big|^{p_0}}\bigg|x+\frac{x-x_{\varepsilon}}{\varepsilon}\bigg|^{p_0}m_{i,\varepsilon}\,dx\nonumber\\
\geq &\mu_j\int_{\mathbb R^N}|x+y_0|^{p_0}m_0\,dx\geq \mu \bar{\nu}_{p_0} ,
\end{align}
where the last two inequalities hold if and only if one has (\ref{136refinedbetanegativenoteshold2}).  As a consequence, we deduce from (\ref{3p5notesbetanegative}) and (\ref{3p6notesbetanegative}) that  
\begin{align}\label{4101negativebetanotes}
e_{\alpha_1,\alpha_2,\beta}\geq &\varepsilon^{-\gamma'}\bigg[2-\frac{\alpha_1+\alpha_2+2\beta}{a^*}\bigg][1+o(1)]+\varepsilon^{p_0}\mu\bar{\nu}_{p_0}[1+o(1)]\nonumber\\
\geq &[1+o(1)]\bigg[\frac{\gamma'+p_0}{p_0}\bigg(\frac{\mu\bar{\nu}_{p_0} p_0}{\gamma'}\bigg)^{\frac{\gamma'}{\gamma'+p_0}}\bigg(\frac{2}{a^*}\bigg)^{\frac{p_0}{\gamma'+p_0}}\bigg(a^*-\frac{\alpha_1+\alpha_2+2\beta}{2}\bigg)^{\frac{p_0}{\gamma'+p_0}}\bigg],
\end{align}
where the equality in the second inequality holds if and only if 
\begin{align*}
\varepsilon=\bigg(\frac{2\gamma'}{p_0\mu\bar{\nu}_{p_0} a^*}\bigg)^{\frac{1}{\gamma'+p_0}}\bigg(a^*-\frac{\alpha_1+\alpha_2+2\beta}{2}\bigg)^{\frac{1}{\gamma'+p_0}}\big(1+o_\varepsilon(1)\big).
\end{align*}
Combining the lower bound (\ref{4101negativebetanotes}) with the upper bound (\ref{2point1negativebetanotes}), we find the equalities in  \eqref{4100negativebetanotes} and (\ref{4101negativebetanotes}) hold.  As a consequence, we obtain (\ref{136refinedbetanegativenoteshold}) and (\ref{136refinedbetanegativenoteshold2}) and finish the proof of this theorem.
\end{proof}

\section{Asymptotic Profiles of Ground States with $\beta<0$}\label{sect520240929}
In this section, we shall discuss the concentration phenomena within (\ref{ss1}) under the repulsive case with $\beta<0.$  Similarly as shown in Section \ref{sect4multipopulation}, we first investigate the basic blow-up profiles of ground states with some assumptions imposed on the potentials, which is summarized as Theorem \ref{thm15blowupnegative}.  Then, we investigate the refined blow-up profiles shown in Theorem \ref{thm17multipopulation} when potentials satisfy local polynomial expansions.

{\it{Proof of Theorem \ref{thm15blowupnegative}:}}
\begin{proof}
As shown in the proof of Theorem \ref{thm11multi}, we have proved that when $\beta<0$,
\begin{align}\label{C4negativenotes}
\lim_{\textbf{a}\nearrow {\textbf{a}}^*}e_{\alpha_1,\alpha_2,\beta}=0.
\end{align}
In addition, 
one obtains from (\ref{GNinequalityused}) that 
\begin{align*}
\mathcal E_{\alpha_i}^i(m_i,w_i)\geq 0\text{~if~}\alpha_i<a^*,
\end{align*}
where $\mathcal E_{\alpha_i}^i(m_i,w_i),$ $i=1,2$ are given by (\ref{mathcalealphaii089}).
Moreover, noting that $\mathcal E_{\alpha_1,\alpha_2,\beta}(m_1,w_1,m_2,w_2)$ defined by (\ref{energy1p3}) can be written as 
\begin{align*}
\mathcal E_{\alpha_1,\alpha_2,\beta}(m_1,w_1,m_2,w_2)=\sum_{i=1}^2\mathcal E_{\alpha_i}^i(m_i,w_i)-\frac{2\beta N}{N+\gamma'}\int_{\mathbb R^N}m_1^{\frac{1}{2}+\frac{\gamma'}{2N}}m_2^{\frac{1}{2}+\frac{\gamma'}{2N}}\,dx,
\end{align*}
 we find from (\ref{C4negativenotes}) that (\ref{C7notesnegative}), (\ref{C8notesnegative}) and (\ref{C9notesnegative}) hold.

 Next, we shall prove (\ref{C10betanegative}) and argue by contradiction.  Assume that
 \begin{align*}
\limsup_{\textbf{a}\nearrow \textbf{a}^*}\int_{\mathbb R^N}C_L\bigg|\frac{w_{i,\textbf{a}}}{m_{i,\textbf{a}}}\bigg|^{\gamma'}m_{i,\textbf{a}}\,dx<+\infty, \text{ for } i=1 \text{ or }2,
\end{align*}
then it follows from (\ref{GNinequalityused}) that 
\begin{align*}
\limsup_{\textbf{a}\nearrow \textbf{a}^*}\int_{\mathbb R^N}m_{i,\textbf{a}}^{1+\frac{\gamma'}{N}}\,dx<+\infty.
\end{align*}
Therefore, we deduce from (\ref{C7notesnegative}) that 
\begin{align*}
\lim_{\textbf{a}\nearrow \textbf{a}^*}\mathcal E_{\textbf{a}^*}^i(m_{1,\textbf{a}},w_{1,\textbf{a}},m_{2,\textbf{a}},w_{2,\textbf{a}})=\lim_{\textbf{a}\nearrow \textbf{a}^*}\mathcal E_{\alpha_i}^i(m_{1,\textbf{a}},w_{1,\textbf{a}},m_{2,\textbf{a}},w_{2,\textbf{a}})=0=e^i_{a^*}.
\end{align*}
This implies that $\{(m_{i,\textbf{a}},w_{i,\textbf{a}})\}$ is a 
 is a bounded minimizing sequence of $e^i_{a^*}$ given by (\ref{problem51innotescopy}) and its limit is a minimizer of $e^{i}_{a^*}$. 
This is a contradiction to the fact that $e^i_{a^*}$ does not admit any minimizer as shown in \cite{cirant2024critical}.  Hence, one finds (\ref{C10betanegative}) holds.

Let
\begin{align*}
\hat \varepsilon_i:=\bigg(C_L\int_{\mathbb R^N}\bigg|\frac{w_{i,\textbf{a}}}{m_{i,\textbf{a}}}\bigg|^{\gamma'}m_{i,\textbf{a}}\,dx\bigg)^{-\frac{1}{\gamma'}}\rightarrow 0\text{ as }\textbf{a}\nearrow \textbf{a}^*.
\end{align*}
Recall that $(m_{1,\textbf{a}},w_{1,\textbf{a}},m_{2,\textbf{a}},w_{2,\textbf{a}})\in\mathcal K$ is  a minimizer and by using Lemma \ref{lemma32multimfg}, one has for $i=1,2$,
\begin{align*}
\lambda_{i,\textbf{a}}=&C_L\int_{\mathbb R^N}\bigg|\frac{w_{i,\textbf{a}}}{m_{i,\textbf{a}}}\bigg|^{\gamma'}m_{i,\textbf{a}}\,dx+\int_{\mathbb R^N}V_im_{i,\textbf{a}}\,dx-\alpha_i\int_{\mathbb R^N}m_{i,\textbf{a}}^{1+\frac{\gamma'}{N}}\,dx-\beta\int_{\mathbb R^N}m_{1,\textbf{a}}^{\frac{1}{2}+\frac{\gamma'}{2N}}m_{2,\textbf{a}}^{\frac{1}{2}+\frac{\gamma'}{2N}}\,dx\\
=&\mathcal E_{\alpha_i}^i(m_{i,\textbf{a}},w_{i,\textbf{a}})-\frac{N\alpha_i}{N+\gamma'}\int_{\mathbb R^N}m_{i,\textbf{a}}^{1+\frac{\gamma'}{N}}\,dx-\beta\int_{\mathbb R^N}m_{1,\textbf{a}}^{\frac{1}{2}+\frac{\gamma'}{2N}}m_{2,\textbf{a}}^{\frac{1}{2}+\frac{\gamma'}{2N}}\,dx\\
=&-\frac{\gamma'}{N}\hat \varepsilon_i^{-\gamma'}+o_{\varepsilon_i}(1),
\end{align*}
which implies 
\begin{align}\label{c15notebetanegative}
\lambda_{i,\textbf{a}}\hat\varepsilon_i^{\gamma'}\rightarrow -\frac{\gamma'}{N}\text{~as~}\hat \varepsilon_i\rightarrow 0^+,~i=1,2.
\end{align}
Since $(u_{1,\textbf{a}},u_{2,\textbf{a}})$ is bounded from below, we have $u_{i,\textbf{a}}(x)\rightarrow +\infty$ as $|x|\rightarrow +\infty.$  Thus, there exist $ x_{i,\hat \varepsilon}$, $i=1,2$ such that 
\begin{align*}
u_{i,\hat\varepsilon}(0)=u_{i,\textbf{a}}(x_{i,\hat \varepsilon})=\inf_{x\in\mathbb R^N}u_{i,\textbf{a}}(x).
\end{align*}

By using \eqref{hatepsilonrescaling20240723} and (\ref{ss11new}), we find $(m_{1,\hat\varepsilon},u_{1,\hat \varepsilon},m_{2,\hat\varepsilon},u_{2,\hat \varepsilon})$ satisfies 
\begin{small}
\begin{align}\label{takelimitnotebetanegative}
\left\{\begin{array}{ll}
-\Delta u_{1,\hat\varepsilon}+C_H|\nabla u_{1,\hat\varepsilon}|^{\gamma}+\lambda_{1,\textbf{a}}\hat \varepsilon_1^{\gamma'}=\hat\varepsilon_1^{\gamma'}V_1(\hat \varepsilon_1 x+x_{1,\hat \varepsilon})-\alpha_1m_{1,\hat \varepsilon}^{\frac{\gamma'}{N}}-\beta\big(\frac{\hat \varepsilon_1}{\hat \varepsilon_2}\big)^{\frac{\gamma'}{2}+\frac{N}{2}}m_{1,\hat \varepsilon}^{\frac{\gamma'}{2N}-\frac{1}{2}}m_{2,\hat \varepsilon}
^{\frac{\gamma'}{2N}+\frac{1}{2}}\big(\frac{\hat \varepsilon_1x+x_{1,\hat \varepsilon}-x_{2,\hat \varepsilon}}{\hat \varepsilon_2}\big),\\
-\Delta m_{1,\hat \varepsilon}=C_H\gamma\nabla\cdot(m_{1,\hat \varepsilon}|\nabla u_{1,\hat \varepsilon}|^{\gamma-2}\nabla u_{1,\hat \varepsilon})=-\nabla \cdot w_{1,\hat \varepsilon},\\
-\Delta u_{2,\hat\varepsilon}+C_H|\nabla u_{2,\hat\varepsilon}|^{\gamma}+\lambda_{2,\textbf{a}}\hat \varepsilon_2^{\gamma'}=\hat \varepsilon_2^{\gamma'}V_2(\hat \varepsilon_2 x+x_{2,\hat \varepsilon})-\alpha_2m_{2,\hat \varepsilon}^{\frac{\gamma'}{N}}-\beta\big(\frac{\hat \varepsilon_2}{\hat \varepsilon_1}\big)^{\frac{\gamma'}{2}+\frac{N}{2}}m_{2,\hat \varepsilon}^{\frac{\gamma'}{2N}-\frac{1}{2}}m_{1,\hat \varepsilon}
^{\frac{\gamma'}{2N}+\frac{1}{2}}\big(\frac{\hat \varepsilon_2x+x_{2,\hat \varepsilon}-x_{1,\hat \varepsilon}}{\hat \varepsilon_1}\big),\\
-\Delta m_{2,\hat \varepsilon}=C_H\gamma\nabla\cdot(m_{2,\hat \varepsilon}|\nabla u_{2,\hat \varepsilon}|^{\gamma-2}\nabla u_{2,\hat \varepsilon})=-\nabla \cdot w_{2,\hat \varepsilon}.
\end{array}
\right.
\end{align}
\end{small}
Then by applying the maximum principle on (\ref{takelimitnotebetanegative}), one finds for $i,j=1,2$ and $i\not=j$ that 
\begin{align*}
    \lambda_{i,\textbf{a}}\hat \varepsilon_{i}^{\gamma'}\geq -\alpha_im_{i,\hat \varepsilon}^{\frac{\gamma'}{N}}(0)+\hat \varepsilon_i^{\gamma'}V_i(\hat \varepsilon_i x+x_{i,\hat \varepsilon})-\beta \bigg(\frac{\hat \varepsilon_i}{\hat \varepsilon_j}\bigg)^{\frac{\gamma'}{2}+\frac{N}{2}}m_{i,\hat \varepsilon}^{\frac{\gamma'}{2N}-\frac{1}{2}}(0)m_{j,\hat \varepsilon}^{\frac{\gamma'}{2N}+\frac{1}{2}}\bigg(\frac{x_{i,\hat \varepsilon}-x_{j,\hat \varepsilon}}{\hat \varepsilon_j}\bigg),
\end{align*}
Noting that $\alpha_i>0$, $\beta<0$ and $V_i\geq 0$ with $i=1,2$, we further have from \eqref{c15notebetanegative} when $\alpha_i\nearrow a^*,$ 
\begin{align}\label{c19notesbetanegative}
C\geq m_{i,\varepsilon}^{\frac{\gamma'}{N}}(0)>\frac{\gamma'}{2a^*N}>0,
\end{align}
where $C>0$ is a constant.  Invoking  \eqref{C9notesnegative} and (\ref{C10betanegative}), we obtain
\begin{align}\label{c22notesbetanegative}
    \int_{\mathbb R^N}V_i(\hat \varepsilon_ix+x_{i,\hat \varepsilon})m_{i,\hat \varepsilon}(x)\,dx\rightarrow 0\text{~as~}\textbf{a}\nearrow \textbf{a}^*,
\end{align}
and
\begin{align}\label{c21innotesbetanegative}
\int_{\mathbb R^N}C_L\bigg|\frac{w_{i,\hat \varepsilon}}{m_{i,\hat \varepsilon}}\bigg|^{\gamma'}m_{i,\hat \varepsilon}\,dx=1,~~\int_{\mathbb R^N}m_{i,\hat \varepsilon}^{1+\frac{\gamma'}{N}}\,dx\rightarrow \frac{N+\gamma'}{Na^*}.
\end{align}
Now, we claim up to a subsequence, 
\begin{align}\label{claimc23notesbetanegative}
x_{i,\hat \varepsilon}\rightarrow x_{i}\text{ with }V_i(x_i)=0,~i=1,2.
\end{align}
Indeed, we have from (\ref{c21innotesbetanegative}) and Lemma \ref{lemma21-crucial-cor} that 
\begin{align}\label{c23primenotebetsnegative}
    \limsup_{\hat \varepsilon_1,\hat\varepsilon_2\rightarrow 0^+}\Vert m_{i,\hat \varepsilon}\Vert_{W^{1,\gamma'}(\mathbb R^N)}<+\infty.
\end{align}
Moreover, since $\gamma'>N$, one gets from Morrey's estimate that 
\begin{align}\label{c24notebetsnegative}
\limsup_{\hat \varepsilon_1,\hat\varepsilon_2\rightarrow 0^+}\Vert m_{i,\hat \varepsilon}\Vert_{C^{0,1-\frac{N}{\gamma'}}(\mathbb R^N)}<+\infty.
\end{align}
(\ref{c24notebetsnegative}) together with (\ref{c19notesbetanegative}) gives  that there exists $ R>0$ such that 
\begin{align}\label{c25notesbetanegative}
    m_{i,\hat \varepsilon}(x)\geq \frac{C}{2}>0,~\forall|x|<R,~i=1,2,
\end{align}
where $C>0$ is a constant independent of $\hat \varepsilon_i$.  As a consequence, we obtain claim (\ref{claimc23notesbetanegative}) thanks to  (\ref{c22notesbetanegative}) and (\ref{c25notesbetanegative}).
In light of (\ref{c26notesbetanegative}) and (\ref{claimc23notesbetanegative}), one finds
\begin{align*}
\lim_{\hat \varepsilon_1,\hat \varepsilon_2\rightarrow 0^+}\frac{|x_{1,\hat \varepsilon}-x_{2,\hat \varepsilon}|}{\hat\varepsilon_i}=+\infty,~i=1,2.
\end{align*}

Next, we study the convergence of $(m_{1,\hat\varepsilon},u_{1,\hat\varepsilon},m_{2,\hat\varepsilon},u_{2,\hat\varepsilon})$ as $\hat\varepsilon_i\rightarrow 0$ with $i=1,2.$  First of all, we have from (\ref{c23primenotebetsnegative}) and (\ref{c25notesbetanegative}) that there exist $0\not\equiv, \leq m_i\in W^{1,\gamma'}(\mathbb R^N)$ with $i=1,2$ such that 
\begin{align*}
m_{i,\hat \varepsilon}\rightharpoonup m_i\text{ in }W^{1,\gamma'}(\mathbb R^N).
\end{align*}
Without loss of the generality, we assume
\begin{align}\label{assumenotesbetanegative}
\hat \varepsilon_1\geq \hat \varepsilon_2.
\end{align}
In light of (\ref{c24notebetsnegative}) and (\ref{assumenotesbetanegative}), one has there exists $C>0$ independent of $\hat \varepsilon_1$ and $\hat \varepsilon_2$ such that 
\begin{align*}
\beta\bigg(\frac{\hat \varepsilon_2}{\hat \varepsilon_1}\bigg)^{\frac{N}{2}+\frac{\gamma'}{2}}m_{1,\hat \varepsilon}^{\frac{1}{2}+\frac{\gamma'}{2N}}\bigg(\frac{\hat \varepsilon_2 x+x_{2,\hat \varepsilon}-x_{1,\hat \varepsilon}}{\hat \varepsilon_1}\bigg)m_{2,\hat \varepsilon}^{\frac{\gamma'}{2N}-\frac{1}{2}}(x)\leq C \text{ for all } x\in \mathbb R^N.
\end{align*}
In addition, by using Lemma \ref{sect2-lemma21-gradientu}, one obtains for any $x\in B_{R}(0),$ 
\begin{align}\label{446notenegativebeta}
|\nabla u_{2,\hat \varepsilon}(x)|\leq C_R,
\end{align}
where $C_R>0$ is a constant.  Moreover, the $u_{2,\hat\varepsilon}$-equation in (\ref{takelimitnotebetanegative}) becomes
\begin{align*}
-\Delta u_{2,\hat \varepsilon}=-C_H|\nabla u_{2,\hat \varepsilon}|^{\gamma}+h_{\hat \varepsilon}(x),
\end{align*}
where $$h_{\hat\varepsilon}(x):=-\lambda_{2,\textbf{a}}\hat \varepsilon_2^{\gamma'}+\hat \varepsilon_2^{\gamma'}V_2(\hat \varepsilon_2 x+x_{2,\hat \varepsilon})-\alpha_2m_{2,\hat \varepsilon}^{\frac{\gamma'}{N}}-\beta\big(\frac{\hat \varepsilon_2}{\hat \varepsilon_1}\big)^{\frac{\gamma'}{2}+\frac{N}{2}}m_{2,\hat \varepsilon}^{\frac{\gamma'}{2N}-\frac{1}{2}}m_{1,\hat \varepsilon}
^{\frac{\gamma'}{2N}+\frac{1}{2}}\big(\frac{\hat \varepsilon_2x+x_{2,\hat \varepsilon}-x_{1,\hat \varepsilon}}{\hat \varepsilon_1}\big).$$ is given by \eqref{rewriteueqthm13c2regularity} with $\varepsilon$ replaced by $\hat\varepsilon$.  We further find from (\ref{446notenegativebeta}) that $|-C_H|\nabla u_{2,\hat \varepsilon}|^{\gamma}+g_{\hat \varepsilon}|\leq \tilde C_R$ with $\tilde C_R>0$.  Then we apply the standard elliptic regularity to get $\Vert u_{2,\hat \varepsilon}\Vert_{C^{2,\theta}(B_R)}\leq C_R,$
where $C_R>0$ is a constant and $\theta\in(0,1).$  Thus, we take the limit in the $u_{2,\hat\varepsilon}$-equation and $m_{2,\hat\varepsilon}$-equation of (\ref{takelimitnotebetanegative}), use the diagonalization procedure and Arzel\`{a}-Ascoli theorem to deduce that 
$u_{2,\hat \varepsilon}\rightarrow u_2\text{ in }C^{2,\hat\theta}(\mathbb R^N)$ as $\hat\varepsilon_1$, $\hat\varepsilon_2\rightarrow 0^+,$
and $(m_2,u_2)$ satisfies
\begin{align*}
\left\{\begin{array}{ll}
-\Delta u_2+C_H|\nabla u_2|^{\gamma}-\frac{\gamma'}{N}=a^*m_2,\\
-\Delta m_2=C_H\gamma\nabla\cdot(m_2|\nabla u_2|^{\gamma-2}\nabla u_2)=-\nabla\cdot w_2,\\
0<\int_{\mathbb R^N}m_2\,dx\leq 1.
\end{array}
\right.
\end{align*}
Similar as the derivation of (\ref{4182024106}), one uses Lemma \ref{poholemma} to get $\int_{\mathbb R^N}m_2\,dx=1.$  It follows that $m_{2,\hat \varepsilon}\rightarrow m_{2}$ in $L^1(\mathbb R^N)$.  Combining this with (\ref{c24notebetsnegative}) , we deduce 
\begin{align}\label{455lqbetanegative}
m_{2,\hat \varepsilon}\rightarrow m_2 \text{ in }L^q(\mathbb R^N), \forall q\geq 1.
\end{align}
Invoking Lemma \ref{lowerboundVkgenerallemma22},  (\ref{c15notebetanegative}) and \eqref{c24notebetsnegative}, one has
\begin{align}\label{47820240811lowerbound}
u_{2,\hat \varepsilon}(x)\geq C\max\big\{|x|,\big(\varepsilon_2^{\gamma'}V_2(\hat \varepsilon_2 x+x_{2,\hat \varepsilon})\big)^{\frac{1}{\gamma}}\big\},\text{ if }|x|>R,
\end{align}
where $C>0$ and $R>0$ are constants independent of $\hat\varepsilon_1$ and $\hat\varepsilon_2$.  Indeed, it suffices to prove $u_{2,\hat\varepsilon}(x)\geq C|x|$ for some constant $C>0$ when $|x|>R.$  To this end, we find from (\ref{takelimitnotebetanegative}) that when $\hat \varepsilon_i$, $i=1,2$ are small,  
\begin{align}\label{20240811newlowerboundu}
-\Delta u_{2,\hat\varepsilon}+C_H|\nabla u_{2,\hat \varepsilon}|^{\gamma}+\lambda_0\geq \frac{\gamma'}{3N}-\alpha_2m_{2,\hat \varepsilon}^{\frac{\gamma'}{N}}-\beta\bigg(\frac{\hat \varepsilon_2}{\hat \varepsilon_1}\bigg)^{\frac{\gamma'}{2}+\frac{N}{2}}m_{2,\hat \varepsilon}^{\frac{\gamma'}{2N}-\frac{1}{2}}m_{1,\hat \varepsilon}
^{\frac{\gamma'}{2N}+\frac{1}{2}}\bigg(\frac{\hat \varepsilon_2x+x_{2,\hat \varepsilon}-x_{1,\hat \varepsilon}}{\hat \varepsilon_1}\bigg),
\end{align}
where $\lambda_0:=-\frac{\gamma'}{2N}$ and we have used (\ref{c15notebetanegative})  and the positivity of $V_2$.  In addition, (\ref{assumenotesbetanegative}) and \eqref{455lqbetanegative} indicate that as $|x|\rightarrow+\infty,$
\begin{align}\label{20240811newlowerboundu1}
-\alpha_2m_{2,\hat \varepsilon}^{\frac{\gamma'}{N}}-\beta\bigg(\frac{\hat \varepsilon_2}{\hat \varepsilon_1}\bigg)^{\frac{\gamma'}{2}+\frac{N}{2}}m_{2,\hat \varepsilon}^{\frac{\gamma'}{2N}-\frac{1}{2}}m_{1,\hat \varepsilon}
^{\frac{\gamma'}{2N}+\frac{1}{2}}\bigg(\frac{\hat \varepsilon_2x+x_{2,\hat \varepsilon}-x_{1,\hat \varepsilon}}{\hat \varepsilon_1}\bigg)\rightarrow 0\text{ uniformly in }\hat \varepsilon_1\text{ and }\hat \varepsilon_2.
\end{align}
Thus, one further obtains from (\ref{20240811newlowerboundu}) and (\ref{20240811newlowerboundu1}) that 
\begin{align}\label{48120240811}
-\Delta u_{2,\hat\varepsilon}+C_H|\nabla u_{2,\hat \varepsilon}|^{\gamma}+\lambda_0>0\text{ when }|x|\gg 1.
\end{align}
Now, we fix any $|\tilde x|$ large enough and define 
$$h(x):=K_1|\tilde x|\chi\bigg(\frac{x}{|\tilde x|}\bigg),$$
where constant $K_1>0$ will be chosen later and $\chi(\cdot)\geq 0$ denotes the smooth cut-off function satisfying $\chi\equiv 0$ when $x\in(0,\frac{1}{2})\cup(\frac{3}{2},+\infty).$  We compute to get 
\begin{align}\label{48220240811}
-\Delta h+C_H|\nabla h|^{\gamma}+\lambda_0\leq \frac{K_1}{|\tilde x|}+C_HK_1^{\gamma}+\lambda_0<0,
\end{align}
if we choose $K_1$ small enough.  Applying the comparison principle into (\ref{48120240811}) and (\ref{48220240811}), one has 
\begin{align*}
u_{2,\hat\varepsilon}(x)\geq h(x)\text{ for }\frac{1}{2}|\tilde x|<|x|<\frac{3}{2}|\tilde x|,
\end{align*}
which finishes the proof of (\ref{47820240811lowerbound}).

Next, we claim that for any $p>1$, there exist $ R>0$ and $C>0$ such that
$$m_{2,\hat \varepsilon}(x)\leq C|x|^{-p},~\forall|x|>R. $$
Indeed, let $\phi=u_{2,\hat \varepsilon}^p$, then we have 
\begin{align}\label{C34notesbetanegativelya}
&-\Delta\phi+C_H\gamma|\nabla u_{2,\hat \varepsilon}|^{\gamma-2}\nabla u_{2,\hat \varepsilon}\cdot\nabla 
\phi\nonumber\\
=&pu^{p-1}_{2,\hat \varepsilon}[-\Delta u_{2,\hat \varepsilon}-(p-1)\frac{|\nabla u_{2,\hat \varepsilon}|^2}{u_{2,\hat \varepsilon}}+C_H\gamma|\nabla u_{2,\hat \varepsilon}|^{\gamma}]\nonumber\\
=&pu_{2,\hat \varepsilon}^{p-1}\bigg[C_H(\gamma-1)|\nabla u_{2,\hat \varepsilon}|^{\gamma}-\lambda_2\hat \varepsilon_2^{\gamma'}-(p-1)\frac{|\nabla u_{2,\hat \varepsilon}|^2}{u_{2,\hat \varepsilon}}\nonumber\\
&+\hat \varepsilon_2^{\gamma'}V_2(\hat \varepsilon_2x+x_{2,\hat \varepsilon})-\alpha_2m_{2,
\hat \varepsilon}^{\frac{\gamma'}{N}}-\beta\bigg(\frac{\hat \varepsilon_2}{\hat \varepsilon_1}\bigg)^{\frac{\gamma'}{2}+\frac{N}{2}}m_{2,\hat \varepsilon}^{\frac{\gamma'}{2N}-\frac{1}{2}}m_{1,\hat \varepsilon}
^{\frac{\gamma'}{2N}+\frac{1}{2}}\bigg(\frac{\hat \varepsilon_2x+x_{2,\hat \varepsilon}-x_{1,\hat \varepsilon}}{\hat \varepsilon_1}\bigg)\bigg]\nonumber\\
:=&pu_{2,\hat\varepsilon}^{p-1}G_{\hat\varepsilon}(x).
\end{align}
Lemma \ref{sect2-lemma21-gradientu} implies
\begin{align}\label{456betanegative20240723}
|\nabla u_{2,\hat \varepsilon}|\leq C\big[1+\hat \varepsilon^{\gamma'}_{2}V_2(\hat \varepsilon_2x+x_{2,\hat \varepsilon})\big]^{\frac{1}{\gamma}}.
\end{align}
Hence, we deduce from (\ref{47820240811lowerbound}) that
\begin{align*}
\frac{|\nabla u_{2,\hat \varepsilon}|^{2-\gamma}}{u_{2,\hat \varepsilon}}\leq C\frac{\bigg[1+\hat \varepsilon_2^{\gamma'}V_2(\hat \varepsilon_2 x+x_{2,\hat \varepsilon})^{\frac{2-\gamma}{\gamma}}\bigg]}{\max\{|x|,[\hat \varepsilon_2^{\gamma'}V_2(\hat \varepsilon_2 x+x_{2,\hat \varepsilon})]^{\frac{1}{\gamma}}\}}\leq \frac{C_H(\gamma-1)}{2(p-1)},~\text{for~}|x|>R.
\end{align*}
Thus,
\begin{align*}
&C_H(\gamma-1)|\nabla u_{2,\hat \varepsilon}|^{\gamma}-(p-1)\frac{|\nabla u_{2,\hat \varepsilon}|^2}{u_{2,\hat \varepsilon}}\nonumber\\
=&|\nabla u_{2,\hat \varepsilon}|^{\gamma}\bigg[C_H(\gamma-1)-(p-1)\frac{|\nabla u_{2,\hat \varepsilon}|^{2-\gamma}}{u_{2,\hat \varepsilon}}\bigg]>0~\text{for }|x|>R.
\end{align*}
In light of (\ref{C34notesbetanegativelya}), we further find
\begin{align}\label{459notesbetanegative}
-\Delta\phi+C_H\gamma|\nabla u_{2,\hat \varepsilon}|^{\gamma-2}\nabla u_{2,\hat \varepsilon}\cdot\nabla \phi\geq Cpu^{p-1}_{2,\hat \varepsilon},~\text{for }|x|>R.
\end{align}
By using Theorem 3.1 in \cite{metafune2005global}, one gets 
$\int_{\mathbb R^N}m_{2,\hat \varepsilon}u^{p-1}_{2,\hat \varepsilon}\,dx<+\infty$. 
We next show that 
\begin{align}\label{461betanegative20240723}
  \limsup_{\hat \varepsilon_1,\hat\varepsilon_2\rightarrow 0^+}\int_{\mathbb R^N}m_{2,\hat \varepsilon}u^{p-1}_{2,\hat \varepsilon}\,dx<+\infty. 
\end{align}
Indeed, we test the $m_{2,\hat\varepsilon}$-equation in (\ref{takelimitnotebetanegative}) against $\phi$ and integrate it by parts to obtain  
\begin{align*}
0=\int_{\mathbb R^N}m_{2,\hat\varepsilon}[-\Delta\phi+C_H\gamma|\nabla u_{2,\hat \varepsilon}|^{\gamma-2}\nabla u_{2,\hat \varepsilon}\cdot\nabla \phi]\,dx=p\int_{\mathbb R^N}m_{2,\hat\varepsilon}G_{\hat\varepsilon}u_{2,\hat\varepsilon}^{p-1}\,dx.
\end{align*}
It follows that for some large $R_1>0$ independent of $\hat\varepsilon_i$, $i=1,2$,
\begin{align}\label{collecting1202407231}
\int_{\{x||x|>R_1\}}m_{2,\hat\varepsilon}G_{\hat\varepsilon}u_{2,\hat\varepsilon}^{p-1}\,dx=-\int_{\{x||x|\leq R_1\}}m_{2,\hat\varepsilon}G_{\hat\varepsilon}u_{2,\hat\varepsilon}^{p-1}\,dx.
\end{align}
On one hand, in light of (\ref{459notesbetanegative}), one has 
\begin{align}\label{collecting1202407232}
\int_{\{x||x|>R_1\}}m_{2,\hat\varepsilon}u_{2,\hat\varepsilon}^{p-1}\,dx\leq C\int_{\{x||x|>R_1\}}m_{2,\hat\varepsilon}G_{\hat\varepsilon}u_{2,\hat\varepsilon}^{p-1}\,dx,
\end{align}
where $C>0$ is some constant independent of $\hat \varepsilon_i$, $i=1,2$.  On the other hand, by fixing $\inf\limits_{x\in\mathbb R^N} u_{2,\hat\varepsilon}=1$ in (\ref{C34notesbetanegativelya}), we get $|G_{\hat\varepsilon}|\leq C$ for some constant $C>0$ independent of $\hat\varepsilon_i,~ i=1,2.$  Combining this with \eqref{456betanegative20240723}, one has from the boundedness of $|x_{2,\hat \varepsilon}|$ that
\begin{align}\label{collecting1202407233}
\bigg|\int_{\{x||x|\leq R_1\}}m_{2,\hat\varepsilon}G_{\hat\varepsilon}u_{2,\hat\varepsilon}^{p-1}\,dx\bigg|\leq C\int_{\mathbb R^N}m_{2,\hat\varepsilon}\,dx\leq \tilde C,
\end{align}
where  $\tilde C>0$ is  independent of $\hat\varepsilon_i,~ i=1,2.$  Collecting (\ref{collecting1202407231}), (\ref{collecting1202407232}) and (\ref{collecting1202407233}), one finds (\ref{461betanegative20240723}) holds.
Moreover, (\ref{461betanegative20240723}) indicates 
\begin{align*}
m_{2,\hat \varepsilon}(x)\leq C|x|^{1-p},~\forall p>1,
\end{align*}
 As a consequence, for any fixed $x\in\mathbb R^N$, we have
\begin{align*}
\bigg|\frac{\hat \varepsilon_1x+x_{1,\hat \varepsilon}-x_{2,\hat \varepsilon}}{\hat \varepsilon_2}\bigg|\geq \frac{\hat \varepsilon_1|x|}{\hat \varepsilon_2}+\frac{1}{2}\frac{|x_{1,\hat \varepsilon}-x_{2,\hat \varepsilon}|}{\hat \varepsilon_2}\geq \frac{C}{\hat \varepsilon_2}.
\end{align*}
  It follows that 
\begin{align}\label{464implies20240723}
\bigg(\frac{\hat \varepsilon_1}{\hat \varepsilon_2}\bigg)^{\frac{\gamma'}{2}+\frac{N}{2}}m^{\frac{\gamma'}{2N}+\frac{1}{2}}_{2,\hat \varepsilon}\bigg(\frac{\hat \varepsilon_1x+x_{1,\hat \varepsilon}-x_{2,\hat\varepsilon}}{\hat \varepsilon_2}\bigg)\leq \bigg(\frac{\hat \varepsilon_1}{\hat \varepsilon_2}\bigg)^{\frac{\gamma'}{2}+\frac{N}{2}}\hat \varepsilon_2^p\leq \hat \varepsilon_1^{\frac{\gamma'}{2}-\frac{N}{2}}\text{ by choosing }p>\frac{\gamma'}{2}+\frac{N}{2}.
\end{align}
We rewrite the $u_{1,\hat \varepsilon}$-equation in (\ref{takelimitnotebetanegative}) as 
\begin{align}\label{takelimitbefore20240723}
&-\Delta u_{1,\hat \varepsilon}+C_H|\nabla u_{1,\hat \varepsilon}|^{\gamma}+\lambda_{1,\textbf{a}}\varepsilon_1^{\gamma'}\nonumber\\
=&\hat \varepsilon_1V_1(\hat \varepsilon_1 x+x_{1,\hat \varepsilon})-\alpha_1m_1^{\frac{\gamma'}{N}}-\beta\bigg(\frac{\hat \varepsilon_1}{\hat \varepsilon_2}\bigg)^{\frac{\gamma'}{2}+\frac{N}{2}}m_{1,\hat \varepsilon}m_{2,\hat \varepsilon}\bigg(\frac{\hat \varepsilon_1 x+x_{1,\hat \varepsilon}-x_{2,\hat \varepsilon}}{\hat \varepsilon_2}\bigg):=IV_{\hat \varepsilon}.
\end{align}
Since \eqref{464implies20240723} indicates for any $\hat R>0,$ 
$$|IV_{\hat \varepsilon}|\leq C_{\hat R},~\text{for }|x|<\hat R,$$
we have from Lemma \ref{sect2-lemma21-gradientu} that
\begin{align*}
|\nabla u_{1,\hat \varepsilon}|\leq C_{\hat R},\text{ for }|x|<\hat R.
\end{align*}
Thus, we find by the standard diagonal procedure that  
\begin{align*}
u_{1,\varepsilon}\rightarrow u_1\text{ in }C^{1,\alpha}_{\text{loc}}(\mathbb R^N)\text{ with }\alpha\in(0,1),
\end{align*}
then take the limit in (\ref{takelimitbefore20240723}) to obtain $u_1$ satisfies
\begin{align*}
\left\{\begin{array}{ll}
-\Delta u_1+C_H|\nabla u_1|^{\gamma}-\frac{\gamma'}{N}=a^* m_1^{\frac{\gamma'}{N}},\\
-\Delta m_1=C_H\gamma\nabla\cdot(m_1|\nabla u_1|^{\gamma-2}\nabla u_1),\\
0<\int_{\mathbb R^N}m_1\,dx\leq 1.
\end{array}
\right.
\end{align*}
Moreover, we deduce from  Lemma \ref{poholemma}  and \eqref{GNinequalityused} that
$\int_{\mathbb R^N}m_1\,dx=1.$
It then follows  from \eqref{c23primenotebetsnegative} that 
\begin{align*}
m_{1,\hat \varepsilon}\rightarrow m_1\text{ in }L^p(\mathbb R^N),~\forall p\geq 1,
\end{align*}
which finishes the proof of this theorem.
 \end{proof}

Next, we focus on the refined blow-up rate of minimizers under the case $\beta\leq 0$ and proceed to complete the proof of Theorem \ref{thm17multipopulation}.  Before proving Theorem \ref{thm17multipopulation}, we collect the results of the existence of minimizers to (\ref{problem51innotescopy}) and the corresponding asymptotic profiles as follows
\begin{lemma}\label{lemma42notesbetacopy}
Problem (\ref{problem51innotescopy}) admits a  minimizer $(m_i,w_i)\in W^{1,p}(\mathbb R^N)\times L^p(\mathbb R^N)$ with $p>1$, 
where $w_i=-C_H\gamma m_i|\nabla u_i|^{\gamma-2}\nabla u_i$ with $u_i\in C^2(\mathbb R^N)$ for 
$i=1,2$ .  Moreover, the following conclusions hold: 
\begin{itemize}
    \item[(i).] $\epsilon_i:=\bigg(C_L\int_{\mathbb R^N}\big|\frac{w_i}{m_i}\big|^{\gamma'}m_i\,dx\bigg)^{-\frac{1}{\gamma'}}\rightarrow 0$\text{ as }$\alpha_i\nearrow a^*$;
    \item[(ii).] Let $x_{i,\epsilon_i}$ be a global minimum point of $u_i$, then
    \begin{align}\label{5p6notebetacopy}
    u_{i,\epsilon}:=\epsilon_i^{\frac{2-\gamma}{\gamma-1}}u_i(\epsilon_i x+x_{i,\epsilon_i}),~m_{i,\epsilon}:=\epsilon_i^{N}m_i(\epsilon_i x+x_{i,\epsilon_i}),~w_{i,\epsilon}:=\epsilon_i^{N+1}w(\epsilon_i x+x_{i,\epsilon_i})
    \end{align}
    satisfies up to a subsequence,
    \begin{align}\label{5p7betanotecopy}
    u_{i,\epsilon_i}\rightarrow \bar u_i\text{ in }C^2_{\text{loc}}(\mathbb R^N),~m_{i,\epsilon_i}\rightarrow \bar m_i\text{ in }L^p(\mathbb R^N),\forall p\in[1,+\infty],~w_{i,\epsilon_i}\rightarrow \bar w_i\text{ in }L^{\gamma'}(\mathbb R^N),
    \end{align}
    where $(\bar m_i,\bar w_i)$ is a minimizer of (\ref{GNinequalitybest}) and $(\bar m_i,\bar u_i)$ satisfies (\ref{equmpotentialfree});
    \item[(iii).] if $V_i$ satisfies (\ref{52viinnotesbetacopy}) and set 
    \begin{align}\label{nupidefinedbybetacopy}
\nu_{q_i}:=\inf_{y\in\mathbb R^N}\int_{\mathbb R^N}|x+y|^{q_i}\bar m_i\,dx,
    \end{align}
    then $\nu_{q_i}=\bar\nu_{q_i}$ with $\bar \nu_{q_i}$ given in (\ref{Hmoibarnupi}) and 
\begin{align}\label{eialphainotebetacopy}
    e^i_{\alpha_i}:=(1+o(1))\frac{q_i+\gamma'}{q_i}\bigg(\frac{q_i\bar\nu_{q_i}b_i}{\gamma'}\bigg)^{\frac{\gamma'}{\gamma'+1}}\bigg(\frac{a^*-\alpha_i}{a^*}\bigg)^{\frac{q_i}{\gamma'+q_i}},~   \epsilon_i=(1+o(1))\bigg(\frac{\gamma'(a^*-\alpha_i)}{a^*b_i\bar\nu_{q_i}q_i}\bigg)^{\frac{1}{\gamma'+q_i}}.
    \end{align}
Moreover, we have 
    \begin{align*}
    \frac{x_{i,\epsilon_i}-x_i}{\epsilon_i}\rightarrow y_i,
    \end{align*}
    where $y_i\in\mathbb R^N$ satisfies
    \begin{align*}
    H_{\bar m_i,q_i}(y_i)=\inf_{y\in\mathbb R^N}H_{\bar m_i,q_i}(y)=\bar \nu_{q_i}.
    \end{align*}
    In particular, there exist $ R>0$, $C>0$ and $\kappa_1,$ $\delta_0>0$ small such that 
    \begin{align}\label{algebraicdecaynotebeta}
0<m_{i,\epsilon_i}\leq Ce^{-\frac{\kappa_1}{2}|x|^{\delta_0}}\text{ for }|x|>R,
\end{align}
\end{itemize}
\end{lemma}
\begin{proof}
Proceeding the similar arguments shown in \cite{cirant2024critical}, we are able to show Conclusion (i), (ii) and (iii) with $\bar\nu_{p_i}$ replaced by $\nu_{p_i}$.  (\ref{algebraicdecaynotebeta}) follows directly from Proposition \ref{appenexp} shown in Appendix \ref{appendixA}.  It is left to show $\bar \nu_{p_i}=\nu_{p_i}$ with $\nu_{p_i}$ defined by (\ref{nupidefinedbybetacopy}). 
 First of all, it is straightforward to see that $\bar \nu_{q_i}\leq \nu_{q_i}$.  Then, we argue by contradiction and assume 
\begin{align*}
\bar\nu_{q_i}<\nu_{q_i},~\text{ for }i=1 \text{ or }2.
\end{align*}
In light of the definition of $\bar \nu_{q_i}$ given in (\ref{Hmoibarnupi}), we find that there exists $ (m,w)\in \mathcal M$ with $\mathcal M$ defined by \eqref{mathcalMdefinedbythm17}  and $y_i\in \mathbb R^N$ such that 
\begin{align*}
\nu_{i0}:=\inf_{y\in\mathbb R^N}\int_{\mathbb R^N}|x+y|^{q_i}m(x)\,dx=\int_{\mathbb R^N}|x+y_i|^{q_i}m(x)\,dx<\nu_{q_i}.
\end{align*}
Let
\begin{align*}
m_{\tau}:=\tau^{N}m(\tau(x-x_i)-y_i),~w_{\tau}:=\tau^{N+1}w(\tau(x-x_i)-y_i),
\end{align*}
where $\tau=\bigg(\frac{b_i\nu_{i0} q_ia^*}{2\gamma'\big(a^*-\frac{\alpha_1+\alpha_2+2\beta}{2}\big)}\bigg)^{\frac{1}{\gamma'+q_i}}$.  Then one can obtain 
\begin{align}\label{4109reachcontradiction2024089}
e^i_{\alpha_i}\leq (1+o(1))\frac{q_i+\gamma'}{p_i}\bigg(\frac{q_i\nu_{i0}b_i}{\gamma'}\bigg)^{\frac{\gamma'}{\gamma'+1}}\bigg(\frac{a^*-\alpha_i}{a^*}\bigg)^{\frac{q_i}{\gamma'+q_i}}.
\end{align}
Whereas, (\ref{eialphainotebetacopy}) gives that 
\begin{align*}
e^i_{\alpha_i}=(1+o(1))\frac{q_i+\gamma'}{q_i}\bigg(\frac{q_i\nu_{q_i}b_i}{\gamma'}\bigg)^{\frac{\gamma'}{\gamma'+1}}\bigg(\frac{a^*-\alpha_i}{a^*}\bigg)^{\frac{q_i}{\gamma'+q_i}},
\end{align*}
which reaches a contradiction to (\ref{4109reachcontradiction2024089}).
\end{proof}
Now, we establish the lower and upper bounds of $e_{\alpha_1,\alpha_2,\beta}$ in the following lemma.   

\begin{lemma}\label{lemma4320240809}
    Assume that each $V_i$ satisfies \eqref{52viinnotesbetacopy} with $x_1\not=x_2$ and (\ref{assumethm17notecopybeta}) holds.  Let $(m_{1,\textbf{a}},w_{1,\textbf{a}},m_{2,\textbf{a}},w_{2,\textbf{a}})$ be a minimizer of $e_{\alpha_1,\alpha_2,\beta}$ defined by (\ref{problem1p1}) with $\beta<0.$  Then for any $q>\max\{q_1,q_2\}$, we have there exists $C_q>0$ such that 
    \begin{align}\label{513notebetanegativecopy}
e^1_{\alpha_1}+e^2_{\alpha_2}\leq& e_{\alpha_1,\alpha_2,\beta}=\mathcal E(m_{1,\textbf{a}},w_{1,\textbf{a}},m_{2,\textbf{a}},w_{2,\textbf{a}})\nonumber\\
    \leq & e^1_{\alpha_1}+e^2_{\alpha_2}+C_q\tilde \epsilon_2^q,\text{ as }(\alpha_1,\alpha_2)\nearrow (a^*,a^*),
    \end{align}
    where $e_{\alpha_i}^i$, $i=1,2$ are defined by (\ref{problem51innotescopy}).
    In particular, the following estimates hold:
\begin{align}\label{514notebetanegativecopy}
    e^i_{\alpha_i}\leq \mathcal E^i_{\alpha_i}(m_{i,\textbf{a}}.w_{i,\textbf{a}})\leq e^i_{\alpha_i}+C_q\tilde \epsilon^q_{2},~i=1,2.
    \end{align}
\end{lemma}
\begin{proof}
Noting that $\beta<0$, we deduce from (\ref{problem51innotescopy}) 
 and (\ref{mathcalealphaii089}) that 
\begin{align}\label{515notebetacopynegative}
\mathcal E(m_{1,\textbf{a}},w_{1,\textbf{a}},m_{2,\textbf{a}},w_{2,\textbf{a}})\geq \sum_{i=1}^2\mathcal E_{\alpha_i}^i(m_{i,\textbf{a}},w_{i,\textbf{a}})\geq e^1_{\alpha_1}+e^2_{\alpha_2}.
\end{align}
Moreover, let $(m_i,w_i)$ be the minimizers of $e^i_{\alpha_i}$, $i=1,2$ obtained in Lemma \ref{lemma42notesbetacopy}, then one has
\begin{align}\label{516notecopybetanegative}
e_{\alpha_1,\alpha_2,\beta}\leq\mathcal E(m_1,w_1,m_2,w_2)=&\sum_{i=1}^2\mathcal E_{\alpha_i}^i(m_i,w_i)-\frac{2\beta}{1+\frac{\gamma'}{N}}\int_{\mathbb R^N}m_1^{\frac{1}{2}+\frac{\gamma'}{2N}}m_2^{\frac{1}{2}+\frac{\gamma'}{2N}}\,dx\nonumber\\
=&e^1_{\alpha_1}+e^2_{\alpha_2}-\frac{2\beta}{1+\frac{\gamma'}{N}}\int_{\mathbb R^N}m_1^{\frac{1}{2}+\frac{\gamma'}{2N}}m_2^{\frac{1}{2}+\frac{\gamma'}{2N}}\,dx.
\end{align}
By using  (\ref{5p6notebetacopy}), one finds
\begin{align}\label{517notebetacopy}
\int_{\mathbb R^N}m_1^{\frac{1}{2}+\frac{\gamma'}{2N}}m_2^{\frac{1}{2}+\frac{\gamma'}{2N}}\,dx=(\epsilon_1\epsilon_2)^{-N(\frac{1}{2}+\frac{\gamma'}{2N})}\epsilon_1^N\int_{\mathbb R^N}m^{\frac{1}{2}+\frac{\gamma'}{2N}}_{1,\epsilon}(x)m^{\frac{1}{2}+\frac{\gamma'}{2N}}_{2,\epsilon}\bigg(\frac{\epsilon_1}{\epsilon_2}x+\frac{x_{1,\epsilon_1}-x_{2,\epsilon_2}}{\epsilon_2}\bigg)\,dx.
\end{align}
Since $x_{i,\epsilon_i}\rightarrow x_i$, $i=1,2$ and $x_1\not=x_2$, we take $R:=\frac{1}{4}|x_1-x_2|$ and obtain for any $x\in B_{R/\epsilon_1}(0)$,
\begin{align*}
\bigg|\frac{x_{1,\epsilon_1}-x_{2,\epsilon_2}}{\epsilon_2}+\frac{\epsilon_1}{\epsilon_2}x\bigg|\geq \frac{|x_{1,\epsilon_1}-x_{2,\epsilon_2}|}{\epsilon_2}-\frac{\epsilon_1}{\epsilon_2}|x|\geq &\frac{3}{4}\frac{|x_1-x_2|}{\epsilon_2}-\frac{R}{\epsilon_2}\nonumber\\
=&\frac{1}{2}\frac{|x_1-x_2|}{\epsilon_2}=O\bigg(\frac{1}{\epsilon_2}\bigg)\rightarrow +\infty.
\end{align*}
It then follows from (\ref{algebraicdecaynotebeta}) that there exists constant $C>0$ such that 
\begin{align*}
m_{2,\epsilon}\bigg(\frac{\epsilon_1}{\epsilon_2}x+\frac{x_{1,\epsilon_1}-x_{2,\epsilon_2}}{\epsilon_2}\bigg)\leq C\bigg|\frac{x_1-x_2}{\epsilon_2}\bigg|^{-\hat q },~\forall x\in B_{R/\epsilon_1}(0),
\end{align*}
which implies 
\begin{align}\label{518betanotecopy}
\int_{|x|<\frac{R}{\epsilon_1}}m^{\frac{1}{2}+\frac{\gamma'}{2N}}_{1,\epsilon}(x)m_{2,\epsilon}^{\frac{1}{2}+\frac{\gamma'}{2N}}\bigg(\frac{\epsilon_1}{\epsilon_2}x+\frac{x_{1,\epsilon_1}-x_{2,\epsilon_2}}{\epsilon_2}\bigg)\,dx\leq C\epsilon_2^{\hat q \big(\frac{1}{2}+\frac{\gamma'}{2N}\big)}\int_{\mathbb R^N}m_{1,\epsilon}^{\frac{1}{2}+\frac{\gamma'}{2N}}\,dx\leq C\epsilon_2^{\hat q\big(\frac{1}{2}+\frac{\gamma'}{2N}\big)},
\end{align}
where we have used  (\ref{5p7betanotecopy}).  In addition, invoking (\ref{algebraicdecaynotebeta}), we obtain $m_{1,\epsilon}$ satisfies for any $\hat q>0,$
\begin{align}\label{by519notebetacopy}
m_{1,\epsilon}(x)\leq C_{\hat q}|x|^{-\hat q }\text{ for }|x|>{R},
\end{align}
where $C_{\hat q}>0$ is some constant. 
 On the other hand, (\ref{5p7betanotecopy}) indicates that 
\begin{align}\label{4120combining20809}
\limsup_{\epsilon_2\rightarrow 0^+}\Vert m_{2,\epsilon}\Vert_{L^\infty}<+\infty.
\end{align}
Combining (\ref{by519notebetacopy}) with (\ref{4120combining20809}), we deduce that 
\begin{align}\label{520notebetacopy}
\int_{|x|>\frac{R}{\epsilon_1}}m^{\frac{1}{2}+\frac{\gamma'}{2N}}_{1,\epsilon}m^{\frac{1}{2}+\frac{\gamma'}{2N}}_{2,\epsilon}\bigg(\frac{\epsilon_1}{\epsilon_2}x+\frac{x_{1,\epsilon_1}-x_{2,\epsilon_2}}{\epsilon_2}\bigg)\,dx\leq C\int_{|x|>\frac{R}{\epsilon_1}}m_{1,\epsilon_1}^{\frac{1}{2}+\frac{\gamma'}{2N}}\,dx\nonumber\\
\leq C_{\hat q}\int_{\frac{R}{\epsilon_1}}^{+\infty}r^{-\hat q \big(\frac{1}{2}+\frac{\gamma'}{2N}\big)}r^{N-1}\,dr\leq C_{\hat q }\epsilon_1^{\hat q \big(\frac{1}{2}+\frac{\gamma'}{2N}\big)-N},
\end{align}
 where we have used (\ref{by519notebetacopy}). 
 Upon collecting (\ref{517notebetacopy}), (\ref{518betanotecopy}) and (\ref{520notebetacopy}), we deduce that 
 \begin{align}\label{4122havefrom2024089}
 \int_{\mathbb R^N}m_1^{\frac{1}{2}+\frac{\gamma'}{2N}}m_2^{\frac{1}{2}+\frac{\gamma'}{2N}}\,dx\leq C_{\hat q}(\epsilon_1\epsilon_2)^{-N(\frac{1}{2}+\frac{\gamma'}{2N})}\epsilon_1^N\bigg(\epsilon_2^{\hat q(\frac{1}{2}+\frac{\gamma'}{2N})}+\epsilon_1^{\hat q(\frac{1}{2}+\frac{\gamma'}{2N})-N}\bigg).
 \end{align}
Since (\ref{assumethm17notecopybeta}) and (\ref{eialphainotebetacopy}) imply up to a subsequence,
\begin{align*}
\lim_{\alpha_i\nearrow a^*}\frac{\epsilon_i}{\tilde \epsilon_i}=C_i,~i=1,2, C_i>0\text{ are constants},
\end{align*}
we have from (\ref{4122havefrom2024089}) that  
\begin{align}\label{4124obtainfrom20204089}
 \int_{\mathbb R^N}m_1^{\frac{1}{2}+\frac{\gamma'}{2N}}m_2^{\frac{1}{2}+\frac{\gamma'}{2N}}\,dx\leq  &C_{\hat q}\tilde \epsilon_1^{-N(\frac{\gamma'}{2N}-\frac{1}{2})}\tilde \epsilon_2^{(\hat q-N)\big(\frac{1}{2}+\frac{\gamma'}{2N}\big)}+C_{\hat q}\tilde \epsilon_{1}^{(\hat q-N)\big(\frac{1}{2}+\frac{\gamma'}{2N}\big)}\tilde \epsilon_2^{-N\big(\frac{1}{2}+\frac{\gamma'}{2N}\big)}\nonumber\\
=&C_{\hat q}\tilde \epsilon^{(\hat q-N)\big(\frac{1}{2}+\frac{\gamma'}{2N}\big)-sN\big(\frac{\gamma'}{2N}-\frac{1}{2}\big)}_2+\tilde \epsilon_2^{s(\hat q-N)\big(\frac{1}{2}+\frac{\gamma'}{2N}\big)-N\big(\frac{1}{2}+\frac{\gamma'}{2N}\big)}.
 \end{align}
By choosing $\hat q>0$ large enough, one finds for any $q>\max\{p_1,p_2\}$, 
 $$(\hat q-N)\bigg(\frac{1}{2}+\frac{\gamma'}{2N}\bigg)-sN\bigg(\frac{\gamma'}{2N}-\frac{1}{2}\bigg)>q,~s(\hat q-N)\bigg(\frac{1}{2}+\frac{\gamma'}{2N}\bigg)-N\bigg(\frac{1}{2}+\frac{\gamma'}{2N}\bigg)>q.$$
Thus, we obtain from (\ref{4124obtainfrom20204089}) that  
\begin{align}\label{4125notebetacopynegative}
\int_{\mathbb R^N}m_1^{\frac{1}{2}+\frac{\gamma'}{2N}}m_2^{\frac{1}{2}+\frac{\gamma'}{2N}}\,dx\leq C_q\tilde \epsilon_2^q,~\forall q>\max\{q_1,q_2\},
\end{align}
where $C_{q}>0$ is a constant. 
 Finally, (\ref{4125notebetacopynegative})
 together with (\ref{515notebetacopynegative}) and (\ref{516notecopybetanegative}) implies (\ref{513notebetanegativecopy}).

We next show estimate (\ref{514notebetanegativecopy}).  First of all, it is straightforward to obtain from the definitions of $e^i_{\alpha_i}$ that \begin{align}\label{521notebetanegativecopy}
e^i_{\alpha_i}\leq \mathcal E^i_{\alpha_i}(m_{i,\textbf{a}},w_{i,\textbf{a}}),~i=1,2.
\end{align} 
Then, we argue by contradiction to establish the estimate shown in the right hand side of (\ref{514notebetanegativecopy}).  Without loss of generality, we assume for $i=1,$ 
\begin{align*}
\mathcal E^1_{\alpha_1}(m_{1,\textbf{a}},w_{1,\textbf{a}})\geq e^1_{\alpha_i}+\Gamma\tilde \epsilon_2^{q},
\end{align*}
where $\Gamma>0$ is large enough and note that $\tilde \epsilon_2^q\ll \min\{e^1_{\alpha_1},e^2_{\alpha_2}\}$ thanks to   (\ref{eialphainotebetacopy}).
Whereas, by using  (\ref{521notebetanegativecopy}),
one has 
\begin{align*}
e^1_{\alpha_1}+e^2_{\alpha_2}+\Gamma\tilde \epsilon_2^q\leq \sum_{i=1}^2\mathcal E_{\alpha_i}(m_{i,\textbf{a}},w_{i,\textbf{a}})\leq e_{\alpha_1,\alpha_2,\beta},
\end{align*}
which is contradicted to \eqref{513notebetanegativecopy}.  Therefore, we find (\ref{514notebetanegativecopy}) holds for $i=1$.  Proceeding the similar argument, we can show (\ref{514notebetanegativecopy}) holds for $i=2.$

\end{proof}
\begin{remark}
We remark that by using the exponential decay properties of $m$ shown in Proposition \ref{appenexp}, the conclusion in Lemma \ref{lemma4320240809} holds when (\ref{assumethm17notecopybeta}) is replaced by the following condition 
\begin{align}\label{550condition20240929}
  \lim_{\bf{a}\nearrow \bf{a}^*}  \frac{e^{{-\tilde \epsilon_1}^{-\hat\delta}}}{\tilde \epsilon_2^{q_2}}=0,
    \end{align}
    where $q_2$ is given in \eqref{52viinnotesbetacopy} and constant $\hat \delta>0$ depends on $\delta_0$ and $\kappa_1$, which are defined in Proposition \ref{appenexp}.  Moreover, with the aid of (\ref{550condition20240929}), one can show all conclusions of Theorem \ref{thm17multipopulation}.  In other words, assumption (\ref{assumethm17notecopybeta}) can be relaxed as (\ref{550condition20240929}) if the exponential decay properties of $m_1$ and $m_2$ are established.
    
\end{remark}

Now, we are ready to prove Theorem \ref{thm17multipopulation}, which is 

{\it{Proof of Theorem \ref{thm17multipopulation}:}}
\begin{proof}
First of all, we have the fact that 
\begin{align*}
\mathcal E^i_{\alpha_i}(m_{i,\textbf{a}},w_{i,\textbf{a}})\geq \hat \varepsilon_i^{\gamma'}\bigg(1-\frac{\alpha_i}{a^*}\bigg)+\int_{\mathbb R^N}V_i(\hat \varepsilon_i x+x_{i,\hat\varepsilon})m_{i,\hat\varepsilon}\,dx.
\end{align*}
We compute to get
\begin{align}\label{413120240809}
\int_{\mathbb R^N}V_i(\hat\varepsilon_i x+x_{i,\hat\varepsilon})m_{i,\hat\varepsilon}\,dx&=\hat\varepsilon_i^{q_i}\int_{\mathbb R^N}\frac{V_i(\hat\varepsilon_i x+x_{i,\hat\varepsilon})}{|\hat\varepsilon_i x+x_{i,\hat\varepsilon}-x_i|^{p_i}}\bigg|x+\frac{x_{i,\hat\varepsilon}-x_i}{\hat\varepsilon_i}\bigg|m_{i,\hat\varepsilon}\,dx=\hat\varepsilon^{q_i}_iI_{\hat \varepsilon}.
\end{align}
By using (\ref{problem51innotescopy}), (\ref{514notebetanegativecopy}) and (\ref{eialphainotebetacopy}), we proceed the similar argument shown in Theorem \ref{thm16refinedblowup}, then obtain up to a subsequence,
\begin{align*}
\frac{x_{i,\hat\varepsilon}-x_i}{\hat\varepsilon_i}\rightarrow y_{i0}\text{ for some }y_{i0}\in\mathbb R^N.
\end{align*}
Hence, one has $I_{\hat\varepsilon}$ defined in (\ref{413120240809}) satisfies
\begin{align*}
\lim_{\hat\varepsilon\rightarrow 0}I_{\hat \varepsilon}\geq b_i\int_{\mathbb R^N}|x+y_{i0}|^{q_i}m_i(x)\,dx\geq \bar \nu_{q_i}b_i,
\end{align*}
where the ``=" in the second inequality holds if and only if $H_{\bar m_i,q_i}(y_{i0})=\inf_{y\in\mathbb R^N}H_{\bar m_i,q_i}(y)=\bar \nu_{q_i}.$
It then follows that 
\begin{align*}
\mathcal E^i_{\alpha_i}(m_{i,\textbf{a}},w_{i,\textbf{a}})\geq& \hat\varepsilon_i^{\gamma'}\frac{a^*-\alpha_i}{a^*}
+\hat\varepsilon_i^{q_i}b_i\bar\nu_{q_i}(1+o(1))\nonumber\\
\geq &(1+o(1))\frac{q_i+\gamma'}{p_i}\bigg(\frac{q_i\bar\nu_{q_i}b_i}{\gamma'}\bigg)^{\frac{\gamma'}{\gamma'+1}}\bigg(\frac{a^*-\alpha_i}{a^*}\bigg)^{\frac{q_i}{\gamma'+q_i}},
\end{align*}
and the equality holds if and only if 
\begin{align*}
\hat\varepsilon_i^{\gamma'}=(1+o(1))\bigg(\frac{\gamma'(a^*-\alpha_i)}{a^*b_i\bar\nu_{q_i}q_i}\bigg)^{\frac{1}{\gamma'+q_i}}.
\end{align*}
Comparing the lower bound and the upper bound of $\mathcal E_{\alpha_i}^i(m_{i,\textbf{a}},w_{i,\textbf{a}})$ with $i=1,2$ shown in Lemma \ref{lemma4320240809}, we see  that \eqref{eq-1.40} and \eqref{eq-1.40} hold. The proof of this theorem is finished.
\end{proof}
Theorem \ref{thm17multipopulation} exhibits the refined blow-up profiles of ground states when interaction coefficient $\beta<0$ under some technical assumptions \eqref{52viinnotesbetacopy} and \eqref{assumethm17notecopybeta}.  It is worthy mentioning that with the aid of Proposition \ref{appenexp}, we are able to improve the condition (\ref{assumethm17notecopybeta}) such that the conclusion shown in Theorem \ref{thm17multipopulation} still holds.

\section{Conclusions}
In this paper, we have studied the stationary multi-population Mean-field Games system (\ref{ss1}) with decreasing cost self-couplings and interactive couplings under critical mass exponents via variational methods.  Concerning the existence of ground states, we classified the existence of minimizers to constraint minimization problem (\ref{problem1p1}) in terms of self-focusing coefficients and interaction coefficients, in which the attractive and repulsive interactions were discussed, respectively.  In particular, when all coefficients are subcritical, we showed the existence of ground states to (\ref{ss1}) by the duality argument.  Then, the basic and refined blow-up profiles of ground states were studied under some mild assumptions of potential functions $V_i,$ $i=1,2.$  

We would like to mention that there are also some open problems deserve explorations in the future.  In this paper, we focus on the existence and asymptotic profiles of ground states to (\ref{ss1}) with mass critical local couplings under the case of $\gamma<N'$ with $\gamma$ given in (\ref{MFG-H}) since population density $m$ can be shown in some H\"{o}lder space by using Morrey's estimate and system (\ref{ss1}) enjoys the better regularity.  Whereas, if $\gamma\geq N'$, nonlinear terms \eqref{alpha12betadef} in (\ref{ss1}) become singular and one can only show $m\in L^p(\mathbb R^N)$ for some $p>1$ by standard Sobolev embedding.  Correspondingly, the positivities of $m_{1}$ and $m_2$ given in (\ref{ss1}) can not be shown due to the worse regularities.  Hence, when $\gamma\geq N'$, it seems a challenge but interesting to prove the existence of ground states even under the mass subcritical local couplings.  On the other hand, while discussing the concentration phenomena in (\ref{ss1}), we impose some assumptions on potential functions $V_i$, $i=1,2.$  In detail, when the interaction coefficient $\beta$ satisfies $\beta>0$, (\ref{moreassumptionforthm12}) is assumed for the convenience of analysis.  However, when $V_1$ and $V_2$ satisfy $\inf_{x\in\mathbb R^N}(V_1(x)+V_2(x))>0$, the classification of the existence of minimizers is more intriguing and the corresponding blow-up profiles analysis might be more complicated.  Similarly, if the interaction is repulsive, the investigation of the concentration property of global minimizers is also challenging when $V_1$ and $V_2$ have common global minima.

\section*{Acknowledgments}
We thank Professor M. Cirant for stimulating discussions and many insightful suggestions.  Xiaoyu
Zeng is supported by NSFC (Grant Nos. 12322106, 11931012,  12171379, 12271417).






\begin{appendices}
\setcounter{equation}{0}
\renewcommand\theequation{A.\arabic{equation}}

\section{Exponential Decay Estimates of Population Densities}\label{appendixA}
In this appendix, we investigate the exponential decay property of population density $m$.  More precisely, we consider the following system:
\begin{align*}
\left\{\begin{array}{ll}
-\Delta u_{\varepsilon}+C_H|\nabla u_{\varepsilon}|^{\gamma}+\lambda_{\varepsilon}=\varepsilon^{\gamma}V(\varepsilon x+x_{\varepsilon})+g_{\varepsilon} (x),&x\in\mathbb R^N,\\
-\Delta m_{\varepsilon}+C_H\gamma\nabla\cdot (m_{\varepsilon}|\nabla u_{\varepsilon}|^{\gamma-2}\nabla u_{\varepsilon})=0,&x\in\mathbb R^N,
\end{array}
\right.
\end{align*}
where $\gamma>1$, $V$ and $g_{\varepsilon}$ are given.  Under some assumptions of $g_{\varepsilon}$ and $\lambda_{\varepsilon}$, one can show $m_{\varepsilon}$ satisfies the exponential decay property, which is
\begin{proposition}\label{appenexp}
Denote $(m_{\varepsilon},u_{\varepsilon},\lambda_{\varepsilon})\in W^{1,p}(\mathbb R^N)\times C^2(\mathbb R^N)\times \mathbb R$ as the solution to 
\begin{align*}
\left\{\begin{array}{ll}
-\Delta u_{\varepsilon}+C_H|\nabla u_{\varepsilon}|^{\gamma}+\lambda_{\varepsilon}=\varepsilon^{\gamma}V(\varepsilon x+x_{\varepsilon})+g_{\varepsilon} (x),&x\in\mathbb R^N,\\
-\Delta m_{\varepsilon}+C_H\gamma\nabla\cdot (m_{\varepsilon}|\nabla u_{\varepsilon}|^{\gamma-2}\nabla u_{\varepsilon})=0,&x\in\mathbb R^N,
\end{array}
\right.
\end{align*}
where $m_{\varepsilon}\geq0$ in $\mathbb R^N$, $u_{\varepsilon}$ is uniformly bounded from below and H\"{o}lder continuous function $V$ satisfies (\ref{Vicondition1}) and (\ref{Vicondition2}), and $g_{\varepsilon}\in C^{0,\theta}(\mathbb R^N)$ with $\theta\in(0,1)$ independent of $\varepsilon$.  Suppose that
\begin{itemize}
    \item[(i).]$\lambda_{\varepsilon}\rightarrow \lambda_0$ up to a subsequence with $\lambda_0<0$;
    \item[(ii).] $g_{\varepsilon}(x)\rightarrow 0$ uniformly as $|x|\rightarrow +\infty$,
\end{itemize}
 then we have there exist constants $C>0$ and $R>0$ independent of $\varepsilon$ such that
\begin{align}\label{A020241008}
0<m_{\varepsilon}\leq Ce^{-\frac{\kappa_1}{2}|x|^{\delta_0}}\text{ when }|x|>R,
\end{align}
where constant $\delta_0\in(0,\min\{\gamma-1,1\})$, constant $\kappa_1>0$ and they are independent of $\varepsilon.$
\end{proposition}
\begin{proof}
By following the same argument shown in the proof of Theorem \ref{thm15blowupnegative}, one has 
\begin{align}\label{appendixA4lowerbound}
u_{\varepsilon}(x)\geq C|x|\text{ for }|x|>\hat R,
\end{align}
where $\hat R>0$ is some constant.  Then we define the  Lyapunov function $\Phi=e^{\kappa u_{\varepsilon}^{\delta_0}}$ with $0<\kappa<1$ and $0<\delta_0<1$ will be determined later.  We compute to get
\begin{align*}
&-\Delta\Phi+C_H\gamma|\nabla u_{\varepsilon}|^{\gamma-2}\nabla u_{\varepsilon}\cdot \nabla\Phi\nonumber\\
=&\kappa\delta_0\Phi u_{\varepsilon}^{\delta_0-1}[-\Delta u_{\varepsilon}-(\kappa\delta_0 u_{\varepsilon}^{\delta_0-1}+(\delta_0-1)u_{\varepsilon}^{-1})|\nabla u_{\varepsilon}|^2+C_H\gamma|\nabla u_{\varepsilon}|^{\gamma}]\nonumber\\
=&\kappa\delta_0\Phi u_{\varepsilon}^{\delta_0-1}[C_H(\gamma-1)|\nabla u_{\varepsilon}|^{\gamma}-\lambda_{\varepsilon}+\varepsilon^{\gamma}V(\varepsilon x+x_{\varepsilon})+g_{\varepsilon}(x)-(\kappa\delta_0 u_{\varepsilon}^{\delta_0-1}+(\delta_0-1)u_{\varepsilon}^{-1})|\nabla u_{\varepsilon}|^2].
\end{align*}
Without loss of generality, we assume $u_{\varepsilon}\geq 1$ by fixing $u_{\varepsilon}(0)$.  Then it is straightforward to show that 
\begin{align*}
(\kappa\delta_0 u_{\varepsilon}^{\delta_0-1}+(\delta_0-1)u_{\varepsilon}^{-1})|\nabla u_{\varepsilon}|^2\leq 2\kappa\delta_0 u_{\varepsilon}^{\delta_0-1}|\nabla u_{\varepsilon}|^2,~|x|>R,
\end{align*}
where $R>0$ is a large constant and we have used $u^{\alpha-1}\geq u^{-1}$.  In addition, by using Lemma \ref{sect2-lemma21-gradientu} and Lemma \ref{lowerboundVkgenerallemma22}, we have the facts that 
\begin{align}\label{eq-A03}
|\nabla u_\varepsilon|^{2-\gamma}\leq C(1+\varepsilon^{\gamma}V)^{\frac{2-\gamma}{\gamma}},\text{ and }u_\varepsilon^{1-\delta_0}\geq C(1+\varepsilon^{\gamma}V)^{\frac{1-\delta_0}{\gamma}}\text{ for }|x|>R,
\end{align}
where $C>0$ is a constant and $0<\delta_0<1$.

Next, we would like to prove there exists $R>0$ independent of $\varepsilon$ such that 
\begin{align}\label{A820240811appen}
\frac{C_H(\gamma-1)}{2}|\nabla u_{\varepsilon}|^{\gamma}\geq 2\kappa\delta_0 u_{\varepsilon}^{\delta_0-1}|\nabla u_{\varepsilon}|^2,~\forall |x|>R.
\end{align}
Actually, when $\gamma\geq 2$, it is easy to show (\ref{A820240811appen}) holds by \eqref{eq-A03} and choosing $\kappa$ small enough.  When $1<\gamma<2$, by taking $\delta_0$ and $\kappa$ such that $2-\gamma\leq 1-\delta_0$ and $\kappa$ small, one finds (\ref{A820240811appen}) holds. 
In summary, upon choosing $\delta_0\in(0,\min\{\gamma-1,1\})$ and $\kappa$ small enough, we apply Condition (i) and (ii) to get 
\begin{align*}
-\Delta \Phi+C_H\gamma|\nabla u_{\varepsilon}|^{\gamma-2}\nabla u_{\varepsilon}\cdot \nabla\Phi\geq C\kappa\delta_0 u_{\varepsilon}^{\delta_0-1}\Phi,\text{ if }|x|>R.
\end{align*}
 Proceeding the similar argument shown in the proof of (\ref{461betanegative20240723}), one finds
\begin{align*}
\sup_{\varepsilon}\int_{\mathbb R^N}e^{\kappa u_{\varepsilon}^{\delta_0}}u_{\varepsilon}^{\delta_0-1}m_{\varepsilon}\,dx<+\infty.
\end{align*}
Therefore, by using the uniformly H\"{o}lder continuity of $m_{\varepsilon}$ and the fact that $u_{\varepsilon}\geq 1$, we obtain for $|x|>R$ with constant $R>0$ independent of $\varepsilon,$
\begin{align}\label{appenhavefrombefore}
0<m_{\varepsilon}(x)\leq Ce^{-\frac{\kappa}{2}u_{\varepsilon}^{\delta_0}},~\delta_0\in(0,\gamma-1),
\end{align}
where $C>0$, $\kappa>0$ is small and $\delta_0\in(0,\min\{\gamma-1,1\})$, which are all independent of $\varepsilon$.  Moreover, in light of (\ref{appendixA4lowerbound}), one has from (\ref{appenhavefrombefore}) that  (\ref{A020241008}) holds. 
\end{proof}
\end{appendices}

\bibliographystyle{plainnat}

\bibliography{ref}

\end{document}